\documentclass[11pt,reqno]{amsart}
\usepackage[applemac]{inputenc}
\usepackage[T1]{fontenc}
\usepackage[notrig]{physics}
\usepackage{graphicx,enumerate,amssymb,bbold,upgreek}
\usepackage[ruled,vlined]{algorithm2e}
\usepackage[left=2.5cm,right=2.5cm,top=2.5cm,bottom=2.5cm,includeheadfoot]{geometry} 
\usepackage[export]{adjustbox}
\usepackage{pifont}
\newcommand{\xmark}{\ding{55}}

\usepackage[colorlinks=true, pdfstartview=FitV, linkcolor=blue, 
            citecolor=blue, urlcolor=blue]{hyperref}

\numberwithin{equation}{section}
\numberwithin{figure}{section}
\theoremstyle{plain}
\newtheorem{theorem}{Theorem}[section]
\newtheorem{lemma}[theorem]{Lemma}
\newtheorem{proposition}[theorem]{Proposition}
\newtheorem{corollary}[theorem]{Corollary}

\theoremstyle{definition}
\newtheorem{definition}[theorem]{Definition}

\newtheorem{remark}[theorem]{Remark}

\newcommand{\bitem}{\begin{itemize}}
\newcommand{\eitem}{\end{itemize}}
\newcommand{\mc}[1]{\mathcal{#1}}

\newcommand{\mr}[1]{\mathrm{#1}}

\newcommand{\N}{\mathbb{N}}
\newcommand{\R}{\mathbb{R}}

\newcommand{\EE}{\mathbb{E}}

\newcommand{\bpm}{\begin{pmatrix}}
\newcommand{\epm}{\end{pmatrix}}
\newcommand{\bvm}{\begin{vmatrix}}
\newcommand{\evm}{\end{vmatrix}}
\newcommand{\bsm}{\left(\begin{smallmatrix}}
\newcommand{\esm}{\end{smallmatrix}\right)}
\newcommand{\T}{\top}

\newcommand{\ol}[1]{\overline{#1}}
\newcommand{\wh}[1]{\widehat{#1}}
\newcommand{\wt}[1]{\widetilde{#1}}
\newcommand{\la}{\langle}
\newcommand{\ra}{\rangle}

\newcommand{\mrm}[1]{\mathrm{#1}}

\newcommand{\veps}{\varepsilon}

\newcommand{\vphi}{\varphi}

\newcommand{\eins}{\mathbb{1}}

\DeclareMathSymbol{\mydiv}{\mathbin}{symbols}{"04}

\DeclareMathOperator{\Diag}{Diag}

\DeclareMathOperator{\dom}{dom}
\DeclareMathOperator{\intr}{int}
\DeclareMathOperator{\rint}{rint}

\DeclareMathOperator{\argmin}{arg min}

\DeclareMathOperator{\ggrad}{grad}

\DeclareMathOperator{\Exp}{Exp}

\DeclareMathOperator{\KL}{KL}

\makeatletter
\def\widebreve{\mathpalette\wide@breve}
\def\wide@breve#1#2{\sbox\z@{$#1#2$}
     \mathop{\vbox{\m@th\ialign{##\crcr
\kern0.08em\brevefill#1{0.8\wd\z@}\crcr\noalign{\nointerlineskip}
                    $\hss#1#2\hss$\crcr}}}\limits}
\def\brevefill#1#2{$\m@th\sbox\tw@{$#1($}
  \hss\resizebox{#2}{\wd\tw@}{\rotatebox[origin=c]{90}{\upshape(}}\hss$}
\makeatletter

\newcommand{\authorfootnotes}{\renewcommand\thefootnote{\@fnsymbol\c@footnote}}

\title[Accelerated Bregman Divergence Optimization: An Information Geometric Point of View]{
Accelerated Bregman Divergence Optimization with SMART: \\ An Information Geometric Point of View}

\date{\today}

\begin{document}

\maketitle
\vspace{-0.5cm}
\centerline{\textit{Dedicated to Professor Yair Censor on occasion of his 80th birthday.}}\bigskip

\begin{center}
    \normalsize
      \authorfootnotes
      Maren Raus\footnote{raus@stud.uni-heidelberg.de}\textsuperscript{1}, Yara Elshiaty\footnote{elshiaty@math.uni-heidelberg.de}\textsuperscript{1},
      Stefania Petra\footnote{stefania.petra@uni-a.de}\textsuperscript{2} \par \bigskip
    
      \textsuperscript{1}Institute of Mathematics, Heidelberg University, INF 205, 69120 Heidelberg, Germany \par
      \textsuperscript{2}Institute for Mathematics, Augsburg University, Universit{\"a}tsstra{\ss}e 14, 86159 Ausgburg, Germany \par 
\end{center}

\begin{abstract}
We investigate the problem of minimizing Kullback-Leibler divergence between a linear model $Ax$ and a positive vector $b$ in different convex domains (positive orthant, $n$-dimensional box, probability simplex). Our focus is on the SMART method that employs efficient multiplicative updates. We explore the exponentiated gradient method, which can be viewed as a Bregman proximal gradient method and as a Riemannian gradient descent on the parameter manifold of a corresponding distribution of the exponential family. This dual interpretation enables us to establish connections and achieve accelerated SMART iterates while smoothly incorporating constraints. The performance of the proposed acceleration schemes is demonstrated by large-scale numerical examples.
\end{abstract}

\tableofcontents

\section{Introduction}

\subsection{Overview, Motivation}
We consider basic  \emph{SMART (Simultaneous Multiplicative Algebraic
Reconstruction Technique)} iterations \cite{Byrne:1993aa} for solving
\begin{equation}\label{eq:def-KL-objective}
\min_{x \in \R_{+}^{n}\cap \dom\phi} f(x),\qquad
f(x) = \KL(Ax,b), \qquad A \in \R_{+}^{m \times n},\quad b \in \R_{++}^{m},
\end{equation}
where the generalized \textit{Kullback-Leibler (KL)} divergence (aka relative entropy and I-divergence \cite{Csiszar:75}) is the Bregman divergence $D_{\phi_{f}}$ induced by the Bregman kernel $\phi_{f}$,
\begin{subequations}
\begin{align}\label{eq:def-our-D}
\KL(y,y') 
:= D_{\phi_f}(y,y') 
&= \la y, \log y - \log y'\ra - \la \eins, y-y' \ra,
\qquad y\in\R_{+}^{m},\quad y'\in\R_{++}^{m}
\\ \label{eq:def-Shannon-ent}
\phi_f(y) &=  \la y, \log y \ra - \la \eins, y\ra.
\end{align}
\end{subequations}
The KL divergence plays a distinguished role among all divergence functions \cite[Section 3.4]{InformationGeometry-2010}. The SMART iteration with parameter $\tau_{k}>0$ reads
\begin{equation}\label{eq:SMART-iteration-orthant}
(x_{k+1})_j = (x_{k})_j \prod_{i\in[m]}\Big(\frac{b_{i}}{(A x_{k})_{i}}\Big)^{\tau_{k} A_{ij}},\qquad j\in[n], \qquad x_{0} \in  \R^{n}_+, \qquad k = 0,1,\dotsc .    
\end{equation}

The objective function $f(x) = \KL(Ax,b)$ is motivated by inverse problems with inconsistent linear systems, where no $x\in\R_{+}^{n}$ satisfies $A x=b$. Taking nonnegativity into account, the KL divergence can be employed instead of the usual least-squares approach with the squared Euclidean norm. Since the KL divergence lacks symmetry, the order of arguments matters. The reverse order $\KL(b, A x)$ is applied, for example, when measurements arise from counting discrete events (e.g., photons, electrons) subject to noise described by a Poisson process \cite{Herman1985a,Bertero_2009}.

Besides nonnegativity $x\in\R_{+}^{n}$, problem \eqref{eq:def-KL-objective} takes also into account convex constraints $x\in\dom\phi$ in a generic way, in terms of another Bregman kernel $\phi$ used for proximal regularization. In this paper, the domain $\dom\phi$ of feasible solutions is turned into a Riemannian manifold using basic information geometry. This enables us to relate the basic multiplicative updates of the SMART iteration to a Riemannian gradient flow and also paves the way for accelerated optimization, as briefly reviewed below.

\subsection{Related Work}
Regarding the SMART iteration \eqref{eq:SMART-iteration}
introduced in \cite{DarrochRatcliff72},
significant contributions have been made by Censor in \cite{Lent1991a,CensorSegman87} and by Byrne in \cite{Byrne:1993aa,Byrne2014}. The authors of \cite{Petra2013a} considered SMART as a \textit{mirror descent (MD)} method and derived a convergence rate without the common $L$-smoothness assumption which does not hold
for \eqref{eq:def-KL-objective}.

MD has a rich history in optimization and has been extensively studied. Initially proposed by \cite{Nemirovski:1983}, its basic convergence property was studied in \cite{Beck:2003aa} from the viewpoint of the Bregman Proximal Gradient (BPG) method. Recently, the connection to information geometry was explored in \cite{Raskutti:2015aa,Kahl:2023aa}. The paper \cite{Raskutti:2015aa} does not cover the methods employed in the present paper, however, including the systematic use of e-geodesics. Moreover, advancements like Accelerated Bregman Proximal Gradient (ABPG) methods \cite{Hanzely:2021vc,Gutman:2022tu} and unified algorithms for both MD and dual averaging (DA) \cite{Juditsky:2022wf} have shown superior performance in specific scenarios. Other innovations, such as new regularization methods unifying additive and multiplicative updates \cite{Ghai:2020va} and efforts to substantiate discrete schemes by continuous-time ODEs and numerical integration \cite{Krichene:2015}, have expanded the understanding of MD methods, even though some of them become computationally more expensive when the objective lacks $L$-smoothness, as in \eqref{eq:def-KL-objective}.

Below, MD is shown to be closely related to Riemannian gradient descent involving a particular retraction. Riemannian optimization includes foundational work by \cite{Absil:2008aa} which emphasizes the use of retractions and vector transport for effcient optimization. Recent advancements along this line include conjugate gradient methods in the Riemannian setting \cite{Oviedo:2022,Sakai2022} and approaches for optimization on manifolds with lower-bounded curvature \cite{Ferreira:2019}. Acceleration within the Riemannian setting is considered in \cite{Jin:2022aa}. Additionally, first-order methods for geodesically convex optimization \cite{Zhang:2016} contribute to the broader landscape of Riemannian optimization techniques, but do not align precisely with the scenarios studied in this paper.

\subsection{Contribution, Organisation}
Our main contribution can be summarized as follows.
\begin{itemize}
\item Adopting the mirror descent point of view, we characterize SMART as Riemannian gradient descent with the exponential map induced by the e-connection of information geometry \cite{Amari:2000} as retraction.
\item Our generic approach is exemplified using three different domains
\begin{equation}\label{eq:three-simple-sets-boundary}
    \dom \phi = \R_{+}^{n},\qquad\quad
    \dom \phi = [0,1]^n,\qquad\quad
     \dom \phi = \Delta_{n},
\end{equation}
where $\Delta_{n}$ denotes the probability simplex \eqref{eq:Delta-n}. The (relative) interior of these domains are
turned into Riemannian manifolds $(\mc{M},g)$, endowed with the Fisher-Rao geometry $g$, and spefically denoted by
\begin{equation}\label{eq:intro-PBS}
    (\mc{P}_{n},g),\qquad
    (\mc{B}_{n},g),\qquad
    (\mc{S}_{n},g).
\end{equation}
The basic SMART iteration \eqref{eq:SMART-iteration-orthant} valid for the nonnegative orthant then takes the general form
\begin{subequations}\label{eq:SMART-iteration}
\begin{align}
x_{k+1} &= \frac{1}{Z(x_{k})} x_{k} \cdot e^{-\tau_{k} \partial f(x_{k})},\qquad x_{0} \in \mc{M} \in \{\mc{P}_n,\mc{B}_n, \mc{S}_{n}\}
\qquad\qquad(\text{SMART}) 
\intertext{where $\cdot$ denotes componentwise vector multiplication, $\partial f$ denotes the Euclidean gradient of the objective function $f$, and}
Z(x_{k}) &= \begin{cases}
1, &\text{if}\; \mc{M}=\mc{P}_{n}, \\
\eins - x_{k} + x_{k}\cdot e^{-\tau\partial f(x_{k})}, &\text{if}\;\mc{M}=\mc{B}_{n}, \\
\la x_{k}, e^{-\tau_{k}\partial f(x_{k})}\ra, &\text{if}\; \mc{M}=\mc{S}_{n}.
\end{cases}
\end{align}
\end{subequations}
In addition, the basic ingredients for first- and second-order geometric optimization, e-geodesics and vector transport, are computed in each case which extends the conference paper \cite{Kahl:2023aa}.
\item We also apply and empirically study the acceleration methods introduced by \cite{Hanzely:2021vc}.
\item The underlying geometry is exploited to enhance numerical techniques like (non-monotone) Riemannian line search strategies and Riemannian Conjugate Gradient (CG).
\end{itemize}

Section \ref{sec:Preliminaries} introduces basic notation and the information geometry of statistical manifolds in order to specify the Riemannian geometry of the specific manifolds \eqref{eq:intro-PBS}. 
Section \ref{sec:ConvexAcceleration} discusses SMART's acceleration from the viewpoint of Bregman Proximal Gradient, whereas Section \ref{sec:Geometry} exploits SMART's Riemannian geometry from Section \ref{sec:Preliminaries}. This enables to design advanced Riemannian updates and to explore various line search strategies as well as Riemannian conjugate gradient. In Section \ref{sec:Experiments}, results of large-scale experiments such as tomographic reconstruction, image deblurring and sparse signal recovery are reported to validate SMART's competitiveness with state-of-the-art Bregman first-order methods. On the other hand, the accelerated 
$\mc{O}(1/k^2)$ rate of the accelerated version \textit{could not} be certified numerically, whereas the Riemannian variants showcase accelerated convergence, but at a notably higher cost per iterate along e-geodesics. Corresponding open problems for future research on optimization problems that lack $L$-smoothness, are outlined in Section \ref{sec:conclusion}.

\section{Preliminaries}
\label{sec:Preliminaries}

\subsection{Basic Notation}\label{sec:notation}
We set $[n]=\{1,2,\dotsc,n\}$ for $n \in \N$ 
and write $\eins = (1,1,\dotsc,1)^\top$ for the constant one-vector whose dimension depends on the context. $\la \cdot,\cdot \ra$ denotes the Euclidean inner product with induced norm $\|\cdot\| = \sqrt{\la \cdot,\cdot \ra}$. We also use the $\ell^{1}$ and $\ell^{\infty}$ norm indicated by subscripts $\|\cdot\|_{1}$ and $\|\cdot\|_{\infty}$. 
For a vector $x$, the diagonal \textit{matrix} with the components of $x$ as entries is denoted by $\Diag(x)$. Matrix norms $\|A\|, \|A\|_{1}, \|A\|_{\infty}$ always are the norms induced by the respective vector norms. The nonnegative orthant in $\R^{n}$ is denoted by $\R^{n}_{+}$ and its interior by $\R^{n}_{++} = \intr(\R^{n}_{+})$. The relative interior of a set $S\subset\R^{n}$ contained in an affine subspace of $\R^{n}$ is denoted by $\rint(S)$.

\textit{Unary functions} like the natural logarithm $\log$ and the exponential function $e^{(\cdot)}$ apply \textit{componentwise} to vectors, e.g.~
\begin{subequations}\label{eq:notation-xz}
\begin{gather}
\log x = (\log x_{1},\dotsc, \log x_{n}),\qquad
 e^{x} = (e^{x_{1}},\dotsc,e^{x_{n}}).
\intertext{Likewise, we write \textit{componentwise multiplication and subdivision} of two vectors as}
z\cdot x = (z_{1} x_{1},\dotsc, z_{n} x_{n}) 
\qquad\text{and}\qquad 
\frac{z}{x} = \Big(\frac{z_{1}}{x_{1}},\dotsc,\frac{z_{n}}{x_{n}}\Big)\quad\text{when}\quad
x_{i}>0,\; \forall i\in[n].
\end{gather}
\end{subequations}
The \emph{probability simplex} is denoted by
\begin{equation}\label{eq:Delta-n}
\Delta_{n} = \{x \in \R^{n} \colon x \ge 0,\; \la \eins, x \ra=1\}
\end{equation}
and the \emph{unit simplex} by
\begin{equation}\label{eq:Delta-unit}
\Delta_{n-1}^0=\{x\in\R^{n-1}\colon x\geq 0,\;\eins^\top x\le 1\}.
\end{equation}

We denote the conjugate function  of a function $h:\R^n \to\R\cup \{+\infty\}$ with
\begin{equation}
h^{\ast}(p) = \sup_{x \in \R^{n}}\big\{\la p,x \ra - h(x)\big\}
\end{equation}
and the gradient of a continuously differentiable function $h$ is denoted by
\begin{equation}
\partial h(x) = \Big(\frac{\partial}{\partial x_{1}} h(x),\frac{\partial}{\partial x_{2}} h(x),\dotsc\Big)^{\T}.
\end{equation}
 The Hessian of the smooth function $h$ is denoted by
\begin{equation}
\partial ^2 h(x) \in \R^{n \times n},\qquad
\big(\partial ^2 h(x)\big)_{ij} = \frac{\partial^{2}}{\partial x_{i}\partial x_{j}} h(x),\quad i,j\in[n].
\end{equation}

For a smooth Riemannian manifold $(\mc{M},g)$ with metric $g$, $T_x\mc{M}$ denotes the tangent space at $x\in\mc{M}$ and $\mr{d}_x h : T_x\mc{M} \to\R$ the differential of a smooth function $h:\mc{M}\to \R$. For the Riemannian metric $g$, we use the notation
\begin{subequations}
\begin{align}
    g_{x}\colon T_{x}\mc{M}\times T_{x}\mc{M}\to\R,\qquad
    \la u,v\ra_{x} &:= g_{x}(u,v),\qquad x\in\mc{M} 
    \intertext{with the induced Riemannian norm}
    \|v\|_{x} = \sqrt{g_{x}(v,v)}.
\end{align}
\end{subequations}
The Riemannian gradient $\ggrad h(x) \in T_x\mc{M}$ of $h$ is uniquely defined by 
\begin{equation}
\mr{d}_x h(v) = g_x\left (\ggrad h(x), v \right ),\quad\forall v \in T_x \mc{M}. 
\end{equation}

\subsection{Bregman Divergences and Related Definitions}

\begin{definition}{(\textbf{Legendre functions} \cite[Chap. 26]{Rockafellar:1997})} A lower semicontinuous proper
convex function \( \vphi \colon \R^n \rightarrow (-\infty,\infty] \) is called
\begin{enumerate}
  \item[\textnormal{(i)}] \textit{essentially smooth}, if \( \vphi \) is differentiable on \( \intr(\dom \vphi) \neq\emptyset \) and \( \| \vphi(x_{k}) \| \to \infty \) for every sequence \( (x_{k})_{k\in\N} \subseteq \intr(\dom \vphi) \) converging to a boundary point of \( \dom \vphi \) for \( k \to \infty \);
  \item[\textnormal{(ii)}] of \emph{Legendre type}, if \( \vphi \) is essentially smooth and strictly convex on \( \intr(\dom \vphi) \).
\end{enumerate}
\end{definition}

\begin{proposition}[{\textbf{Legendre functions on affine subspaces} \cite[Prop.~5.3]{Alvarez:2004}}]\label{prop:Legendre-affine-subspace}
    Let
    \begin{equation}
        \mc{A} = \{x\in\R^{n}\colon A x = b\},\qquad
        A\in\R^{m\times n},\quad
        b\in\mrm{Im}(A),
    \end{equation}
    and suppose the convex function $\vphi\colon\R^{n}\to\R\cup \{+\infty\}$ is of Legendre type. If $\intr(\dom \vphi) \cap\mc{A}\neq\emptyset$, then the restriction $\vphi|_{\mc{A}}$ of $\vphi$ to $\mc{A}$ is of Legendre type and $\intr_{\mc{A}}\dom \vphi|_{\mc{A}} = \intr(\dom \vphi) \cap\mc{A}$ (where $\intr_{\mc{A}}\dom \vphi|_{\mc{A}}$ is the interior of $\dom \vphi|_{\mc{A}}$ in $\mc{A}$ as a topological subspace of $\R^{n}$).
\end{proposition}

\begin{theorem}\cite[Thm 26.5]{Rockafellar:1997}\label{th:gradient_mapping_one_to_one}
A convex function $\vphi$ is of Legendre type if and only if its conjugate $\vphi^\ast$ is. In this case, the gradient mapping $\partial \vphi\colon \intr(\dom \vphi) \mapsto \intr(\dom \vphi^\ast)$ is one-to-one and
\begin{subequations} \label{eq:one-to-one-gradient}
\begin{align}
\left(\partial \vphi\right)^{-1} 
&= \partial \vphi^\ast,
\\
\vphi^\ast(\partial \vphi(x)) 
&= \la x,\partial \vphi (x)\ra - \vphi(x),\qquad
\forall x\in\intr(\dom\vphi).
\end{align}
\end{subequations}
\end{theorem}

A given Legendre function $\vphi$ induces a corresponding \textit{Bregman divergence},
\begin{equation}\label{def:Bregman}
D_\vphi (x, y) = \vphi(x) - \vphi(y)  -\langle \partial \vphi (y), x - y \rangle,
\end{equation}
as studied, e.g., by \cite{Censor:1981vy,Censor1981,Censor:1997aa}.

\begin{lemma}[\textbf{Bregman projection} {\cite[Thm.~3.12]{Bauschke:1997aa}}]\label{lem:Bregman-projection}
 Suppose $\vphi$ is closed proper convex and differentiable on $\intr(\dom\vphi)$, $C$ is closed convex with $C \cap \intr(\dom\vphi) \neq \emptyset$, and $y_{0} \in \intr(\dom\vphi)$. If $\vphi$ is Legendre, then the Bregman projection $\wt{y}$ of $y_{0}$ defined by 
\begin{equation} \label{eq:B-projection}
 \underset{y \in C \cap \dom\vphi}{\argmin} D_{\vphi}(y,y_{0})
 = \{\widetilde{y}\},\qquad \widetilde{y} \in \intr(\dom\vphi).
\end{equation}
is unique and contained in $\intr(\dom\vphi)$.
\end{lemma}

\begin{lemma}[\textbf{Three-points identity} {\cite[Lem.~3.1]{Chen_Teboulle_1993}}]\label{lem:3-point}
Let $\vphi$  closed proper convex and differentiable on $\intr(\dom\vphi)$.
For any $r \in \dom \vphi$, and $p,q\in \intr(\dom \vphi)$ the following identity holds
\begin{equation}\label{eq:Bregman-Fenchel}
D_{\vphi}(r,q) + D_{\vphi}(q,p) - D_{\vphi}(r,p)
= \la r-q,\partial \vphi (p)-\partial \vphi(q)\ra.
\end{equation}
\end{lemma}
   
\subsection{Fisher-Rao Geometry}\label{sec:FR-Geometry}

We summarize basic concepts of the information geometry of regular probability distributions of the exponential family, given by densities of the form
\begin{equation}\label{eq:p-exponential}
    p(y;\theta)=h(y) e^{\la\theta, t(y)\ra-\psi(\theta)}, \qquad \theta\in\Theta,
\end{equation}
where $h$ is the base measure, $t$ is the sufficient statistic and $\psi$ is the \textit{log-partition function} 
\begin{subequations}\label{eq:psi-Z}
\begin{align}
\psi(\theta)&=\log Z(\theta) 
\intertext{with the normalizing constant}
Z(\theta) &= \int h(y) e^{\la \theta, t(y) \ra} \dd y, 
\end{align}
\end{subequations}
and $\Theta$ parameter space. Basic references include \cite{Brown:1986vy}, \cite{Amari:2000}.

The density $p$ is \textit{regular} if there are no dependencies among the functions forming the vector $y\mapsto (1,t_{1}(y),t_{2}(y),\dotsc)$. As a consequence, the domain $\Theta$ of the \textit{exponential} (a.k.a.~\textit{natural, canonical}) \textit{parameters} $\theta$ is a open convex set with nonempty interior $\mathring{\Theta}\neq\emptyset$. Moreover, the log-partition function $\psi(\theta)$ is lower-semicontinuous, essentially smooth and convex such that the classical Legendre transformation applies (cf.~\cite[Section 26]{Rockafellar:1997}). 

An alternative parametrization of any regular density $p$ of the exponential family is provided by the so-called \textit{mean-parametrization}
\begin{subequations}\label{eq:mean-parametrization}
\begin{align}
    \eta &= \eta(\theta) = \partial\psi(\theta), 
    \label{eq:eta-by-theta} \\
    \mc{M} &:= \partial\psi(\Theta),
\end{align}
\end{subequations}
where `mean' refers to the expected value of the sufficient statistics, which defines $\eta$ by
\begin{equation}
    \eta \overset{\eqref{eq:psi-Z}}{=}
    \frac{1}{Z(\theta)} \int t(y) p(y;\theta)\dd y = \EE_{p}[t(Y)].
\end{equation}
Denoting the conjugate function of $\psi$ by 
\begin{equation}\label{eq:def-phi-psi-ast}
\phi(\eta) =\psi^{\ast}(\eta)
\end{equation}
and assuming $\eta$ and $\theta$ are related by \eqref{eq:eta-by-theta}, one has by the classical Legendre transform
\begin{subequations}
\begin{align}
    \theta = \theta(\eta) &= \partial\phi(\eta), 
    \label{eq:theta-by-eta} \\ \label{eq:Hpsi-Hphi}
    \big(\partial^{2}\psi(\theta)\big)^{-1} &= \partial^{2}\phi(\eta).
\end{align}
\end{subequations}
The Fisher-Rao metric is given by
\begin{equation}\label{eq:metric-Hpsi}
    H(\theta) \in\R^{n\times n},\qquad
    H_{ij}(\theta) = \EE_{p}\bigg[\Big(\frac{\partial}{\partial\theta_{i}}\log p(Y;\theta)\Big) \Big(\frac{\partial}{\partial\theta_{j}}\log p(Y;\theta)\Big)\Big]
    = \big(\partial^{2}\psi(\theta)\big)_{ij} .
\end{equation}
Under the coordinate change \eqref{eq:eta-by-theta}, it covariantly transforms to (cf.~\cite[Section 1.4]{Jost:2017aa})
\begin{equation}
    H_{ij}(\theta) = \sum_{k,l}G_{kl}(\eta)
    \Big(\frac{\partial}{\partial\theta_{i}}\big(\partial\psi(\theta)\big)_{k}\Big)
    \Big(\frac{\partial}{\partial\theta_{j}}\big(\partial\psi(\theta)\big)_{l}\Big)
    = \Big(\big(\partial^{2}\psi(\theta)\big) G(\eta) \big(\partial^{2}\psi(\theta)\big)\Big)_{ij}.
\end{equation}
Hence, taking into account \eqref{eq:Hpsi-Hphi} and \eqref{eq:metric-Hpsi}, one has
\begin{equation}\label{eq:metric-Geta}
    G(\eta) = \partial^{2}\phi(\eta).
\end{equation}
As a result, an exponential family of densities $p$ of the form \eqref{eq:p-exponential} may be identified with the Riemannian manifolds
\begin{equation}
        (\Theta, h)\qquad\text{and}\qquad
        (\mc{M}, g),
\end{equation}
with metrics $h$ and $g$ given by the metric tensors $H(\theta)$ and $G(\eta)$ as defined by \eqref{eq:metric-Hpsi} and \eqref{eq:metric-Geta}, respectively.

\subsection{Specific Manifolds}\label{sec:specific-manifolds}

We will consider the three sets
\begin{equation}\label{eq:three-simple-sets}
    \mc{P}_n = \R_{++}^{n},\qquad\quad
    \mc{B}_n = (0,1)^n,\qquad\quad
    \mc{S}_{n} = \rint(\Delta_{n}) = \Delta_{n}\cap\R_{++}^{n}
\end{equation}
and regard them as instances of a Riemannian manifold $\mc{M}$, respectively, as specified in the subsequent sections.

\subsubsection{Positive Orthant}
The Poisson distribution of a discrete random variable $Y\in\N_{0}$ is given by the probability mass function
\begin{equation}\label{eq:p-Poisson-eta}
    p(y;\eta) = \frac{\eta^{y} e^{-\eta}}{y!},\qquad \eta\in (0,\infty).
\end{equation}
Rewriting this as density of the exponential family gives
\begin{equation}
    p(y;\theta) = \frac{1}{y!} e^{\theta y - \psi(\theta)},\qquad
    \theta = \theta(\eta) = \log\eta,\quad
    \psi(\theta) = e^{\theta}.
\end{equation}
We note that denoting the parameter of \eqref{eq:p-Poisson-eta} with the symbol $\eta$ used for the mean parameter is consistent since $\EE_{p}[Y]=\psi'(\theta) = e^{\theta} = \eta$. Computing the conjugate function \eqref{eq:def-phi-psi-ast}, yields 
\begin{equation}\label{eq:Shannon-entropy-1D}
\phi(\eta) = \eta\log\eta - \eta
\end{equation}
and $\phi''(\eta) = \frac{1}{\eta}$.

In the $n$-dimensional case, one has the density in product form
\begin{equation}
    p(y;\eta) = \prod_{i\in[n]}\frac{\eta_{i}^{y_{i}} e^{-\eta_{i}}}{y_{i}!},\qquad
    \eta\in\mc{P}_{n}=\R_{++}^{n},
\end{equation}
and it is straightforward to check that one obtains the Riemannian manifold
\begin{equation}\label{eq:G-orthant}
    (\mc{P}_{n},g),\qquad G(\eta) = \Diag\Big(\frac{\eins_{n}}{\eta}\Big)
\end{equation}
with metric $g$ and corresponding metric tensor $G(\eta)$ given by \eqref{eq:metric-Geta}, yielding the diagonal matrix \eqref{eq:G-orthant}.

\subsubsection{Unit Box}\label{sec:Bernoulli-manifold} 
The Bernoulli distribution of  random variable $Y\in\{0,1\}$ reads
\begin{equation}
    p(y;\eta) = \eta^{y}(1-\eta)^{1-y},\qquad \eta\in (0,1).
\end{equation}
Rewriting this in exponential form gives
\begin{equation}\label{eq:p-Bernoulli-exp}
    p(y;\theta) = e^{\theta y-\psi(\theta)},\qquad
    \theta=\log\frac{\eta}{1-\eta},\quad
    \psi(\theta)=\log(1+e^{\theta})
\end{equation}
with conjugation function 
\begin{equation}\label{eq:Fermi-Dirac-entropy-1D}
\phi(\eta) = \psi^{\ast}(\eta) 
    = \eta\log\eta + (1-\eta)\log(1-\eta)
\end{equation}
 and $\phi''(\eta) = \frac{1}{\eta (1-\eta)}$.
Thus, in the $n$-dimensional case $\mc{B}_{n}=(0,1)^{n}$, we obtain by \eqref{eq:metric-Geta} the Riemannian manifold
\begin{equation}\label{eq:G-box}
    (\mc{B}_{n},g),\qquad G(\eta) = \Diag\Big(\frac{\eins_{n}}{\eta\cdot (\eins_{n}-\eta)}\Big),
\end{equation}
with a metric $g$ and the corresponding metric tensor $G(\eta)$ given by \eqref{eq:G-box}, which is a diagonal matrix.

\subsubsection{Probability Simplex}\label{sec:Probability-Simplex}
Every point $p\in\mc{S}_{n}$ is a discrete (a.k.a.~categorical) distribution governing a random variable $Y\in[n]$ by $p_{i} = \Pr(Y=i),\; i\in[n]$. Introducing local coordinates
\begin{equation}\label{eq:def-eta-simplex}
    \vartheta\colon\mc{S}_{n}\to\R_{++}^{n-1},\qquad
    \eta := \vartheta(p) = (p_{1},\dotsc,p_{n-1})^{\T},\qquad
    \eta_{i}=\EE_{p}[Y=i],\quad i\in[n-1]
\end{equation}
and
\begin{equation}\label{eq:open-unit-simplex}
\vartheta(\mc{S}_n) = \{\eta\in\R_{++}^{n-1}\colon\la\eins_{n-1},\eta\ra<1\}
\end{equation}
one has
\begin{equation}\label{eq:inv-varphi}
    p = \vartheta^{-1}(\eta) = \bpm \eta \\ 1-\la\eins_{n-1},\eta\ra \epm.
\end{equation}
Note that 
\begin{equation}\label{eq:Delta0-by-Sn}
\vartheta(\mc{S}_n)=\intr( \Delta^0_{n-1}).
\end{equation}

Defining the mapping 
\begin{equation}
\delta\colon[n]\to\{0,1\}^{n-1},\qquad
\delta_{i}(y):=\begin{cases}
1,&\text{if}\; y=i, \\
0,&\text{otherwise},
\end{cases}\qquad i\in[n-1],
\end{equation}
a minimal parameterization of the categorical distribution of a random variable 
$Y\in[n]$ reads
\begin{equation}
    p(y;\eta) = \prod_{i=1}^{n-1} p_{i}^{\delta_{i}(y)} p_{n}^{1-\la\eins_{n-1},\delta(y)\ra}.
\end{equation}
Rewriting this as distribution of the exponential family \cite{Brown:1986vy} gives
\begin{equation}\label{eq:def-psi-simplex}
    p(y;\theta) = e^{\la\theta,\delta(y)\ra - \psi(\theta)},\qquad
    \psi(\theta) = \log(1 + \la\eins_{n-1},e^{\theta}\ra)
\end{equation}
with the natural parameters
\begin{equation}\label{eq:def-theta}
    \theta_{i} := \log\frac{\eta_{i}}{1-\la\eins_{n-1},\eta\ra},\qquad i\in[n-1].
\end{equation}
Computing the conjugate function \eqref{eq:def-phi-psi-ast} we obtain
\begin{equation}
\phi(\eta) = \psi^\ast(\eta)= \la\eta,\log \eta\ra
      + (1-\la\eins_{n-1},\eta\ra)\log (1-\la\eins_{n-1},\eta\ra).
\end{equation}

Using \eqref{eq:metric-Geta} yields the metric tensor
\begin{equation}\label{eq:G-eta-simplex}
    G(\eta) = \Diag\Big(\frac{\eins_{n-1}}{\eta}\Big) + \frac{1}{1-\la\eins_{n-1},\eta\ra} \eins_{n-1}\eins_{n-1}^{\T}
\end{equation}
and in turn the Riemannian manifold (cf.~\eqref{eq:Delta0-by-Sn})
\begin{equation}\label{eq:Sn-g-local}
    \big(\vartheta(\mc{S}_{n}), g\big)
\end{equation}
with the Fisher-Rao metric $g$ given in coordinates by \eqref{eq:G-eta-simplex}.

It will be convenient to express \eqref{eq:Sn-g-local} in terms of the ambient coordinates $p = \vartheta^{-1}(\eta)$, see \eqref{eq:inv-varphi}. 

\begin{lemma}[\textbf{probability simplex, Fisher-Rao metric}]\label{lem:FR-metric}
    The manifold \eqref{eq:Sn-g-local} is isomorphic to
    \begin{equation}
        (\mc{S}_{n},g)
    \end{equation}
    where the Fisher-Rao metric is given in ambient coordinates by
    \begin{equation}
        G(p) = \Diag\Big(\frac{\eins_{n}}{p}\Big),\qquad p\in\mc{S}_{n}.
    \end{equation}
\end{lemma}

\begin{proof}
The constraint $\la\eins_{n},p\ra=1$ yields for any smooth curve $t\mapsto p(t)\in\mc{S}_{n}$ the condition $\la\eins_{n},\dot p\ra=0$ for tangent vectors, that is the tangent space
\begin{equation}\label{eq:def-T0}
    T_{0}:=\{v\in\R^{n}\colon\la\eins_{n},v\ra=0\}
\end{equation}
and the trivial tangent bundle $T\mc{S}_{n} = \mc{S}_{n}\times T_{0}$.

A basis of $T_{0}$ is given by $d\vartheta^{-1}(e_{i}) = e_{i}-e_{n},\; i\in[n-1]$, where $e_{1},\dotsc,e_{n}$ denote the canonical basis vectors of the ambient space $\R^{n}$. Collecting this basis of $T_{0}$ as column vectors of a matrix yields the coordinate representation of tangent vectors
\begin{equation}\label{eq:Jacobian-varphi-inv}
    v = B_{0} v_{\eta},\qquad
    B_{0} := (e_{1}-e_{n},\dotsc,e_{n-1}-e_{n}) = \bpm I_{n-1} \\ -\eins_{n-1}^{\T} \epm \in \R^{n\times (n-1)}
\end{equation}
and the equation
\begin{subequations}
\begin{align}
    g_{\eta}(u_{\eta},v_{\eta})
    = \la u_{\eta},G(\eta) v_{\eta} \ra
    = \la u_{\eta}, B_{0}^{\T} G(p) B_{0} v_{\eta} \ra
    &= \la u, G(p) v\ra = g_{p}(u,v),\qquad u,v\in T_{0}
\intertext{with}\label{eq:G-simplex}
    G(p) &= \Diag\Big(\frac{\eins_{n}}{p}\Big),\qquad p\in\mc{S}_{n}.
\end{align}
\end{subequations}
\end{proof}

We conclude this section by deriving the expression for Riemannian gradients \cite[p.~89]{Jost:2017aa} of functions specified in \textit{ambient} coordinates.

\begin{lemma}[\textbf{Riemannian gradient in ambient coordinates}]\label{lem:Rgrad-simplex}
    Let $f\colon\mc{S}_{n}\to\R$ be a smooth function specified in ambient coordinates $f(p)$ with gradient $\partial f(p) = (\partial_{p_{1}} f(p),\dotsc,\partial_{p_{n}} f(p))$. Then the Riemannian gradient of $f$ with respect to the Fisher-Rao metric reads
    \begin{subequations}\label{eq:Riemannian-grad-simplex}
    \begin{align}\label{eq:grad-f-simplex}
        \ggrad f(p) &= \Pi_{p}\partial f(p),\quad p\in\mc{S}_{n}
        \intertext{with}\label{eq:replicator-map}
        \Pi_{p} &:= \Diag(p)-p p^{\T}.
    \end{align}
    \end{subequations}
\end{lemma}

\begin{proof}
    Writing $f(p)=f(\vartheta^{-1}(\eta))$ using \eqref{eq:inv-varphi}, the chain rule gives
    \begin{equation}
        \partial_{\eta} f(p) = B_{0}^{\T}\partial f(p),\qquad p = \vartheta^{-1}(\eta).
    \end{equation}
    A straightforward computation yields the inverse of the metric tensor \eqref{eq:G-eta-simplex}
    \begin{equation}\label{eq:def-pi-eta}
        G(\eta)^{-1} = \Diag(\eta)-\eta \eta^{\T}:= \Pi_{\eta},
    \end{equation}
    that we denote by $\Pi_{\eta}$ and, in turn, the Riemannian gradient in \textit{local} coodinates
    \begin{equation}
        \ggrad_{\eta} f(p) = \Pi_{\eta} B_{0}^{\T}\partial f(p),\qquad p = \vartheta^{-1}(\eta),
    \end{equation}
    which by definition is a tangent vector in local coordinates. Adopting the coordinate representation \eqref{eq:Jacobian-varphi-inv} finally yields
    \begin{equation}
        \ggrad f(p) = B_{0} \Pi_{\eta} B_{0}^{\T}\partial f(p)
        = \Pi_{p}\partial f(p)
    \end{equation}
    with $\Pi_{p}$ given by \eqref{eq:replicator-map}.
\end{proof}

\subsubsection{Summary}
The following tables summarize the basic expressions derived above. The relation of the last two rows of Table \ref{tab:distributions-Bregman} will be examined in Remark~\ref{rem:phi}.
\begin{table}[ht]
  \begin{center}
    \caption{Elementary Riemannian manifolds $\mc{M}$, the corresponding distributions $p(y)$, metric tensors $G$ and inverse metric tensors $G^{-1}$. 
    For the probability simplex $\mc{S}_{n}$, working with the ambient coordinates is a convenient alternative (last row).}
    \label{tab:distributions-list}
    \begin{tabular}{|l|l|l|l|} 
    \hline
     $\mc{M}$ & $p(y)$ &  $G$ &  $G^{-1}$ \\
      \hline
      $\mc{P}_{n}$ & $\prod_{i\in[n]}\frac{\eta_{i}^{y_{i}} e^{-\eta_{i}}}{y_{i}!}$ & \eqref{eq:G-orthant} & $\Diag(\eta)$\\
      $\mc{B}_{n}$ & $\prod_{i\in[n]}\eta_{i}^{y_{i}}(1-\eta_{i})^{1-y_{i}}$ &  \eqref{eq:G-box} &  $\Diag(\eta\cdot(\eins-\eta))$\\
      $\vartheta(\mc{S}_{n})$  &                    
      $\prod_{i=1}^{n-1} \eta_{i}^{\delta_{i}(y)} (1-\la\eins_{n-1},\eta\ra)^{1-\la\eins_{n-1},\delta(y)\ra}$ &
     \eqref{eq:G-eta-simplex} & \eqref{eq:def-pi-eta}\\
      $\mc{S}_{n}$  &
      $\prod_{i\in[n-1]} p_{i}^{\delta_{i}(y)}
      p_{n}^{1-\la\eins_{n-1},\delta(y)\ra}$ & 
      \eqref{eq:G-simplex} &  \eqref{eq:replicator-map}\\
      \hline
    \end{tabular}
  \end{center}
\end{table}

\begin{table}[ht]
  \begin{center}
    \caption{The conjugate potentials $\psi, \phi$ corresponding to the elementary Riemannian manifolds $\mc{M}$ of Table \ref{tab:distributions-list}. Using the convention $0 \log 0=0$, the domain of each $\phi$ is  extended to the closure $\ol{\dom\phi}$.}
    \label{tab:distributions-Bregman}
    \begin{tabular}{|l|l|l|l|l|} 
    \hline
     $\mc{M}$ &   $\psi$ & $\dom \psi$ &
      $\phi=\psi^{\ast} $ & $\ol{\dom \phi}$ \\
      \hline
      $\mc{P}_{n}$ & $\la\eins_{n},e^{\theta}\ra$ & $\R^n$ & $\la\eta,\log\eta\ra-\la\eins_{n},\eta\ra$
      & $\R^n_{+}$\\
      $\mc{B}_{n}$ & $\sum_{i\in[n]}\log(1+e^{\theta_{i}})$& $\R^n$ & 
    $\la\eta,\log\eta\ra + \la \eins_{n-1}-\eta,\log(\eins_{n-1}-\eta)\ra$ & ${[0,1]}^n$ \\
      $\vartheta(\mc{S}_{n})$ 
      & $\log(1 + \la\eins_{n-1},e^{\theta}\ra)$&       
      $\R^{n-1}$& $\la\eta,\log \eta\ra
      + (1-\la\eins_{n-1},\eta\ra)\log (1-\la\eins_{n-1},\eta\ra) $ & $\Delta^0_{n-1}$ \\
    $\mc{S}_{n}$ &
    $\frac{e^{q}}{\la\eins_{n},e^{q}\ra}$ &
    $T_{0}\;\eqref{eq:def-T0}$ & 
    $\la p,\log p\ra$ &
    $\Delta_{n}$
    \\
      \hline

    \end{tabular}
  \end{center}
\end{table}

\begin{remark}[\textbf{Bregman kernels}]\label{rem:phi}
 Each function $\phi$ listed in Table \ref{tab:distributions-Bregman} is Legendre. This includes in particular the correspondence $\phi=\psi^{\ast}$ in the last row due to \cite[Example 11.12]{RockafellarWets2010}, with the lack of \textit{strict} convexity of the log-exponential function $\psi=\frac{e^{q}}{\la\eins_{n},e^{q}\ra}$ removed by restricting $\psi$ to the tangent space $T_{0}$. Proposition \ref{prop:Legendre-affine-subspace} with the affine subspace $\{x\in\R^{n}\colon \la\eins_{n},x\ra=1\}$ and $\vphi=\phi$, where $\dom\phi=\R_{+}^{n}$, yields the assertion.

 Accordingly, we denote the functions listed in Table \ref{tab:distributions-Bregman} by
 \begin{equation}\label{eq:phi-list}
     \phi = \begin{cases}
         \phi_{+} &\text{if}\; \mc{M}=\mc{P}_{n}, \\
         \phi_{\square} &\text{if}\; \mc{M}=\mc{B}_{n}, \\
         \phi_{\Delta}^{0} &\text{if}\; \mc{M}=\vartheta(\mc{S}_{n}), \\
         \phi_{\Delta} &\text{if}\; \mc{M}=\mc{S}_{n}.
     \end{cases}
 \end{equation}
\end{remark}

We compute directly using \eqref{def:Bregman}
\begin{subequations}\label{eq:KL-list}
\begin{align}
    \label{eq:D-phi-orthant}
    D_{\phi_{+}}(x,y)& =\la x, \log x - \log y\ra - \la \eins, x-y \ra =: \KL(x,y) , 
    \intertext{and }
     \label{eq:D-phi-box}
    D_{\phi_{{\square}}}(x,y)
    &= D_{\phi_{+}}(x,y) + D_{\phi_{+}}(\eins_{n}-x,\eins_{n}-y),
    \\ \label{eq:D-simplex-equal}
    D_{\phi_{\Delta}^{0}}(\eta,\eta')
    &= D_{\phi_{\Delta}}(p,p'),
    \qquad\qquad (\text{Lemma \ref{lem:D-simplex-equal}})
    \\
    D_{\phi_{\Delta}}(p,p')
    &= \Big\la p,\log\frac{p}{p'}\Big\ra.
\end{align}
\end{subequations}

It remains to establish Equation \eqref{eq:D-simplex-equal}.

\begin{lemma}[\textbf{Bregman divergence using ambient coordinates}]\label{lem:D-simplex-equal}
    One has
    \begin{equation}\label{eq:Dphi=KL}
        D_{\phi_{\Delta}^{0}}(\eta,\eta')
        = D_{\phi_{\Delta}}(p,p')
        = \KL(p,p') = \Big\la p,\log\frac{p}{p'}\Big\ra,\qquad
        \eta=\vartheta(p),\;\eta'=\vartheta(p'),\quad
        \forall p, p' \in\mc{S}_{n}.
    \end{equation}
\end{lemma}

\begin{proof}
Setting $\phi=\phi_{\Delta}^{0}$, we directly compute using \eqref{eq:inv-varphi}
\begin{subequations}
\begin{align}
\phi(\eta) &= \la p, \log p\ra|_{p = \vartheta^{-1}(\eta)}
\label{eq:phi-simplex-local} \\
&= \la\eta,\log \eta\ra
+ (1-\la\eins_{n-1},\eta\ra)\log (1-\la\eins_{n-1},\eta\ra),
\\ \label{eq:partial-phi-simplex}
\partial\phi(\eta) 
&= B^{\T} \log p|_{p = \vartheta^{-1}(\eta)}
= \log\eta - \log(1-\la\eins_{n-1},\eta\ra)\eins_{n-1}
\intertext{Hence}
D_{\phi}(\eta,\eta') 
&= \phi(\eta)-\phi(\eta')-\la\partial\phi(\eta'), \eta-\eta'\ra
\\
&=
\la \eta,\log\eta\ra + p_{n}\log p_{n} -(\la \eta',\log\eta'\ra 
+ p_{n}'\log p_{n}') 
\\ &\qquad\qquad
-\la \log\eta' - \log p_{n}' \eins_{n-1}, \eta - \eta'\ra
\\
&= \la \eta,\log \eta\ra + p_{n}\log p_{n} 
- p_{n}'\log p_{n}'- \la \eta,\log \eta'\ra + \log p_{n}'\la \eins_{n-1} ,\eta-\eta'\ra
\\
&= \la p, \log p\ra - p_{n}'\log p_{n}' - \la \eta,\log\eta'\ra + \log p_{n}'\la \eins_{n-1} ,\eta-\eta'\ra + \underbrace{\log p_{n}' - \log p_{n}'}_{=0} \\
& =  \la p, \log p\ra - \la \eta,\log\eta'\ra
- \underbrace{\left(p_{n}' + \la \eins_{n-1} ,\eta'\ra  - 1\right)}_{=0} \log p_{n}' - \underbrace{\left(1- \la \eins_{n-1} ,\eta\ra\right)}_{=p_{n}} \log p_{n}'\\
&= \la p, \log p-\log p'\ra 
= \KL(p,p').
\end{align}
\end{subequations}
\end{proof}

\section{SMART: Convergence and Acceleration}
\label{sec:ConvexAcceleration}

In this section, we adopt the viewpoint of the Bregman Proximal Gradient (BPG) (Section \ref{BPG-section}) to understand the convergence rate of SMART (Section \ref{sec:SMART-convergence}).  Furthermore, we detail various accelerated variants of the BPG iteration (Section \ref{ABPG-section}) which are supposed to accelerate the BPG iteration at minimal cost.

\subsection{Bregman Proximal Gradient (BPG)}\label{BPG-section}

The classical \textit{mirror descent (MD)} method \cite{Nemirovski:1983} for solving the convex optimization problem 
\begin{equation}\label{eq:objective-f-general}
\min_{x \in C_{F}} F(x),\qquad \dom F=:C_{F} \subset \R^{n},
\end{equation}
with $F$ continuously differentiable convex and $C_{F}$ nonempty closed convex, was investigated by Beck and Teboulle 
\cite{Beck:2003aa} in more general form after
replacing the squared Euclidean norm by general distance-like functions as proximal term. Specifically, using a Bregman divergence  \eqref{def:Bregman} as distance function, the update scheme with stepsize $\tau_{k}>0$ reads
\begin{subequations}\label{eq:MDA-subproblem}
\begin{align}
x_{k+1} &= \underset{x \in  C_{\vphi}}{\argmin}\;
\la \partial F(x_{k}), x \ra + \frac{1}{\tau_{k}} D_{\vphi}(x, x_{k}),
\qquad
k = 0,1,\dotsc,
\\
&=:\mrm{Mirr}_{\vphi}\big(\tau_k \partial F(x_k)\big),
\qquad x_{0} \in C_{\vphi}:=C_{F}\cap\intr(\dom \vphi).
\end{align}
\end{subequations}
This may also be seen as a linearization of the classical proximal point iteration \cite{Rockafellar1976} and using a Bregman divergence $D_{\vphi}$ as proximity measure instead of the squared Euclidean distance and, in this sense, as a linearization of the approach of Censor and Zenios \cite{Censor1992}. Rewriting \eqref{eq:MDA-subproblem} in the form (cf.,~e.g., \cite{Beck:2003aa})
\begin{subequations}\label{eq:mirror-step}
\begin{align}
\mrm{Mirr}_{\vphi}\big(\tau_k \partial F(x_{k})\big) 
&= \arg\min_{x\in C_{\vphi}}\; D_{\vphi}(x,x_{k+1})
\\
&=
x_{k+1} 
= \partial \vphi^{\ast}\big(\partial \vphi(x_{k})-\tau_k \partial F(x_{k})\big) \in C_{\vphi} 
\end{align}
\end{subequations}
reveals that the \emph{mirror descent map} $\mrm{Mirr}_{\vphi}$ equals the \emph{interior Bregman gradient step} studied, e.g., in \cite{Auslender-IGA-2006}. Evaluating \eqref{eq:mirror-step} yields the specific update expressions \eqref{eq:SMART-iteration}.

\subsection{Convergence Analysis}\label{sec:SMART-convergence}
In this section, we establish the convergence rate $\mc{O}(\frac{1}{k})$ of the SMART iteration for solving problem \eqref{eq:def-KL-objective}, by recognizing it as special case of the BPG scheme \eqref{eq:MDA-subproblem}. Since the objective function has a gradient which is \textit{not} $L$-Lipschitz, the following weaker notion is instrumental.

\begin{definition}[\textbf{relative $L$-smoothness}]\label{def:L-smooth}\cite[Prop. 1]{Bauschke:2017aa} A convex function $F$ is called \emph{$L$-smooth relative} to the Legendre function $\vphi$ if $\dom\vphi \subset \dom F$ and there exists a constant $L>0$ such that
\begin{equation}\label{eq:L-phi-f-reformulated}
D_{F}(x,y) \leq L D_{\vphi}(x,y),\qquad\forall (x,y)\in \dom\vphi\times \intr(\dom\vphi).
\end{equation}
\end{definition}

\noindent
We adopt the following 
\begin{description}
\item[Assumptions A] $\text{ }$
\begin{enumerate}[(i)]
\item the Bregman kernel $\vphi$ is of Legendre type;
\item the objective function $F:\R^n\to \R\cup \{+\infty\}$ is a convex
function and continuously differentiable on $\intr \dom\vphi$, with $\dom\vphi \subset \dom F$;
\item  the objective is bounded from below, i.e. $-\infty < \inf_{x \in \dom \vphi} F(x)=:F_\ast$;

\item the objective $F$ is $L$-smooth relative to $\vphi$.
\end{enumerate}
\end{description}

\begin{remark}[\textbf{well posedness}]\label{rem:MD-well-posedness}
    Assumption A imply that the mirror descent map \eqref{eq:MDA-subproblem} is a well-defined, single-valued map $\mrm{Mirr}_{\vphi}\colon\intr(\dom\vphi)\to\intr(\dom\vphi)$ (Lemma \ref{lem:Bregman-projection}).
\end{remark}

Now we consider the problem \eqref{eq:def-KL-objective},
\begin{subequations}
\begin{gather}\label{eq:f-KL-Dphif}
    f(x) = \KL(A x, b) =: D_{\phi_{f}}(A x,b)
    \intertext{with regularizing proximal term}
    D_{\phi}(x,y)
\end{gather}
\end{subequations}
given by \eqref{eq:KL-list} and $\phi$ given by any function listed in \eqref{eq:phi-list}. 
In each case, $(f,\phi)$ play the role of $(F,\vphi)$ in the Assumptions A.

\begin{proposition}[\textbf{Existence}]\label{prop:Existence}
    Consider $\phi_{f}$ in \eqref{eq:def-Shannon-ent},  $f$ defined by \eqref{eq:f-KL-Dphif}, with $A$ and $b$  as in \eqref{eq:def-KL-objective} and $\phi$ given by any function listed in \eqref{eq:phi-list}.  Set $K = \{A x\colon x\in \dom \phi \subset\R^{n}_{+}\}\subset \R^{m}_{+}$
    and assume $b\in\intr(K)$ and $\dom\phi_{f} \cap K \neq\emptyset$. Then
    \begin{equation}\label{eq:prop-existence-problem-f}
        \underset{\R^{n}_{+}\cap\dom\phi}{\argmin} f(x) \neq \emptyset.
    \end{equation}
\end{proposition}

\begin{proof}
    Problem \eqref{eq:prop-existence-problem-f} satisfies the assumptions of Lemma \ref{lem:Bregman-projection} with the specified $K$ and $\vphi=\phi_{f}$.
\end{proof}

We continue with Lemmata preparing the main result, Theorem \ref{th:SMART1}.

\begin{lemma}\cite[Lemma 1]{Kahl:2023aa}\label{lem:def-Bregman-gap}
Let $f$ be defined by \eqref{eq:f-KL-Dphif} and consider $\phi_{f}$ in \eqref{eq:def-Shannon-ent}.
Then 
\begin{equation}
D_{f}(x,y)=D_{\phi_{f}}(A x, A y)=\KL(Ax,Ay).
\end{equation}
\end{lemma}

\begin{lemma}\label{lem:DhA-contraction}
Suppose $A \in \R_{+}^{m \times n}$ and consider $\phi_{+}$ from \eqref{eq:phi-list}. Then
\begin{equation}\label{eq:Dphi-inequality}
D_{\phi_{f}}(A x, A y) \leq \|A\|_1 D_{\phi_{+}}(x,y),\qquad
\forall x, y \in \R_{+}^{n}.
\end{equation}
\end{lemma}

\begin{proof}
We apply the log-sum inequality \cite[Thm.~2.7.1]{Cover:2006aa} 
\begin{equation}\label{eq:log-sum-inequality}
\forall a, b \in \R_{+}^{n},\qquad
\Big\la a, \log\frac{a}{b} \Big\ra
\geq \la \eins, a \ra \log\frac{\la\eins, a \ra}{\la\eins, b\ra},
\end{equation}
which naturally extends by continuity to the boundary of $\R_{+}^{n}$. By assumption, $\|A\|_1 \ge \sum_{i \in [m]} A_{ij},\; j \in [n]$. We compute 
\begin{subequations}
\begin{align}
\|A\|_1 D_{\phi_{+}}(x,y) 
&\ge  \sum_{j \in [n]} \Big(\sum_{i \in [m]} A_{ij}\Big)\Big(x_{j} \log\frac{x_{j}}{y_{j}}-x_{j}+y_{j}\Big)
\\
&= \sum_{i \in [m]}\Big(\sum_{j \in [n]}
A_{ij} x_{j} \log\frac{A_{ij} x_{j}}{A_{ij}y_{j}} - \sum_{j \in [n]} A_{ij} (x_{j}-y_{j}) \Big)
\\
&\overset{\eqref{eq:log-sum-inequality}}{\geq} 
\sum_{i \in [m]}\Big[\Big(\sum_{j \in [n]} A_{ij} x_{j}\Big)
\log\frac{\sum_{j \in [n]} A_{ij} x_{j}}{\sum_{j \in [n]} A_{ij} y_{j}} - \sum_{j \in [n]} A_{ij} (x_{j}-y_{j}) \Big]
\\
&= \Big\la (A x)\log\frac{A x}{A y}\Big\ra - \la\eins_{m}, A x - A y \ra \\
&= D_{\phi_{f}}(A x, A y).
\end{align}
\end{subequations}
\end{proof}

Next, we check the relative $L$-smoothness (Definition \ref{def:L-smooth}) of the SMART objective function for the three relevant scenarios.

\begin{proposition}\label{prop:L-smooth}
Let $f(x)=D_{\phi_{f}}(A x,b)$.  
Then $f$ is $\|A\|_1$-smooth relative to any $\phi \in\{\phi_{+},\phi_{\square},\phi_{\Delta}\}$ specified by \eqref{eq:phi-list} and Table \ref{tab:distributions-Bregman}.
\end{proposition}
\begin{proof}
We check each case.
\begin{itemize}
\item
The case $\phi_{+}$ results from Lemma \ref{lem:def-Bregman-gap} and Lemma \ref{lem:DhA-contraction}. 
\item
The case $\phi_{\square}$ follows from the explicit expression of $D_{\phi_{\square}}$ given by \eqref{eq:D-phi-box}, from the case $\phi_{+}$ just established, and from the nonnegativity $D_{\phi_{+}}(\eins_{n}-x,\eins_{n}-y)\ge 0$ for any $x, y\in \mc{B}_n$.
\item 
The case $\phi_{\Delta}$ follows from $\Delta_{n}\subset\R^{n}_{+}$ and the case $\phi_{+}$.
\end{itemize}
\end{proof}

The following main result in this section draws upon \cite[Thm. 2]{Byrne:1993aa}, \cite[Thm. 3.1]{Byrne_bounded_SMART}, and \cite[Thm. 2]{Petra2013a}. These references utilize properties of Legendre functions, specifically Prop.~\ref{prop:L-smooth} and Lem.~\ref{lem:3-point}, alongside the optimality criteria of \eqref{eq:mirror-step} and the convexity of $f$.
Compare also [Thm.1, Thm.2 (i)]\cite{Bauschke:2017aa}.

\begin{theorem}[\textbf{convergence rate}]\label{th:SMART1} Let $X_\ast$ be the solution set of \eqref{eq:def-KL-objective} with $f$ and $\phi \in\{\phi_{+},\phi_{\square},\phi_{\Delta}\}$ as in Proposition \ref{prop:L-smooth} and $L=\|A\|_1$. For  $(x_{k})_{k \in \mathbb N}$
generated by  \eqref{eq:SMART-iteration},
with starting point
$x_0\in \rint (\dom \phi)$ and $\tau_k =\tau\le 1/L$, one has the following.
\begin{itemize}
\item[(i)] The sequence $(x_k)_{k \in \mathbb N}$ converges to a unique point in $x_{\ast} \in X_\ast$, 
\begin{equation}
x_\ast:=\arg\min_{x\in X_\ast} D_\phi(x,x_0).
\end{equation}
\item[(ii)] For every $k$,
\begin{equation}
f(x_{k})- f(x_\ast)\le \frac{L D_\phi(x_\ast,x_0)}{k} \ .
\end{equation}
\end{itemize}
\end{theorem}

\subsection{Accelerated Bregman Proximal Gradient (ABPG)}\label{ABPG-section}
This section briefly reviews recent related work (Section \ref{sec:acceleration-related}) and corresponding algorithms (Section \ref{sec:accelerated-algorithms}), which served as baselines for the experimental evaluation reported in Section \ref{sec:Experiments}.

\subsubsection{Acceleration, Related Work}\label{sec:acceleration-related}

In \cite{Dragomir:2022aa}, the optimality of the $\mc{O}(1/k)$ rate for a wide range of Bregman proximal gradient (BPG) algorithms was established, under \textit{general} assumptions regarding the objective function $F$ and the Bregman function $\vphi$. In particular, $L$-smoothness of the objective function $F$ relative to the Bregman function $\vphi$ is merely required, rather than $L$-Lipschitz continuity of the gradient $\partial F$.

The derivation of \textit{accelerated} Bregman proximal gradient (ABPG) algorithms, on the other hand, require additional assumptions. The authors of \cite{Hanzely:2021vc} consider assumptions which implies a $\mc{O}(1/k^{\gamma})$ rate with $\gamma\in[1,2]$:
\begin{itemize}
    \item The \textit{triangle-scaling property} with \textit{uniform triangle-scaling exponent (TSE)} $\gamma$,
    \begin{subequations}\label{eq:TSE}
    \begin{align}
    D_{\vphi}\big((1-\theta) x + \theta z, (1-\theta) x + \theta \wt{z}\big)
    \leq \theta^{\gamma}D_{\vphi}(z,\wt{z}),\quad &\forall\theta\in[0,1],
    \quad(\textbf{uniform TSE})
    \\
    & 
    \forall x, z, \wt{z}\in\intr\dom \vphi.
    \end{align}
    \end{subequations}
    Major examples are provided by \textit{jointly} convex Bregman divergences $D_{\vphi}$ which satisfy the inequality with TSE $\gamma=1$. In particular, the KL-divergence is jointly convex and each of the related Bregman divergences \eqref{eq:KL-list}.
    \item 
    The \textit{intrinsic} TSE defined as
    \begin{subequations}
    \begin{align}
    \gamma_{\mrm{in}}
    := \limsup_{\theta\searrow 0}\frac{D_{\vphi}\big((1-\theta)x+\theta z,(1-\theta)x+\theta\wt{z}\big)}{\theta^{\gamma}} &< \infty,\qquad
    (\textbf{intrinsic TSE})
    \\ 
    \forall x, z, \wt{z}&\in\intr\dom \vphi.
    \end{align}
    \end{subequations}
    A broad class of Bregman divergences has $\gamma_{\mrm{in}}=2$ which constitutes a tight upper bound for any uniform TSE.
\end{itemize}
The analysis in \cite{Hanzely:2021vc} rests upon the \textit{triangle-scaling gain} $G(x,z,\wt{z})$ defined by the relaxed triangle-scaling inequality
\begin{equation}\label{eq:triangle-scaling gain}
    D_{\vphi}\big((1-\theta) x + \theta z, (1-\theta) x + \theta \wt{z}\big)
    \leq G(x,z,\wt{z}) \theta^{\gamma}D_{\vphi}(z,\wt{z}),\qquad\forall\theta\in[0,1].
\end{equation}
$G(x,z,\wt{z})$ is bounded based on the relative scaling of the Hessian of $\vphi$ at different points. In particular, using \eqref{eq:triangle-scaling gain}, \textit{adaptive} ABPG algorithms are proposed of the form
\begin{equation}
    \min_{x\in C} F(x)+\Psi(x),
\end{equation}
with $F$ being $L$-smooth relative to $\phi$, $C$ closed and $C,\Psi$ convex and simple, in the sense that the key step of the ABPG method 
\begin{equation}\label{eq:ABPG}
    z_{k+1}=\argmin_{x\in C}\Big\{
    F(y_{k})+\la\partial F(y_{k}), x-y_{k}\ra + \theta_{k}^{\gamma-1}L D_{\vphi}(x,z_{k})+\Psi(x)\Big\},
\end{equation}
can be solved efficiently.
The convergence analysis of ABPG uses basic relations derived by \cite{Chen_Teboulle_1993} and \cite{Tseng:2008} in order to relate two subsequent updates.

\subsubsection{Algorithms}\label{sec:accelerated-algorithms}
We specify three accelerated ABPG algorithms and briefly characterize their properties.
Here, the \textit{local} triangle-scaling property \cite{Hanzely:2021vc} is relevant which in view of Lemma \ref{lem:def-Bregman-gap} takes the form
\begin{equation}\label{eq:local-TSE}
 D_{\phi_f}(A x_{k+1}, A y_{k+1}) < \theta_k^{\gamma_k} L D_\phi(z_{k+1}, z_k),\qquad k\in\N
\end{equation}
where the definitions of $x_{k+1}, y_{k+1}, z_{k+1} $ follow below in \eqref{eq:FSMART-1}.
\begin{itemize}
    \item
    \textit{ABPG-e algorithm \cite{Hanzely:2021vc}}. Employing \textit{exponent adaption} with an initial value $\gamma_{0}>2$ in \eqref{eq:triangle-scaling gain}, $\gamma_{k}$ is reduced by a fixed factor $\delta<1$ in subsequent iterations until the local triangle-scaling property \eqref{eq:local-TSE} is met. The resulting value $\gamma_{k}$ serves as a \textit{posterior certificate} of accelerated convergence which cannot be guaranteed beforehand, however. 
    \item
    \textit{ABPG-g algorithm \cite{Hanzely:2021vc}.} Here the local triangle-scaling property is gradually achieved by \textit{gain adaption}, that is by adjusting the gain $G_{k}=G(x_{k},z_{k},\wt{z}_{k})$ within an inner iterative loop.
    \item
    \textit{FSMART algorithm \cite{Petra2013a}.} According to Proposition \ref{prop:L-smooth}, the objective function $f$ in \eqref{eq:def-KL-objective}
    is $L$-smooth relative to $\phi$ in \eqref{eq:phi-list} with $L=\|A\|_1$. 
    The ABPG iteration \eqref{eq:ABPG} leads with $\gamma=1$ and 
    $\Psi\equiv 0$ to the F(ast)-SMART iteration \cite{Petra2013a} 
    \begin{subequations} \label{eq:FSMART-1}
    \begin{align}
        y_{k}&=(1-\theta_{k})x_k+\theta_{k}z_k \\
        z_{k+1}&= z_k \exp \left( -A^\top \log \left( \frac{Ay_k}{b} \right) /  L \right)\\
        x_{k+1} &= (1-\theta_{k})x_k+\theta_{k}z_{k+1},
    \end{align}
    \end{subequations}
where $x_0=z_0\in\intr(\dom\varphi)$ and $\theta_k\in(0,1]$ satisfies 
$
\frac{1-\theta_{k+1}}{\theta_{k+1}^2}\le \frac{1}{\theta_k^2}.
$
As the uniform TSE $\gamma$ equals $1$ for our choices of $\phi$ in \eqref{eq:phi-list}, 
a $\mc{O}(1/k)$ rate can only be guaranteed theoretically according to \cite[Thm.~1]{Hanzely:2021vc}. Convergence of the sequence $(x_k)_{k \in \mathbb N}$ generated by FSMART has remained an open issue. Empirically, however, FSMART effectively accelerates the SMART iteration. In addition, no additional cost per iteration are encountered, as opposed to the inner loop for gain adaption of ABPG-g.
\end{itemize}
In \cite[Section 4.1]{Hanzely:2021vc}, the authors address challenges associated with establishing a $\mc{O}(k^{-2})$ rate. This requires bounding the geometric mean $\ol{G}:= \left (G_0^\gamma G_1 \cdots G_k\right)^\frac{1}{k+\gamma}$ of gains at each step without additional assumptions. The authors conjecture that in practical scenarios a specific reference function $\phi$ can be utilized, which may possess inherent properties contributing to rapid convergence. Exploiting such properties for the considered $\phi$  remains a subject for further research.

\section{SMART: A Geometric Perspective}
\label{sec:Geometry}

In this section, we adopt the viewpoint of information geometry. First, 
we identify suitable retractions  on each of our smooth manifolds 
(Section \ref{sec:infgeo-SMART}) which enables to describe SMART as Riemannian gradient flow with a fixed stepsize. Stepsize selection adapted by line search is studied in 
Section \ref{RG-section-line-search}. Finally, we detail various geometric variants of the conjugate gradient iteration (Section \ref{CG-section}) based on vector transport, which are supposed to accelerate Riemannian first-order gradient iterations.

\subsection{Information Geometry}\label{sec:infgeo-SMART}
In this section, we further explore the geometry of the specific manifolds introduced in Section \ref{sec:specific-manifolds}.

\subsubsection{E-Geodesics, Exponential Maps}
A basic concept of Riemannian geometry \cite[Def. 1.4.3]{Jost:2017aa} is the exponential map $\exp\colon T\mc{M}\to\mc{M}$ defined by $\exp_{p}(t)=\gamma_{v}(1)$, where $\gamma_{v}(t)$ is the geodesic curve through $\gamma_{v}(0)=p\in\mc{M}$ with $v\in T_{p}\mc{M}$ with respect to the Riemannian connection. Since we work with the e-connection, we compute below the analogous \textit{auto-parallel curves} called \textit{e-geodesics}, and define the corresponding exponential maps.

A basic result of information geometry \cite[Section 2.3]{Amari:2000} is that any statistical manifold $\mc{M}$ corresponding to a distribution of the exponential family \eqref{eq:p-exponential} is \textit{flat} with respect to the e-connection and that the natural parameters $\theta$ constitute a corresponding affine coordinate system, which means that the e-geodesics have the simple form 
\begin{equation}
\theta_{u}(t) = \theta + t u.     
\end{equation}
In particular, the natural parameters are defined everywhere, as $\dom\psi$ in Table \ref{tab:distributions-Bregman} displays. Both facts are convenient for numerical computation.

Since we model optimization problems using the $\eta$-coordinates, however, we transform the e-geodesics accordingly and still call them e-geodesics. For each specific manifold $\mc{M}$, this transformation is determined via the mappings $\eta=\eta(\theta)$ and its inverse $\theta = \theta(\eta)$ given by \eqref{eq:eta-by-theta} and \eqref{eq:theta-by-eta}, respectively, while taking into account that  $\mc{M}$ is also flat with respect to the $m$-connection, with corresponding affine coordinates $\eta$ and auto-parallel curves of the form
\begin{equation}\label{eq:eta-affine}
    \eta_{v}(t) = \eta + t v.
\end{equation}
The transformation reads
\begin{subequations}
    \begin{align}\label{eq:e-geodesic-to-eta-a}
        \gamma_{v}(t) &:= \eta\big(\theta_{u}(t)\big)\big|_{\theta = \theta(\eta), u = u(\eta, v)} 
        \\ \label{eq:e-geodesic-to-eta-b}
        &= \eta\big(\theta(\eta) + t u(\eta,v)\big)
        \intertext{with}\label{eq:u-v-tangents}
        u(\eta,v) &= d\theta(\eta) v = \frac{d}{dt}\theta\big(\eta_{v}(t)\big)\big|_{t=0}
         \overset{\substack{ \eqref{eq:eta-affine} }}{=}
         \frac{d}{dt}\theta\big( \eta + t v\big)\big|_{t=0} \overset{\substack{ \eqref{eq:metric-Geta}}}{=} G(\eta)v.\\
    \end{align}
\end{subequations}

\begin{figure}
    \centering
    \includegraphics[width=0.32\textwidth]{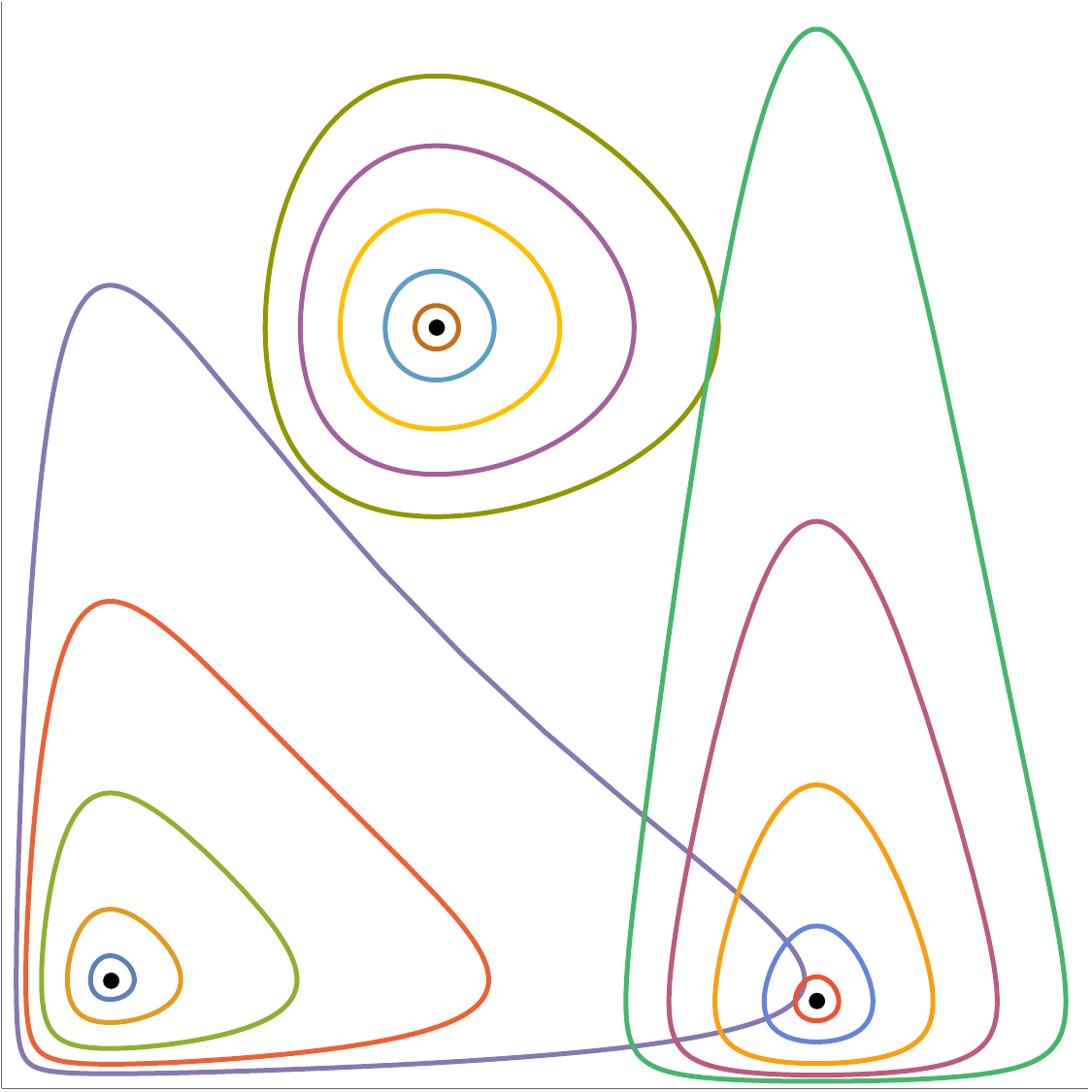}
    \hfill
    \includegraphics[width=0.32\textwidth]{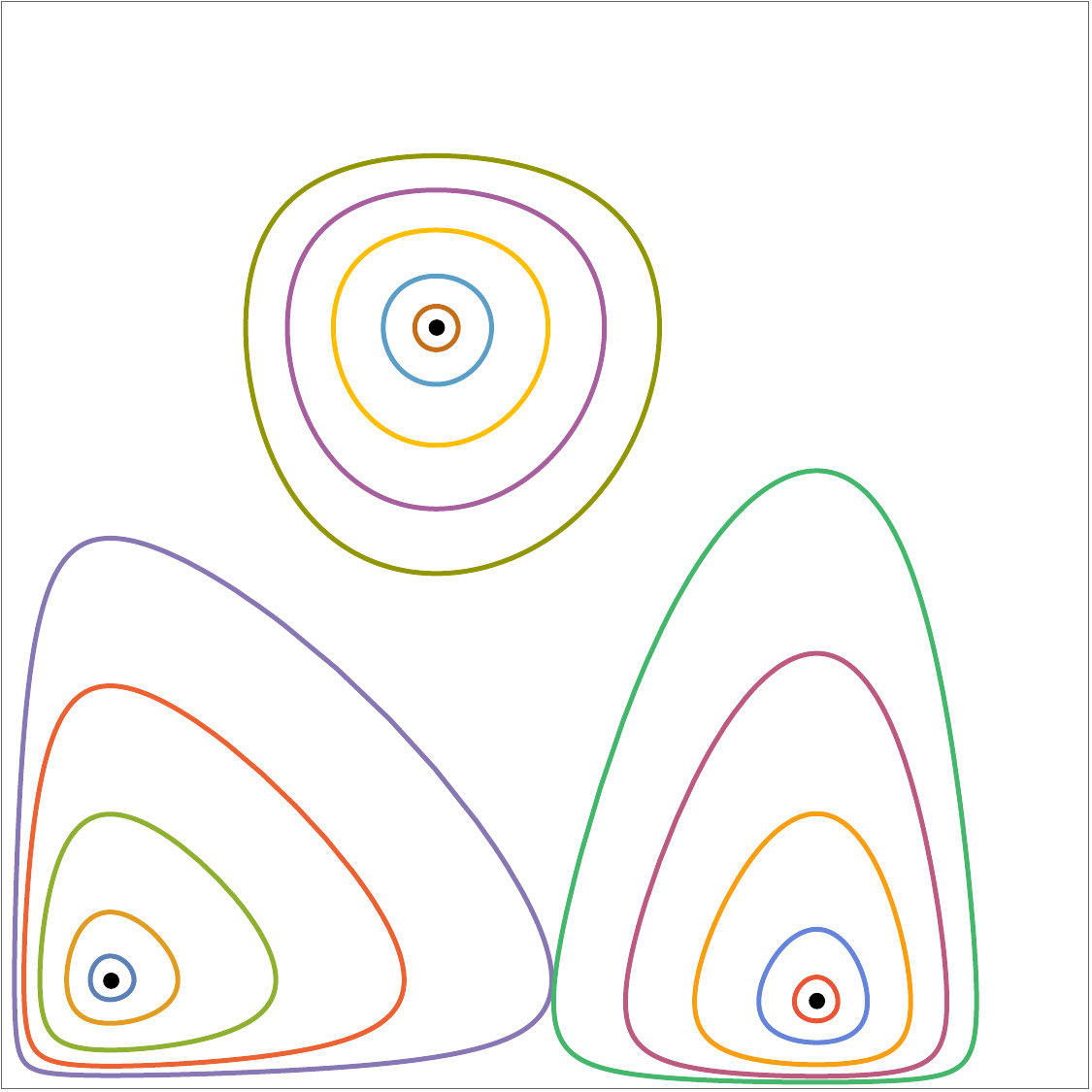}
    \hfill
    \includegraphics[width=0.32\textwidth]{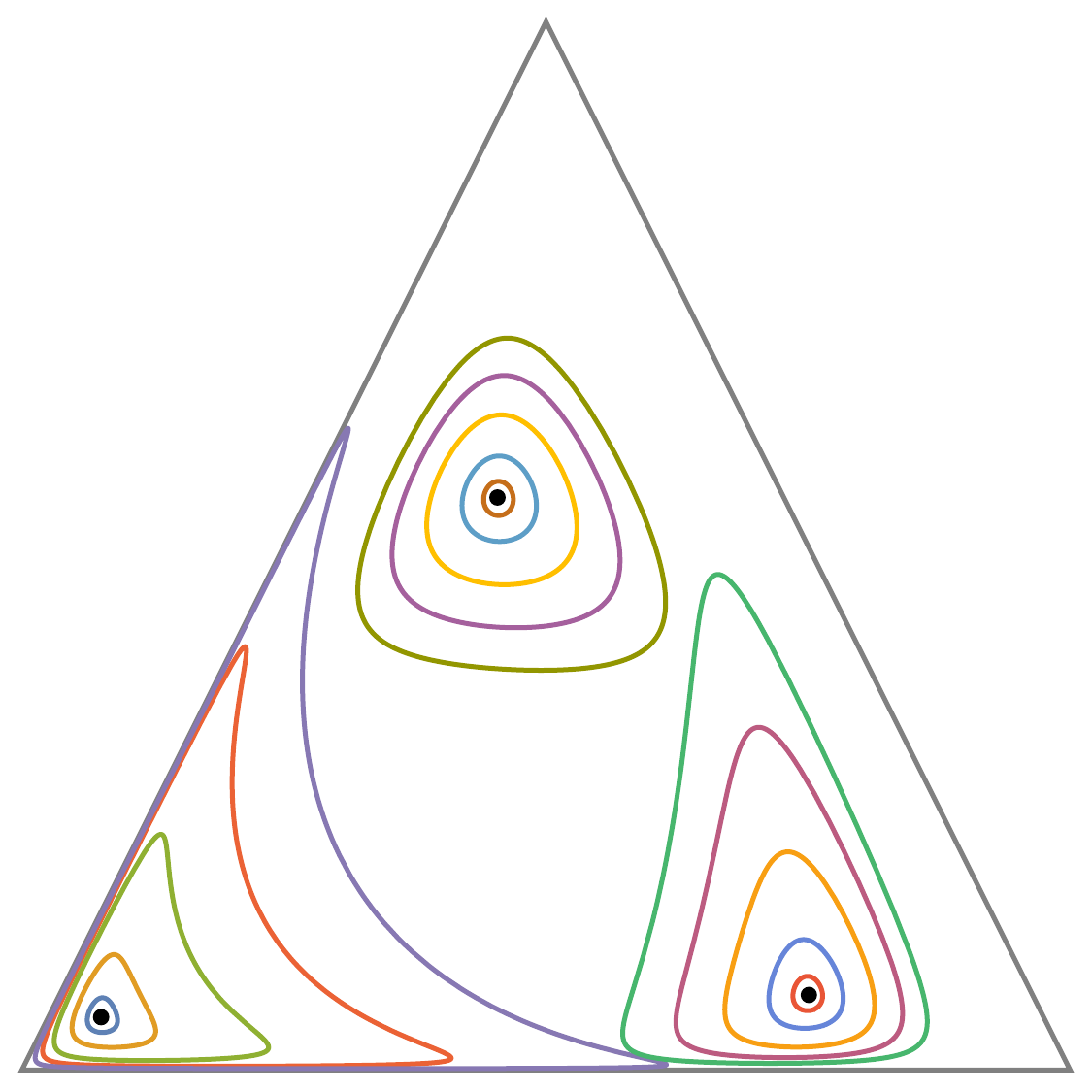}
    \centerline{
    \parbox{0.32\textwidth}{\centering Orthant $\mc{P}_{2}\cap \big((0,1)\times (0,1)\big)$}
    \hfill
    \parbox{0.32\textwidth}{\centering Box $\mc{B}_{2}$}
    \hfill
    \parbox{0.32\textwidth}{\centering Simplex $\mc{S}_{3}$}
    }
    \caption{E-geodesics \eqref{eq:Exp-maps} emanating from three points $\eta$ (black dots) in all directions $v=(\cos \omega,\sin \omega),\,\omega\in[0,2\pi)$ with $t\in\{0.02, 0.05, 0.1, 0.15, 0.2\}$.}
    \label{fig:e-geodesics}
\end{figure}

Defining
\begin{equation}\label{eq:def-e-Exp-map}
    \Exp_{\eta}(t v) := \gamma_{v}(t),
\end{equation}
which by construction is well-defined for any $t\in\R$ and tangent vector $v$ (or $v_\eta$), we obtain for each of the specific manifolds from Section \ref{sec:specific-manifolds} the formula for the corresponding \textbf{e-geodesic} (see Figure \ref{fig:e-geodesics} for an illustration)
\begin{subequations}\label{eq:Exp-maps}
    \begin{align}\label{eq:Exp-orthant}
        \mc{P}_{n}\colon\quad
        \Exp_{\eta}(tv) &= 
        \eta\cdot e^{\frac{tv}{\eta}},
        \\ \label{eq:Exp-box}
        \mc{B}_{n}\colon\quad
        \Exp_{\eta}(tv) &=
        \frac{\eta \cdot e^{\frac{tv}{\eta\cdot (\eins-\eta)}}}{\eins-\eta + \eta \cdot e^{\frac{tv}{\eta \cdot (\eins-\eta)}}},
        \\
        \vartheta(\mc{S}_{n})\colon\quad
        \Exp_{\eta}(tv_\eta) &=
        \frac{\eta\cdot e^{t G(\eta)v_{\eta}}}{ 1-\la \eins_{n-1},\eta \ra +\la \eta, e^{t G(\eta)v_\eta}\ra},\qquad G(\eta)~\text{from}~\eqref{eq:G-eta-simplex}
        \label{eq:e-geo-simplex-local} \\ \label{eq:e-geo-simplex}
        \mc{S}_{n}\colon\quad
        \Exp_{p}(tv) &=
        \frac{p\cdot e^{\frac{tv}{p}}}{\la p, e^{\frac{tv}{p}}\ra},\qquad 
        p = \vartheta^{-1}(\eta),\quad v = B v_{\eta} \in T_{0},
    \end{align}
\end{subequations}
where \eqref{eq:e-geo-simplex} results from \eqref{eq:e-geo-simplex-local} by applying $\vartheta^{-1}$ and the relations $G(\eta)= B^{\T} G(p) B$, $v = B v_{\eta}$ and $B^{\T} z = \ol{z}-z_{n}\eins_{n-1},\;\forall z\in\R^{n}$ with $B$ given by \eqref{eq:Jacobian-varphi-inv}.

\begin{lemma}\label{lemma:char-e-Exp-map}
For the e-geodesics defined by \eqref{eq:def-e-Exp-map},
we have
\begin{equation}\label{eq:char-e-Exp-map}
 \Exp_{\eta}(tv) = \partial\phi^{\ast}\big(\partial\phi(\eta) + tG(\eta)v\big).
\end{equation}
\end{lemma}

\begin{proof}
    The definition \eqref{eq:def-e-Exp-map} yields with \eqref{eq:e-geodesic-to-eta-b} 
    \begin{equation}
            \Exp_{\eta}(v)
            = \eta\big(\theta(\eta) + u(\eta,v)\big)
            \overset{\eqref{eq:u-v-tangents}}{=}
            \eta\big(\theta(\eta) + G(\eta) v\big),
    \end{equation}
    and substituting $\eta, \theta$ as given by \eqref{eq:eta-by-theta}, \eqref{eq:theta-by-eta} together with \eqref{eq:def-phi-psi-ast} shows
    \eqref{eq:char-e-Exp-map}.
\end{proof}

E-geodesics and, in particular, the mappings \eqref{eq:Exp-maps}, provide \textit{retractions} in the sense of \cite[Definition 4.1.1]{Absil:2008aa}, that are surrogates for the exponential map with respect to the Riemannian (Levi Civita) connection. This fact follows generally from \cite[Prop.~5.4.1]{Absil:2008aa}, but we include a short proof below. 

\begin{lemma}[\textbf{e-geodesics provide retractions}]\label{lem:retractions}
    Let $\mc{M}\in\{\mc{P}_{n},\mc{B}_{n},\mc{S}_{n}\}$ denote either manifold and denote by 
    \begin{equation}\label{eq:x-eta-p-notation}
        x = \begin{cases}
            \eta\in\mc{M} &\text{if}\;\mc{M}=\mc{P}_{n}\;\text{or}\;\mc{M}=\mc{B}_{n}, \\
            p\in\mc{M} &\text{if}\;\mc{M}=\mc{S}_{n},
        \end{cases}
    \end{equation}
    and by $v\in T_{x}\mc{M}$ a corresponding tangent vector. 
    Each of the mappings \eqref{eq:Exp-maps} provides a retraction
    \begin{subequations}\label{eq:cond-retraction}
        \begin{align}\label{eq:cond-retraction-a}
            R_{x} &\colon T_{x} \mc{M} \to \mc{M},
            \qquad R_{x}:= \Exp_{x},
            \intertext{i.e.~it satisfies the defining conditions}
            \label{eq:cond-retraction-b}
            R_{x}(0)&=x, \qquad\quad 0\in T_{x}\mc{M}, \\ \label{eq:cond-retraction-c}
            dR_{x}(0)&=\mrm{id}_{T_{x}\mc{M}},\quad \forall x\in\mc{M}.
        \end{align}
    \end{subequations}
\end{lemma}

\begin{proof}
    Condition \eqref{eq:cond-retraction-b} is immediate from \eqref{eq:char-e-Exp-map},
    \begin{equation}
        \Exp_{x}(0) = \partial\phi^{\ast}\big(\partial\phi(x) + 0\big) = x.
    \end{equation}
    Furthermore, we have
    \begin{subequations}\label{eq:diff-e-Exp-map}
        \begin{align}
            d\Exp_{x}(u)v & =  \frac{d}{dt}\Exp_{x}(u+tv)\big|_{t=0}
            \overset{\substack{ \eqref{eq:char-e-Exp-map}}}{=}  
            \frac{d}{dt}\partial\phi^{\ast}\big(\partial\phi(x) + G(x)(u+t v)\big)\big|_{t=0}\\
            & = \partial^2\phi^{\ast}\big(\partial \phi(x) + G(x) u\big)G(x)v.
        \end{align}
    \end{subequations}
Thus
\begin{equation}
d\Exp_{x}(0)v =  \partial^2\phi^{\ast}\big(\partial \phi(x) \big)G(x)v  \overset{\substack{ \eqref{eq:def-phi-psi-ast}\\\eqref{eq:theta-by-eta}}}{=} \partial^2 \psi(\theta) G(x)v  \overset{\substack{ \eqref{eq:Hpsi-Hphi}}}{=} v,
\end{equation}
which establishes condition \eqref{eq:cond-retraction-b}.
\end{proof}

\begin{remark}[\textbf{notation}]\label{rem:notation}
    In the remainder of this paper, we adopt the notation \eqref{eq:x-eta-p-notation}.
\end{remark}

In connection with first-order optimization, the tangent vector $v$ in \eqref{eq:Exp-maps} and in \eqref{eq:char-e-Exp-map}, respectively, will be a Riemannian gradient. Adopting the notation \eqref{eq:x-eta-p-notation} and introducing the following shorthands for the corresponding e-geodesics will be convenient.
\begin{subequations}\label{eq:exp-maps}
    \begin{align}
        \mc{P}_{n}\colon\quad
        \exp_{x}(\partial f) &:= \Exp_{x}\big(G(x)^{-1}\partial f\big)
        = x\cdot e^{\partial f}, &
        \partial f &= \partial f(x),
        \\
        \mc{B}_{n}\colon\quad
        \exp_{x}(\partial f) &:= \Exp_{x}\big(G(x)^{-1}\partial f\big)
        = \frac{x \cdot e^{\partial f}}{\eins-x + x \cdot e^{\partial f}}, &
        \partial f &= \partial f(x),
        \\
        \mc{S}_{n}\colon\quad
        \exp_{x}(\partial f) &:= \Exp_{x}(\Pi_{x} \partial f)
        = \frac{x\cdot e^{\partial f}}{\la x, e^{\partial f}\ra}, &
        \partial f &= \partial f(x).
    \end{align}
\end{subequations}

\subsubsection{Vector Transport}\label{sec:vector-transport}
In order to exploit \textit{second-order} information on a manifold for geometric optimization, a \textit{vector transport} is required \cite[Def.~8.1.1]{Absil:2008aa}. Such a mapping provides a surrogate for \textit{parallel transport}
\cite[Def.~3.1.2]{Jost:2017aa}, analogously to replacing the exponential map with respect to the Riemannian connection by a retraction. 

Given a retraction, a convenient way for obtaining a vector transport $\mc{T}$ is to take the differential of the retraction \cite[Section 8.1.2]{Absil:2008aa},
\begin{equation}\label{eq:mcT-via-R}
    \mc{T}\colon T\mc{M}\to T\mc{M},\qquad
    \mc{T}_{(x,u)}(v) = d\Exp_{x}(u) v \in T_{x'}\mc{M},\qquad u,v\in T_{x}\mc{M},\quad x' = \Exp_{x}(u).
\end{equation}
Based on the characterization of e-geodesics given in Lemma \ref{lemma:char-e-Exp-map}, we can express the induced vector transports in closed form.
\begin{lemma} \label{lem:e-transport}
For the e-geodesics defined by \eqref{eq:def-e-Exp-map},
we have
\begin{equation}\label{eq:char-diff-e-Exp-map}
d\Exp_{x}(u)v= G (x')^{-1}G(x)v,\qquad u,v\in T_{x}\mc{M},\quad x' = \Exp_{x}(u).
\end{equation}
\end{lemma}
\begin{proof} With $x' = \Exp_{x}(u)$ and $\theta':=\partial \phi (x')$, we obtain 
\begin{equation}\label{eq:theta-prime-eta-prime}
\theta'=\partial \phi(x) +G(x)u,
\end{equation}
due to \eqref{eq:char-e-Exp-map} and $\partial \phi$  and $\partial \phi^\ast$ being one-to-one.
Further, we have
\begin{equation}
 d\Exp_{x}(u)v  \overset{\substack{\eqref{eq:diff-e-Exp-map}}}{=} \partial^2\phi^{\ast}\big(\partial \phi(x) +G(x) u\big)G(x)v 
 \overset{\substack{\eqref{eq:theta-prime-eta-prime}\\\eqref{eq:def-phi-psi-ast}}}{=}  
 \partial^2 \psi (\theta')G(x)v \overset{\substack{\eqref{eq:theta-by-eta}\\\eqref{eq:Hpsi-Hphi}}}{=}G (x')^{-1}G(x)v.
\end{equation}
\end{proof}
Applying, the previous result to the concrete manifolds from Section \ref{sec:FR-Geometry}, we obtain the vector transports corresponding to the mappings \eqref{eq:Exp-maps},
\begin{subequations}\label{eq:mcT-concrete}
    \begin{align}
        \mc{P}_{n}\colon\quad
        \mc{T}_{(x,u)}(v) &= \frac{x'}{x}\cdot v, &
        x' &= \Exp_{x}(u),
        \\
        \mc{B}_{n}\colon\quad
        \mc{T}_{(x,u)}(v) &= \frac{x'\cdot (\eins-x')}{x\cdot (\eins-x)} \cdot v, &
        x' &= \Exp_{x}(u),
        \\
        \vartheta(\mc{S}_{n})\colon
        \mc{T}_{(\eta,u_\eta)}(v_\eta) &= \Pi_{\eta'} G(\eta) v_{\eta}, &
        \eta' &= \Exp_{\eta}(u_{\eta}),
        \label{eq:transport-simplex-local} \\ \label{eq:transport-simplex} 
        \mc{S}_{n}\colon\quad
        \mc{T}_{(x,u)}(v) &= \Pi_{x'}\Big(\frac{v}{x}\Big),\qquad v=B v_\eta, &
        x' &= \Exp_{x}(u),
    \end{align}
\end{subequations}
where the mapping $\Pi_{\eta'}$ and $\Pi_{x'}$ are defined by \eqref{eq:def-pi-eta} and \eqref{eq:replicator-map}, respectively, and \eqref{eq:transport-simplex} is equivalent to \eqref{eq:transport-simplex-local} due to $B (\Pi_{\eta'} G(\eta) v_{\eta}) = B \Pi_{\eta'} B^{\T} G(x) B v_{\eta} = \Pi_{x'} G(x) v$.

Below, vector transports $\mc{T}$ will be applied to Riemannian gradients $v=\ggrad f(x)$. Defining corresponding shorthands, analogous to the expressions \eqref{eq:exp-maps} for e-geodesics emanating in the direction of Riemannian gradients, will be convenient.

\begin{lemma}[\textbf{vector transport of Riemannian gradients}]\label{lem:VT-short}
    Let 
    \begin{equation}\label{eq:def-mcT-grad-f}
        \mc{T}_{(x,u)}^{g}(f) := \mc{T}_{(x,u)}(\ggrad f)
    \end{equation}
    denote the vector transport of Riemannian gradients. Then
    \begin{equation}\label{eq:mcT-grad-f-dexp}
        \mc{T}_{(x,u)}^{g}(f) = d\exp_{x}(u)\partial f,\qquad
        x' = \exp_{x}(u),\qquad \partial f = \partial f(x),
    \end{equation}
    with $\exp_{x}$ defined by \eqref{eq:exp-maps}.
\end{lemma}

\begin{proof}
    Evaluating the right-hand side of \eqref{eq:mcT-grad-f-dexp} for the expressions \eqref{eq:exp-maps} is a routine calculation (cf.~\eqref{eq:Tg-expressions} below). Substituting the argument $\ggrad f(x)$ into the corresponding expressions \eqref{eq:mcT-concrete} shows equality.
\end{proof}

Evaluating \eqref{eq:def-mcT-grad-f} or \eqref{eq:mcT-grad-f-dexp} yields the following closed-form expressions.
\begin{subequations}\label{eq:Tg-expressions}
    \begin{align}
        \mc{P}_{n}\colon\qquad
        \mc{T}_{(x,u)}^{g}(f) &= x'\cdot\partial f(x), &
        x' &= \exp_{x}(u),
        \\
        \mc{B}_{n}\colon\qquad
        \mc{T}_{(x,u)}^{g}(f) &= x'\cdot (\eins-x')\cdot\partial f(x), &
        x' &= \exp_{x}(u),
        \\
        \mc{S}_{n}\colon\qquad
        \mc{T}_{(x,u)}^{g}(f) &= \Pi_{x'}\partial f(x), &
        x' &= \exp_{x}(u).
    \end{align}
\end{subequations}

\subsection{SMART As Riemannian Gradient Descent}\label{RG-section-line-search}

In the context of minimizing an objective function $f$ defined on a Riemannian manifold $\mc{M}$, any numerical first-order update must effectively utilize the Riemannian gradient $\ggrad f(x) \in T_{x}\mc{M}$ in order to update $x\in \mc{M}$ (recall the notation: Remark \ref{rem:notation}). This typically involves the exponential map with respect to the Levi-Civita connection. However, often a retraction is used, because this is computationally cheaper than evaluating the exponential map or if the latter is not globally defined, as is the case for the simplex manifold, $\mc{S}_{n}$. 

Therefore, in this paper, we employ throughout the exponential maps \eqref{eq:Exp-maps} corresponding to the e-geodesics on each of the three manifolds $\mc{P}_{n}, \mc{B}_{n}$ and $\mc{S}_{n}$. These retractions (Lemma \ref{lem:retractions}), when applied to the Riemannian gradient, yield the expressions \eqref{eq:exp-maps}.

\begin{proposition}[\textbf{equivalence of e-geodesic and mirror descent map}]
    Application of the mappings $\Exp_{\eta}$ defined by \eqref{eq:Exp-maps}, to the Riemannian gradient of an objective function $f$, yields the corresponding mirror descent mappings \eqref{eq:mirror-step} and
    \eqref{eq:SMART-iteration}.
    
\end{proposition}

\begin{proof}
The assertion follows directly from the expressions \eqref{eq:exp-maps}.
\end{proof}

\begin{corollary} Let $(\mc{M},g)$ be endowed with the Riemannian
metrics introduced in Section \ref{sec:FR-Geometry}. Then the SMART iteration  \eqref{eq:SMART-iteration}
equals
\begin{equation}\label{eq:SMART_vs_RG}
   x_{k+1} = {\Exp}_{x_{k}}\big(-\tau_k \ggrad f(x_k)\big)
   = \exp_{x_{k}}\big(-\tau_{k}\partial f(x_{k})\big).
\end{equation}
\end{corollary}

The step size $\tau_k$ can now be adapted to the current iterate by line~search and we still obtain a global convergence result.

\subsubsection{Armijo Line Search}

Much like the Euclidean space's line search method, we can determine the step size $\tau_k$ through a curvilinear search on the manifold. 

\begin{theorem} \cite[Thm. 4.3.1]{Absil:2008aa}\label{thm:ALS-convergence}
Let $(x_{k})_{k \in \N}$ be a sequence generated by the iteration \eqref{eq:SMART_vs_RG} with step-size $\tau_k=\beta^{m} \tau$ and scalars $\tau > 0,\; \beta,\sigma \in (0,1)$, where $m$ is the smallest nonnegative integer such that
\begin{equation}\label{eq:Armijo-LS}
f(x_{k}) - f(x_{k+1}) \geq \sigma \tau_k \| \ggrad f(x_{k})\|_{x_{k}}^{2}.
\end{equation}
Then every cluster point $\wh{x}$ 
is a critical point of $f$, i.e. $\ggrad f(\wh{x})=0$.
\end{theorem}

The Riemannian gradient method with the monotone geometric Armijo line search strategy \eqref{eq:Armijo-LS}, specialized to our specific setup, is stated in  Algorithm~\ref{alg:RG_Armijo}.

\begin{algorithm}\label{alg:RG_Armijo}
\DontPrintSemicolon \textbf{initialization:} Set $k=0$, pick an
initial point $x_{0}\in\mc{M}$ and choose initial step size $\tau$, parameters $\sigma \in (0,1)$, $\beta \in (0,1)$ and $\veps>0$. Set $f_{0}=f(x_0)$, $v_{0}= -\partial f(x_0)$ and $M_0=G(x_0)^{-1}$ (cf.~Table~\ref{tab:distributions-list}).\; 
\While{ $\| M v_k\|>\veps$}{
$x_{k+1} = \exp_{x_k}(\tau v_k)$\;
$f_{k+1} = f(x_{k+1})$\;
\While{$f_{k+1} - f_k > \sigma \tau v_k^\top M_k v_k $}{
$\tau = \beta \tau$\;
$x_{k+1} = \exp_{x_k}(\tau v_k)$ by \eqref{eq:exp-maps}\;
$f_{k+1} = f(x_{k+1})$ with $f(x_{k+1})=\KL(Ax_{k+1},b)$\;
$v_{k+1} = -\partial f(x_{k+1})$ with $\partial f(x_{k+1})=A^\top \log \frac{Ax_{k+1}}{b}$\;
$M_{k+1}= G(x_{k+1})^{-1}$ with inverse metric tensor as listed in Table~\ref{tab:distributions-list}\;
}
$x_{k+1} = \exp_{x_k}(\tau v_k)$\;
Increment $k\leftarrow k+1$.\;
} \caption{SMART with geometric monotone Armijo line search \cite{Absil:2008aa}}
\end{algorithm}

\subsubsection{Hager-Zhang-type Line Search}

In \cite{HagerZhang:2004}, a novel approach to non-monotone line search was introduced, building on \cite{Grippo1986} where non-monotone strategies were first introduced for Newton's method. For a descent direction $v_k$,
a step size $\tau_k$ is determined by
checking either the non-monotone Wolfe conditions 
\begin{align}
f(x_k+\tau_k v_k) & \leq C_k + \rho_1 \tau_k \partial f (x_k)^\top v_k, \\
\partial f(x_k+\tau_k)^\top v_k  & \geq \rho_2 \partial f (x_k)^\top v_k, 
\end{align}
with scalars $\rho_1,\rho_2>0$,  or, alternatively, the non-monotone Armijo condition 
\begin{equation}\label{eq:nonmonotone-Armijo_Euclidean}
    f(x_k+\tau_k v_k) \leq C_k + \sigma \tau_k \ggrad f (x_k)^\top v_k,
\end{equation}
where 
$\tau_k=\beta^{m} \tau$ and scalars $\tau > 0,\; \beta,\sigma \in (0,1)$, with $m$ the smallest nonnegative integer such that \eqref{eq:nonmonotone-Armijo_Euclidean}
holds.

Unlike monotone strategies that strictly ensure a decrease in the sequence of function values $(f(x_{k}))_{k\in\N}$ with each iteration, this approach does not require $f(x_{k+1}) < f(x_k)$ at every step. Sacrificing monotonicity, as seen in accelerated first-order convex optimization, is a deliberate choice made to achieve faster convergence and allows for more aggressive steps. For example, in \cite{Grippo1986}, a maximum of recent function values must decrease, while in \cite{HagerZhang:2004}, the average of successive function values decreases.

\textit{Geometric} Zhang-Hager-type line search has been explored in several works, including \cite{Sutti:2020vn, Sakai2022, Oviedo:2022, Mueller2022multilevel}. In this context, we focus on the approach presented in \cite{Oviedo:2022}, which builds upon the ideas of the Zhang-Hager technique \cite{HagerZhang:2004}.

Specifically, this approach involves determining the step size $\tau_k > 0$ such that it satisfies 
\begin{equation}
f(R_{x_k}(\tau v_k)) \leq C_k + \rho_1 \tau \la \ggrad f(x_k), v_k\ra_{x_{k}} - \rho_2 \tau^2 \|v_k\|^2_{x_{k}},
\end{equation}
where $\rho_1,\rho_2>0$.
Additionally, each reference value $C_k$ is calculated as a convex combination of $C_{k-1}$ and $f(x_k)$. This allows a comparison of the current function value with a weighted average of previous values and enables an average decrease in the function value.

We now specialize \cite[Alg.~1]{Oviedo:2022} to the negative Riemannian gradient and our particular manifolds. This customized version is presented as Algorithm \ref{alg:RG_HZ}.

\begin{algorithm}\label{alg:RG_HZ}
\DontPrintSemicolon \textbf{initialization:} Set $k=0$, pick an
initial point $x_{0}\in\mc{M}$ and choose initial step size $\tau$, parameters $0< \rho_1< \rho_2<1, \beta \in (0,1), \varrho \in [0,1)$.
 Set $C_{0}=f(x_0)$, $Q_0=1$, $v_{0}= -\partial f(x_0)$ and $M_0=G(x_0)^{-1}$ (cf.~Table~\ref{tab:distributions-list}).\; 
\While{ $\| M v_k\|>\veps$}{
$x_{k+1} = \exp_{x_k}(\tau v_k)$\;
$f_{k+1} = f(x_{k+1})$\;
\While{$f_{k+1} - C_k > \tau (\rho_1 - \tau \rho_2) v_k^\top M_k v_k $}{
$\tau = \beta \tau$\;
$x_{k+1} = \exp_{x_k}(\tau v_k)$ by \eqref{eq:exp-maps}\;
$f_{k+1} = f(x_{k+1})$ with $f(x_{k+1})=\KL(Ax_{k+1},b)$\;
$v_{k+1} = -\partial f(x_{k+1})$ with $\partial f(x_{k+1})=A^\top \log \frac{Ax_{k+1}}{b}$\;
$M_{k+1}= G(x_{k+1})^{-1}$ with inverse metric tensor as listed in Table~\ref{tab:distributions-list}\;
}
$x_{k+1} = \exp_{x_k}(\tau v_k)$\;
$Q_{k+1} = \varrho Q_k +1$\;
$C_{k+1} = \frac{\varrho Q_k C_k + f(x_{k+1})}{Q_{k+1}}$\;
Increment $k\leftarrow k+1$.\;
} \caption{SMART with geometric Hager-Zhang-type line search \cite[Alg. 1]{Oviedo:2022}}
\end{algorithm}

The convergence of Algorithm \ref{alg:RG_HZ} is analyzed in \cite{Oviedo:2022}. As $v_k=-\tau_k \ggrad f(x_k)$
satisfies the conditions
\begin{align}
\la \ggrad f(x_k), v_k \ra_{x_k} & \leq -c_1 \lVert \ggrad f(x_k) \rVert^2_{x_k},\\
\lVert v_k \rVert_{x_k} & \leq c_2 \lVert \ggrad f(x_k) \rVert_{x_k}, 
\end{align}
where $c_1, c_2 >0$, convergence of the sequence 
$(x_{k})_{k \in \N}$  generated by Algorithm~\ref{alg:RG_HZ}
to a stationary point follows under the following conditions (\cite[Assumption 1]{Oviedo:2022}):
\begin{enumerate}
    \item $f \colon \mc M \to \R$ is continuously differentiable on $\mc M$ and is bounded below in the level set $\mc L_{x_0}=\{x\in\mc{M}\colon f(x)\le f(x_0)\}$;
    \item the differential $d\wh{f}_x$ of the function
    \begin{equation}
        \wh{f}_{x}\colon T_{x}\mc{M}\to\R,\qquad
        \wh{f}_{x}(v) := f\big(R_{x}(v)\big),\qquad v\in T_{x}\mc{M},
    \end{equation}
    represented by the tangent vector $u_{\wh{f}(v)}$ and the identification $T_{v}(T_{x}\mc{M})\simeq T_{x}\mc{M}$ such that 
    \begin{equation}\label{eq:def-u-by-fhat}
    d\wh{f}_{x}(v)w = \la u_{\wh{f}(v)},w\ra_{x},\qquad\forall v, w\in T_{x}\mc{M},
    \end{equation}
    is required to be uniformly Lipschitz continuous at $0 \in T_{x}\mc{M}$, i.e. there exist $\kappa, L > 0$ such that
    \begin{equation}\label{eq:cond-2-concrete}
    \lVert d\wh{f}_x(v) - d\wh{f}_x(0) \rVert_x 
    = \|u_{\wh{f}(v)}-u_{\wh{f}(0)}\|_{x}
    \leq L \lVert v \rVert_x, \qquad x \in \mc L_{x_0}, \quad v \in T_x \mc M
    \end{equation}
    holds,  where $\lVert v \rVert_x \leq \kappa$.
\end{enumerate}
We elaborate condition (2). By the chain rule and the general relation
\begin{equation}
    df(x)v = \la \ggrad f(x),v\ra_{x},\quad x\in\mc{M},\quad v\in T_{x}\mc{M},
\end{equation}
one has
\begin{equation}
    d\wh{f}_{x}(v)w = df\big(R_{x}(v)\big) dR_{x}(v) w
    = \big\la\ggrad f\big(R_{x}(v)\big), dR_{x}(v) w\big\ra_{x}.
\end{equation}
Thus, by \eqref{eq:def-u-by-fhat},
\begin{equation}
    u_{\wh{f}(v)} = dR_{x}(v)^{\T} \ggrad f\big(R_{x}(v)\big)
\end{equation}
and in particular, due to \eqref{eq:cond-retraction-b}, \eqref{eq:cond-retraction-c},
\begin{equation}
    u_{\wh{f}(0)} = \ggrad f(x).
\end{equation}
Taking now into account \eqref{eq:cond-retraction-a} and \eqref{lem:e-transport}, condition \eqref{eq:cond-2-concrete} becomes 
\begin{subequations}\label{eq:Ovieto-our-case}
    \begin{align}
        \lVert d\wh{f}_x(v) - d\wh{f}_x(0) \rVert_x 
        &= \big\|dR_{x}(v)^{\T} \ggrad f\big(R_{x}(v)\big) - \ggrad f(x)\big\|_{x}
        \\
        &= \big\|G(x)G(x')^{-2}\partial f(x') - G(x)^{-1}\partial f(x)\big\|_{x},\qquad
        x' = R_{x}(v)
        \\
        &\leq L \|v\|_{x} = L \|R_{x}^{-1}(x')\|_{x}.
    \end{align}
\end{subequations}
As a simple concrete scenario, consider the orthant $\mc{P}^{2}\subset\R^{2}$ with objective function $f(x) = D_{\phi}(x,\eins)$ given by \eqref{eq:D-phi-orthant}. Fixing $x$ and subtracting the left-hand side from the right-hand side of \eqref{eq:Ovieto-our-case}, defines the function
\begin{subequations}\label{cond-geom-L-smooth}
\begin{align}
    h_{x}(x') &:= L \|R_{x}^{-1}(x')\|_{x} - \big\|G(x)G(x')^{-2}\partial f(x') - G(x)^{-1}\partial f(x)\big\|_{x}
    \\
    \intertext{where}
    R_{x}^{-1}(x') &= x\cdot \log\Big(\frac{x'}{x}\Big)
\end{align}
\end{subequations}
by \eqref{eq:Exp-orthant} and \eqref{eq:cond-retraction-a}. Using $G(x)=\Diag(x)^{-1}$ and $\partial f(x)=\log x$, we thus have
\begin{equation}
    h_{x}(x') 
    = L\Big\|x\cdot \log\frac{x'}{x}\Big\|_{x} - \Big\|\frac{x'\cdot x'}{x}\cdot \log x' - x \cdot \log x\Big\|_{x}
\end{equation}
and specifically for $x=\eins$
\begin{equation}
    h_{\eins}(x') = L\|\log x'\| - \|x'\cdot x'\cdot \log x'\|.
\end{equation}
Since $\big((x'_{i})^{2}\log x_{i}' - L \log x_{i}'\big)\to\infty$ for $x_{i}'\to\infty$ and fixed $L>0$, 
we conclude that $h_{x}(x') \geq 0$ and thus also condition (2) above \textit{might not} hold in general.

\begin{remark}[\textbf{conditions for convergence of Algorithm \ref{alg:RG_HZ}}]\label{rem:conditions-convergence}
    The above reasoning indicates that specifying \textit{general} conditions for the convergence of Algorithm \ref{alg:RG_HZ} for \textit{each} manifold \eqref{eq:three-simple-sets} considered in this paper is a subtle issue. Although we indicated how condition (2) above might be violated, we did not take into account that $v$ is bounded, $\|v\|_{x}\leq \kappa$, which in turn bounds $x'$ by $x' = R_{x}(v)$. Thus, we regard this issue as open and leave it for future work, motivated also by the fact that Algorithm \ref{alg:RG_HZ} converged in all our numerical experiments. 
  
\end{remark}

\subsubsection{Barzilai-Borwein Method}
An alternative non-monotone line search approach expected to perform well is the Barzilai-Borwein (BB) method, introduced in \cite{BarzilaiBorwein:88} for the Euclidean setting.

Canonical line search methods in Euclidean optimization aim to approximate the exact line search problem given by
\begin{equation}\label{eq:extact-linesearch-Euclidean}
\tau_{k} = \arg \min_\tau f(x_k + \tau v_{k})
\end{equation}
where $v_{k}$ is a descent direction.

The BB method approaches this problem differently. Let $g_k = \partial f(x_k)$ denote the Euclidean gradient of $f$ at $x_k$,  $s_k = x_{k+1} - x_{k}$ and $y_k = g_{k+1} - g_{k}$. The approach is motivated  by a quasi-Newton iteration, where the update rule is given by $x_{k+1} = x_k - B_k^{-1} g_k$.
Here, $B_k$ is an approximation of the Hessian of the objective function, which satisfies the secant equation 
\begin{equation}\label{eq:sectant-equation}
B_{k+1} s_k = y_k.
\end{equation}
Barzilai and Borwein simplify $B_{k+1}$ to a scalar $\frac{1}{\tau_k}$, which typically cannot precisely fulfill the secant equation \eqref{eq:sectant-equation}. Instead, if $\la s_k,y_k\ra>0$, they approximate it as 
\begin{equation}\label{eq:approx-secant-long}
\tau_{k+1} = \arg \min_\tau \Big\| \frac{1}{\tau}s_k - y_k \big \|^2,  
\end{equation}
resulting in the so-called \emph{long BB step length},
characterized by the optimality conditions of \eqref{eq:approx-secant-long} that give
\begin{equation}
\tau_{k+1} = \frac{\langle s_k, s_k \rangle}{\langle s_k, y_k \rangle}.
\end{equation}

We can now translate this condition into the geometric setting by using the Riemannian metric, Riemannian gradient and the retraction corresponding to the constraint manifold.
The resulting Riemannian gradient method with
geometric Barzilai-Borwein line step is introduced in \cite{Iannazzo2017}.
Instead of the difference $x_{k+1}-x_k$ they consider the vector $v_k =  -\tau_k \ggrad f(x_k) \in T_{x_k}\mc M$ and transport it to $T_{x_{k+1}}\mc M$.
This yields
\begin{equation}\label{eq:def-sk-geom}
s_k := \mathcal{T}_{(x_k,v_k)}(v_k) = \mathcal{T}_{(x_k, v_{k})}(-\tau_k \ggrad f(x_k))  = -\tau_k \mathcal{T}^g_{(x_k, v_{k})}(f).
\end{equation}
The difference of the Riemannian gradients is defined  as in \cite{Iannazzo2017}
\begin{equation}\label{eq:def-yk-geom}
y_k := \ggrad f(x_{k+1}) - \mathcal{T}^g_{(x_k,v_k)}(f)) 
 \overset{\substack{ \eqref{eq:def-sk-geom}}}{=}
 \ggrad f(x_{k+1}) + \frac{1}{\tau_k} \mathcal{T}_{(x_k,v_{k})}(v_{k}).
\end{equation}
Analogously to the Euclidean setting the Riemannian long BB step length writes
\begin{equation}
\tau_{k+1} = \frac{\langle s_k, s_k \rangle_{x_{k+1}}}{\langle s_k, y_k \rangle_{x_{k+1}}}.
\end{equation}

To guarantee convergence, the step size additionally needs to fulfill the non-monotone Armijo condition.
Similar to \cite{HagerZhang:2004} the update of the cost $C_k$ is replaced by a combination of the last $m_k\le 10$ function values
\begin{align*}
    C_k = \max_{0 \leq j \leq m_k} f(x_{k-j}),
\end{align*}
where $m_0 = 0$ and $m_k = \min \{m_{k-1} +1, 10\}$ for $k > 0$.
The algorithm is summarized in Algorithm \ref{alg:RG_BB}.

\begin{algorithm}[H]\label{alg:RG_BB}
\SetAlgoLined
\DontPrintSemicolon

\KwData{initial point $x_0\in\mc{M}$, function $f$, initial step size $\tau_0$, $\gamma_{\min} \in [0,1]$, $\gamma_{\max} \geq 1$, \\
$\rho, \beta \in (0,1)$.}
\KwResult{sequence of $(x_k)$ towards $ \arg\min f$.}
$C_0 = f(x_0)$\; 
\For{$k = 0, 1,2,\dots$}{ 
$v_{k} = -\ggrad f(x_k)$\;
$x_{k+1} = R_x(\tau_k v_{k})$\;
\;

calculate $s_k = -\tau_k \mathcal{T}_{(x_k, v_{k})}(\ggrad f(x_k))$  and $y_k = \ggrad f(x_{k+1}) + \frac{1}{\tau_k} s_k$\;
calculate $\gamma_{k} = \frac{\langle s_k, s_k \rangle_{x_{k+1}}}{\vert \langle s_k,y_k \rangle_{x_{k+1}}\vert}$ \;
set $\gamma_{k} = \max(\gamma_{\min}, \min(\gamma_{k}, \gamma_{\max}))$\;
\;
find step size $\tau_k = \gamma_k \beta^m$ where $m$ is the smallest integer such that
\begin{align*}
f(R_{x_k}(\tau_k v_{k})) \leq C_k + \rho \tau_k \langle \ggrad f(x_k), v_{k} \rangle_{x_k} \end{align*}\;
update the iterate $x_{k+1} = R_{x_k}(\tau_k v_{k})$\;
update 
$C_{k+1} = \max_{0 \leq j \leq m_k} f(x_{k+1-j})$\;

}
\caption{Riemannian Gradient Descent with Barzilai-Borwein Line Search \cite{Iannazzo2017}}
\end{algorithm}

\begin{theorem}\label{thm:BB-convergence}
Let $(x_{k})_{k \in \N}$ be a sequence generated by Algorithm~\ref{alg:RG_BB}.
Then every cluster point $\wh{x}$ 
is a critical point of $f$, i.e. $\ggrad f(\wh{x})=0$.
\end{theorem}
\begin{proof}
Since the retractions considered in  \eqref{eq:Exp-maps} are globally defined, the result can be proved analogously to \cite[Thm.~3]{Hu2019}. For details see \cite[Thm.~4.13]{Thesis_Maren}, where the more general update of $C_k$ (as used in
Algorithm~\ref{alg:RG_BB}) is considered.
\end{proof}

\subsection{Accelerating SMART: Riemannian Conjugate Gradient}\label{CG-section}

The algorithms introduced above only exploit the negative Riemannian gradient $v_k = - \ggrad f(x_k)$ as descent direction at every iteration. Alternative choices of descent directions which proved to be efficient in Euclidean scenarios, may be generalized to geometric scenarios as well. 
\cite{Oviedo:2022} extends the conjugate gradient iteration to a Riemannian method by updating the search direction at every iteration to
\begin{align*}
    v_{k+1} = - \ggrad f(x_{k+1}) + \beta_k \mathcal{T}_{(x_{k},\alpha_k v_k)}(v_{k}),
\end{align*}
with the vector transport $\mc{T}$ defined by \eqref{eq:mcT-via-R} and parameter $\beta_k$ yet to be chosen.

Some of the common choices for the parameter $\beta_{k}$ for existing Riemannian conjugate gradient methods include (see \cite{Sakai2022})
\begin{subequations}\label{eq:beta-CG}\allowdisplaybreaks
    \begin{align}
        \beta_{k+1}^{FR} 
        &= \frac{\Vert \ggrad f(x_{k+1}) \Vert_{x_{k+1}}^2}{\Vert \ggrad f(x_k) \Vert_{x_k}^2}, &
        &(\text{Fletcher-Reeves})
        \\
        \beta_{k+1}^{PR} 
        &= \frac{\langle \ggrad f(x_{k+1}), y_k \rangle_{x_{k+1}}}{\Vert \ggrad f(x_k) \Vert_{x_k}^2}, &
        &(\text{Polak-Ribi\`{e}re})
        \\ &\qquad
        y_k = \ggrad f(x_{k+1}) - \mathcal{T}_{(x_{k},\alpha_k v_k)}^{g}(f), &
        &(\text{$\mc{T}^{g}$ by \eqref{eq:Tg-expressions}})
        \\\label{eq:beta-Dai-Yuan}
        \beta_{k+1}^{DY} 
        &= \frac{\Vert \ggrad f(x_{k+1}) \Vert_{x_{k+1}}^2}{\langle \ggrad f(x_{k+1}), \mathcal{T}_{(x_{k},\alpha_k v_k)}(v_k) \rangle_{x_{k+1}}- \langle \ggrad f(x_k), v_k \rangle_{x_k}}, &
        &(\text{Dai-Yuan})
        \\
        \beta_{k+1}^{HS} 
        &= \frac{\langle \ggrad f(x_{k+1}), y_k \rangle_{x_{k+1}}}{\langle \ggrad f(x_{k+1}), \mathcal{T}_{(x_{k},\alpha_k v_k)}(v_k) \rangle_{x_{k+1}}- \langle \ggrad f(x_k), v_k \rangle_{x_k}}, &
        &(\text{Hestenes-Stiefel})
        \\ \label{eq:beta-Hager-Zhang}
        \beta_{k+1}^{HZ} 
        &= \beta_{k+1}^{HS} - \mu \frac{\Vert y_k \Vert^2_{x_{k+1}} \langle \ggrad f(x_{k+1}), \mathcal{T}_{(x_{k},\alpha_k v_k)} (v_k) \rangle_{x_{k+1}}}{(\langle \ggrad f(x_{k+1}), \mathcal{T}_{(x_{k},\alpha_k v_k)} (v_k) \rangle_{x_{k+1}} - \langle \ggrad f(x_k), v_k \rangle_{x_k} )^2}, &
        &(\text{Hager-Zhang})
        \\ \label{eq:beta-oviedo}
        \beta_k^{OV} 
        &= \mu_k \;\frac{\langle \ggrad f(x_{k+1}), \mathcal{T}_{(x_{k},\alpha_k v_k)}(v_k) \rangle_{x_{k+1}}}{- \lVert v_k \rVert^2_{x_k}}, &
        &\text{\cite{Oviedo:2022}}
    \end{align}
\end{subequations}
where $\mu>0$ in \eqref{eq:beta-Hager-Zhang} and $\{\mu_k\}$ in \eqref{eq:beta-oviedo} is a bounded positive sequence such that $0 < \mu_k < \infty,\;\forall k$.

\vspace{0.5cm}
Algorithm \ref{alg:CG} summarizes the Riemannian CG iteration.

\vspace{0.5cm}

\begin{algorithm}[H]\label{alg:CG}
\SetAlgoLined
\DontPrintSemicolon

\KwData{initial point $x_0$, objective function $f$, initial step size $\alpha_0$, $\sigma \in (0,1)$, $\beta \in (0,1)$.}
\KwResult{sequence of $(x_k)_{k\in\N}$ converging to a minimum of $f$.}

Set the initial search direction to $v_0 = - \ggrad f(x_0)$.\;
\For{$k = 0, 1,2,\dots$}{ 
$\ol{\alpha} = \alpha_k$\;
Find a step size $\alpha_k = \ol{\alpha} \beta^m$, where $m$ is the smallest integer such that (with $R_{x_{k}}$ by \eqref{eq:cond-retraction-a})
\[
f\big(R_{x_k}(\alpha_k v_k)\big) - f(x_k) \leq \sigma \alpha_k \langle \ggrad f(x_k), v_k \rangle_{x_k}.\;
\]

Update the iterate to $x_{k+1} = R_{x_k}(\alpha_k v_k)$.\;

Determine $\beta_{k+1}$ by evaluating one of the expressions \eqref{eq:beta-CG}.

Update the search direction to 
\[
v_{k+1} = - \ggrad f(x_{k+1}) + \beta_k \mathcal{T}_{(x_{k},\alpha_k v_{k})}(v_k).\;
\]
$k = k+1$
}
\caption{Riemannian Conjugate Gradient \cite{Sakai2022}}
\end{algorithm}

\begin{remark}[\textbf{convergence conditions for Algorithm \ref{alg:CG}}]\label{rem:conditions-convergence-CG}
Conditions \cite[Ass. 3.1, 3.2]{Sakai2022} that imply convergence can be extended to a retraction. For instance, condition \cite[Ass. 3.2]{Sakai2022} leads to \eqref{cond-geom-L-smooth}, which unfortunately does not hold in general for our manifolds. While we consistently observe convergence in all our numerical experiments, we defer the detailed analysis for future investigation.
\end{remark}

Table \ref{tab:list-of-algorithms} summaries the
algorithms explored in this paper.

\begin{table}[htbp]
\begin{center}
 \begin{tabular}{||c | c  | c | c ||} 
 \hline
Algorithm & Convergence rate & Convergence of iterates & Boundary\\ [0.5ex] 
 \hline\hline
 \texttt{SMART} & $O(1/k)$ {\small (Thm.~\ref{th:SMART1} (ii))} & towards solution {\small (Thm.~\ref{th:SMART1} (i))}&\checkmark\\ 
 \hline
 \texttt{FSMART} & {\small emp.}~$O(1/k^{\gamma})$ {\footnotesize \cite[Thm. 1]{Hanzely:2021vc}} & open &\checkmark\\
 \hline
 \texttt{FSMART-e} & {\small emp.}~$O(1/k^{\gamma})$ {\footnotesize \cite[Thm. 2]{Hanzely:2021vc}} & open &\checkmark\\
 \hline
\texttt{FSMART-g} & {\small emp.}~$O(1/k^{\gamma})$ {\footnotesize \cite[Thm. 3]{Hanzely:2021vc}} & open&\checkmark\\
 \hline
\texttt{RG-Armijo} & open & conv.~subsequence {\small (Thm.~\ref{thm:ALS-convergence})} & \xmark\\ 
\hline
\texttt{RG-HZ} & open & open & \xmark \\
\hline
\texttt{RG-BB}  & open & conv.~subsequence {\small (Thm.~\ref{thm:BB-convergence})} & \xmark\\
\hline
\texttt{CG} & open &  open& \xmark \\[1ex] 
 \hline
\end{tabular}
\end{center}
\caption{
\textbf{Overview. Algorithms and their properties:} convergence rate, convergence of iterates (vs.~function values), boundary behavior. The first four algorithms are well-defined on the boundary of the feasible sets whereas the latter four are not. The FSMART algorithms exhibit empirically acceleration with a convergence rate $\mc{O}(1/k^{\gamma})$, $\gamma \in (1,2)$. For FSMART-e (ABPG-e \cite[Alg. 2]{Hanzely:2021vc} adapted to our scenario) and FSMART-g (ABPG-g \cite[Alg. 3]{Hanzely:2021vc} adapted to our scenario) only this can be checked a posteriori. In summary, SMART is a well understood method that achieves optimal convergence rate in terms of function values and converges
to the minimum Bregman distance solution from any starting point. The table lists a range of algorithms which emerged from SMART. They show empirically a faster convergence rate and stimulate research on various open points relevant to non-Lipschitz scenarios of optimization.
}
\label{tab:list-of-algorithms}
\end{table}

\section{Experiments}
\label{sec:Experiments}
\newcommand{\showExperimentToy}[1]{
\begin{centering}
		\begin{tabular}{c@{\hskip 2.4em}c}
  			\footnotesize \texttt{Trajectory of Iterates} & \footnotesize \texttt{Objective vs. Iterations} 
            \\[-0.6em]
            
		    \includegraphics[valign=T, width=#1\textwidth]{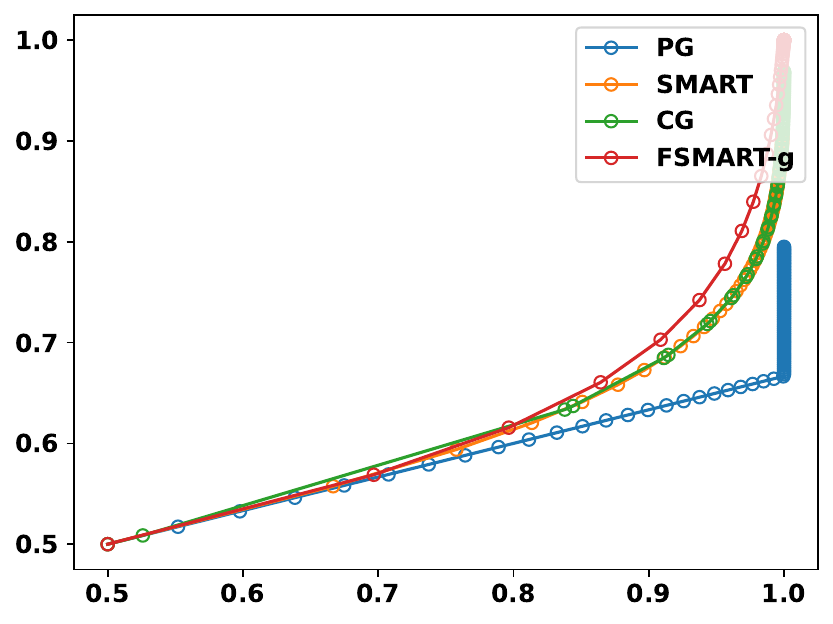}
		    &
            \includegraphics[valign=T, width=#1\textwidth]{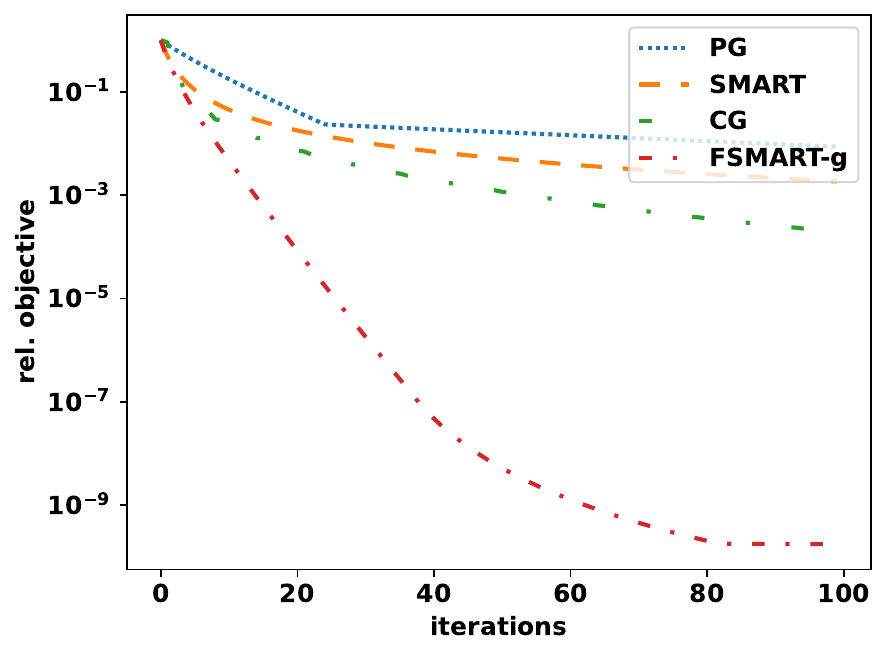}
			
		\end{tabular}
	\end{centering}
}

\newcommand{\showExperimentExpanderCG}[1]{
\begin{centering}
		\begin{tabular}{c@{\hskip 0.4em}c@{\hskip 0.4em}c}
			$\mathbf{m \times n = 40 \times 200}$ & $\mathbf{m \times n = 70 \times 200}$ & $\mathbf{m \times n = 100 \times 200}$
            \\[0.2em]

		    \includegraphics[width=#1\textwidth]{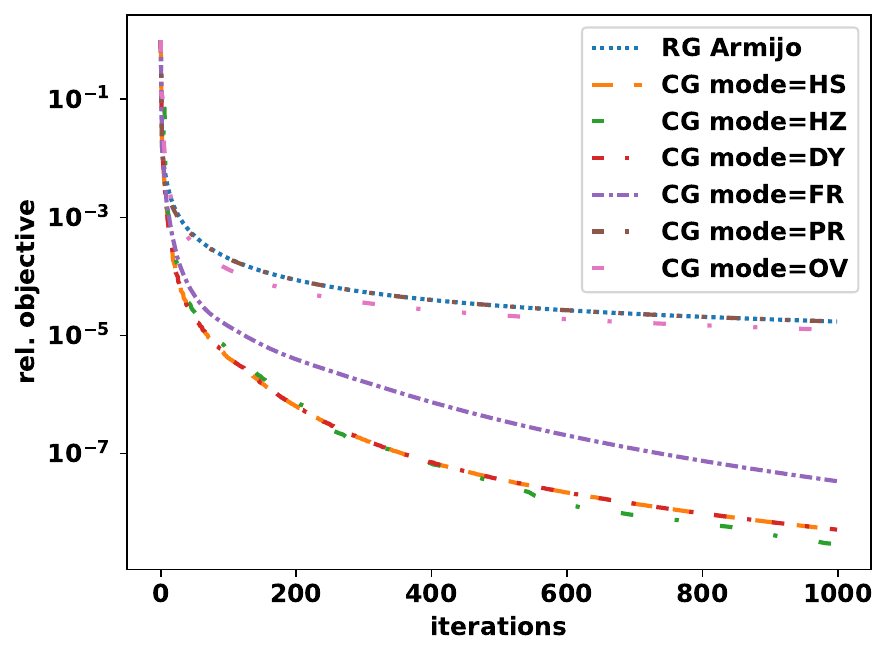}
		    &
		    \includegraphics[width=#1\textwidth]{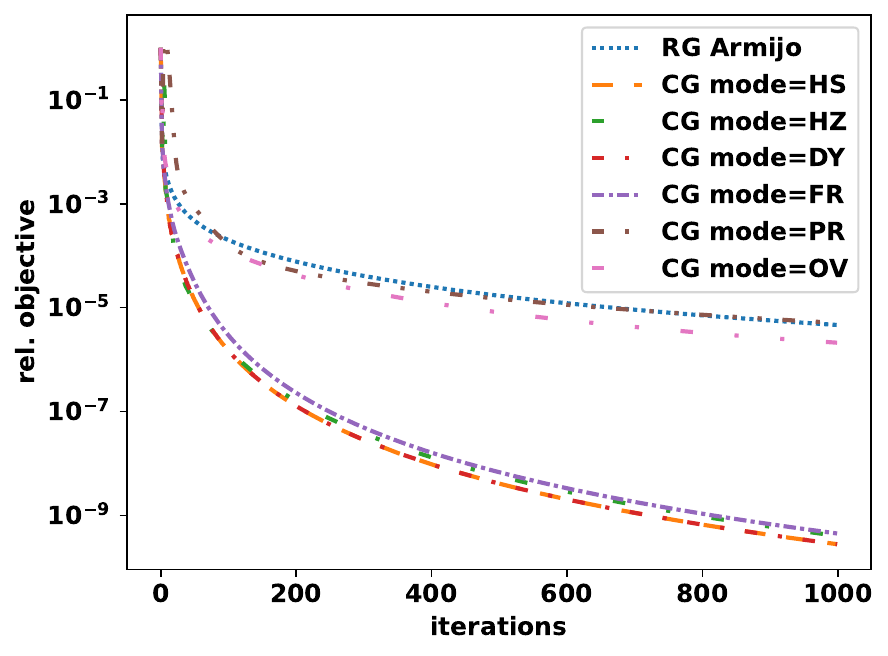}
            &
		    \includegraphics[width=#1\textwidth]{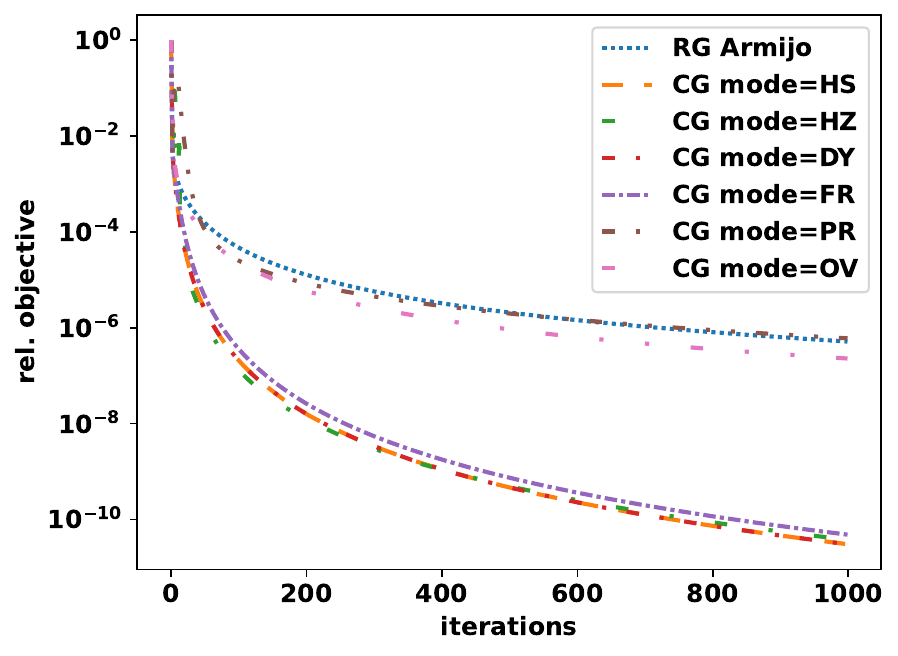}
            \\[0.2em]

		    \includegraphics[width=#1\textwidth]{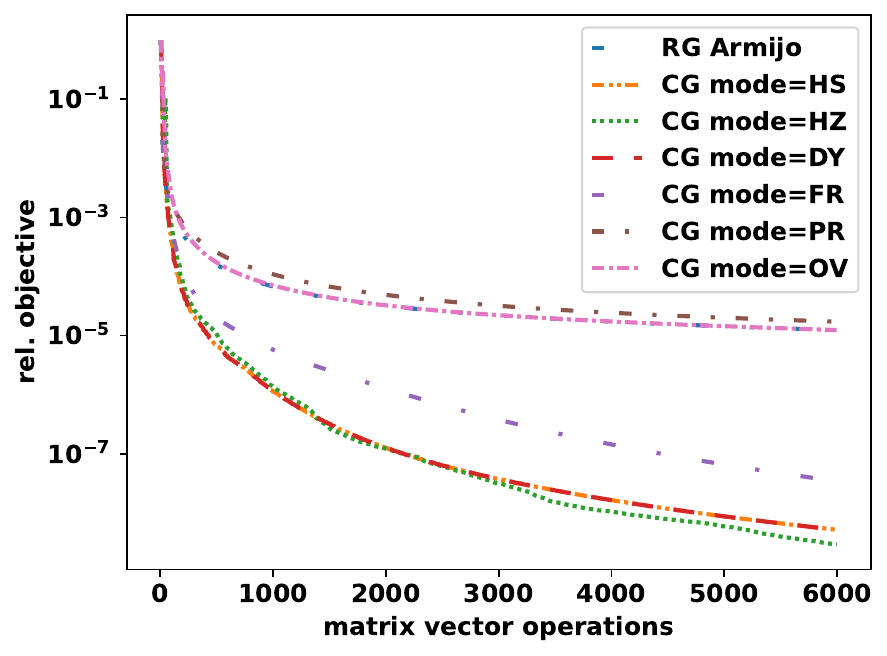}
		    &
		    \includegraphics[width=#1\textwidth]{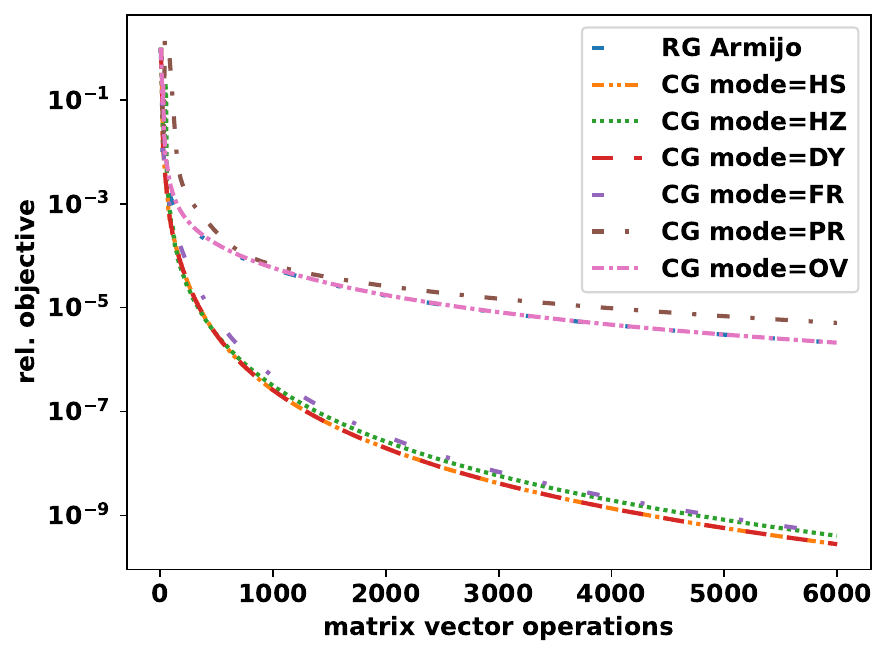}
            &
		    \includegraphics[width=#1\textwidth]{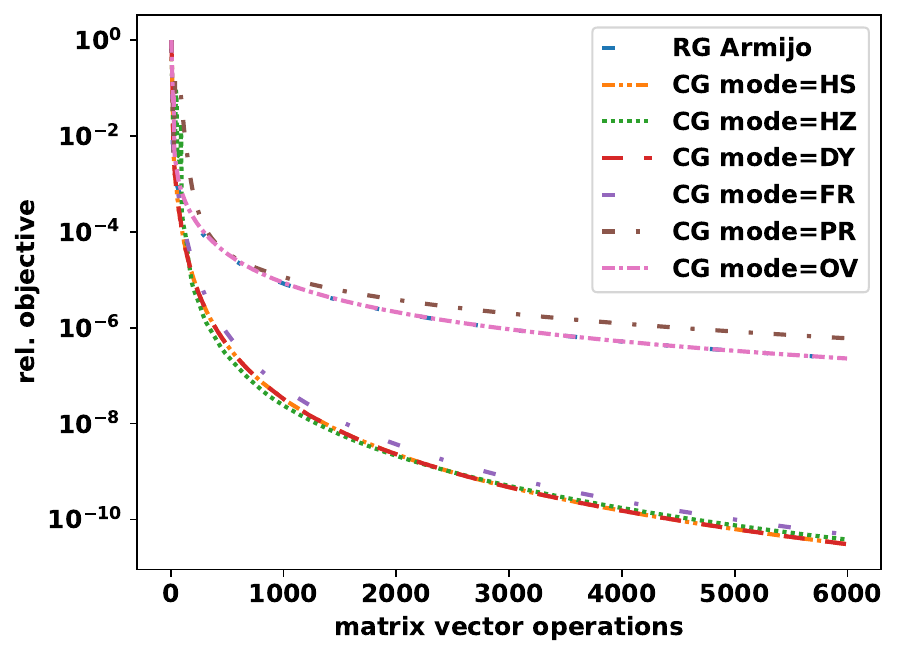}

		\end{tabular}
	\end{centering}
}

\newcommand{\showExperimentExpander}[1]{
\begin{centering}
		\begin{tabular}{c@{\hskip 0.4em}c@{\hskip 0.4em}c}
			$\mathbf{m \times n = 40 \times 200}$ & $\mathbf{m \times n = 70 \times 200}$ & $\mathbf{m \times n = 100 \times 200}$
            \\[0.2em]

		    \includegraphics[width=#1\textwidth]{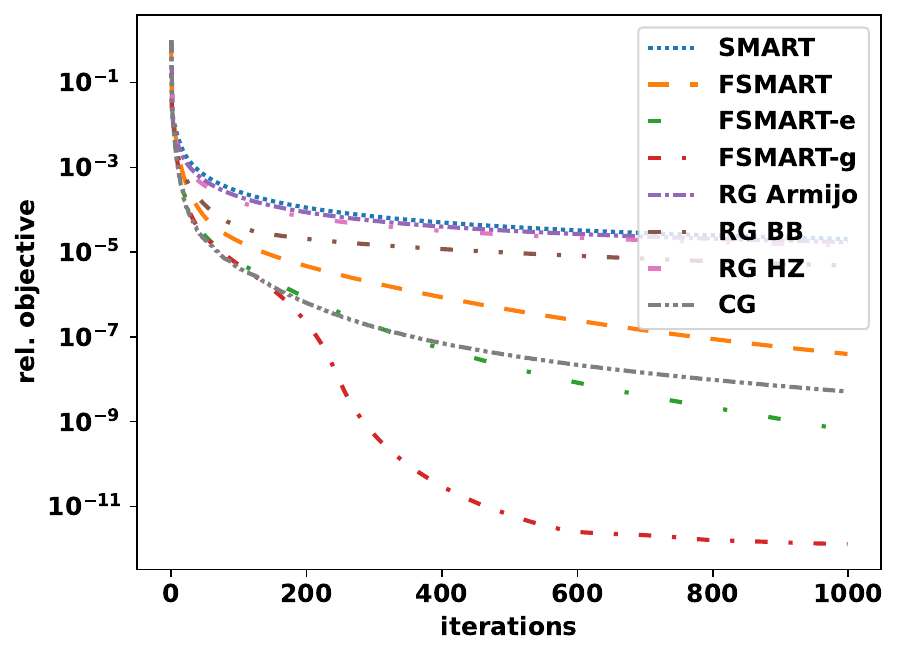}
		    &
		    \includegraphics[width=#1\textwidth]{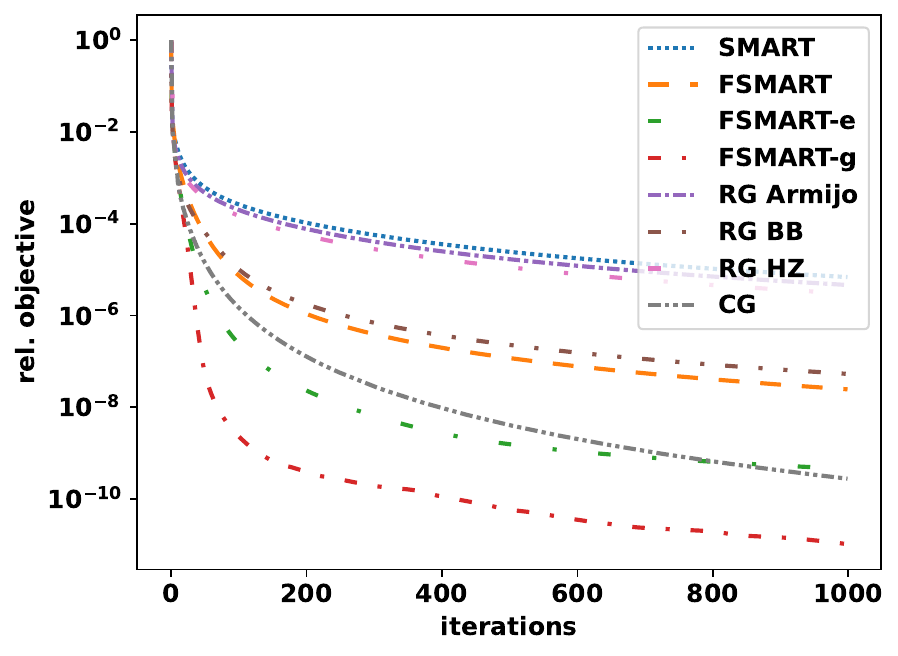}
            &
		    \includegraphics[width=#1\textwidth]{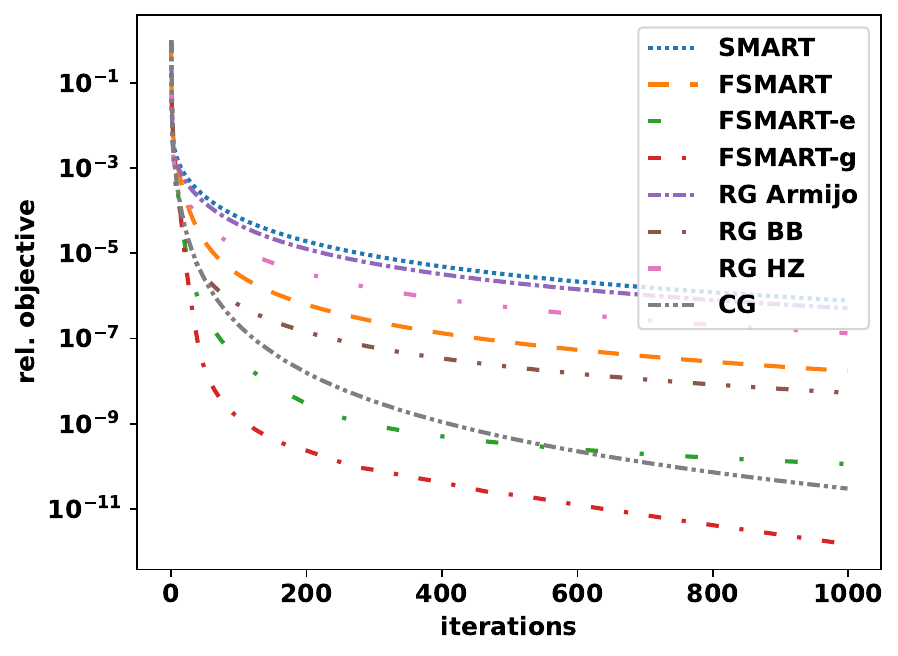}
            \\[0.2em]

		    \includegraphics[width=#1\textwidth]{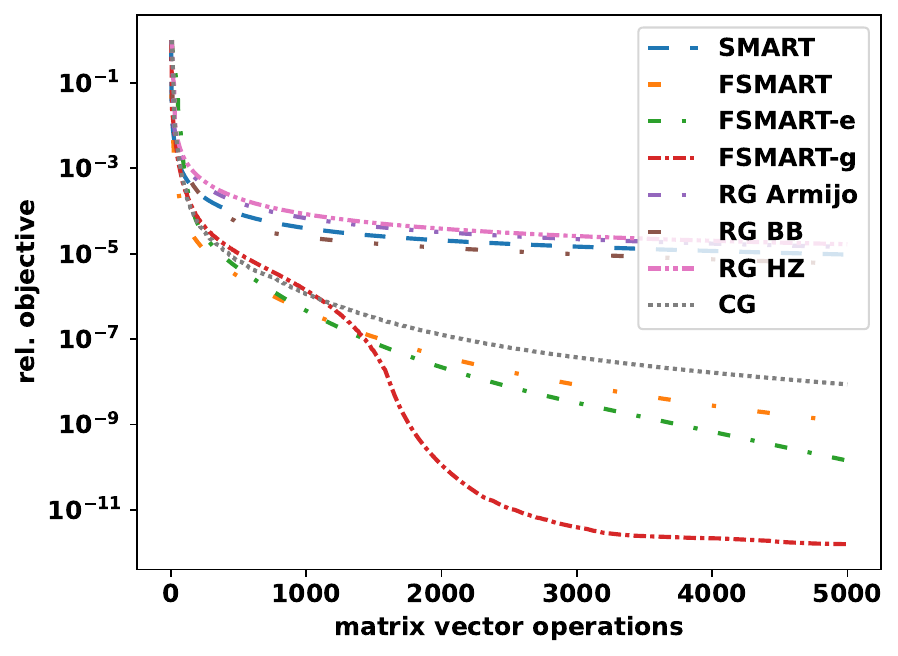}
		    &
		    \includegraphics[width=#1\textwidth]{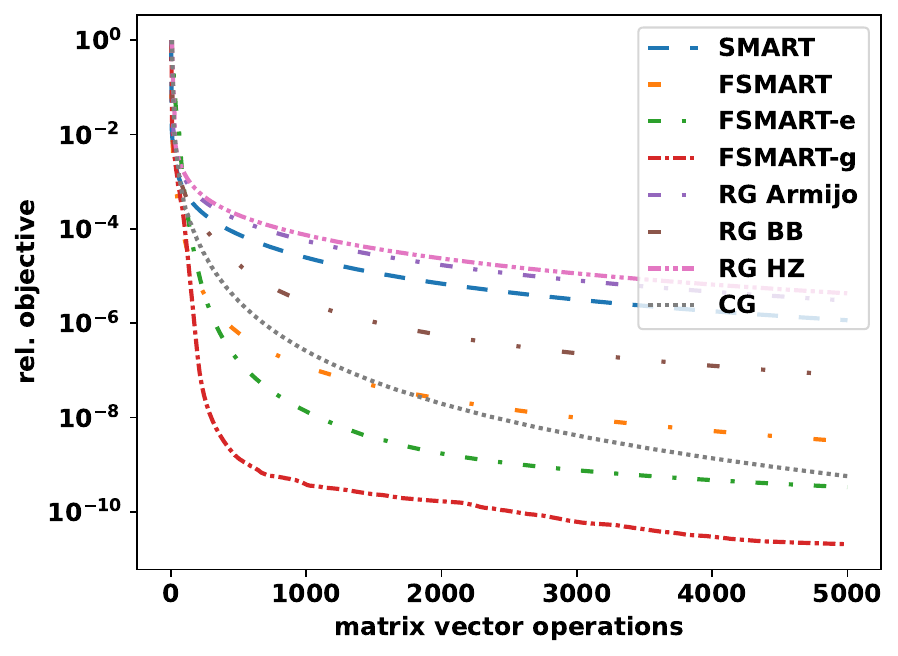}
            &
		    \includegraphics[width=#1\textwidth]{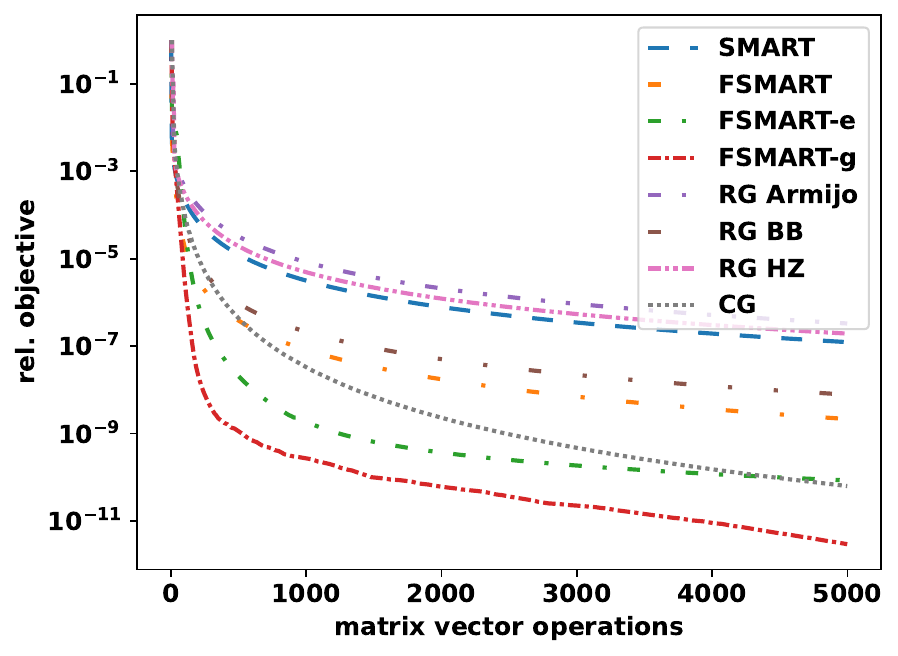}
		\end{tabular}
	\end{centering}
}

\newcommand{\showExperimentTomo}[1]{
\begin{centering}
		\begin{tabular}{c@{\hskip 0.4em}c@{\hskip 0.4em}c}
			\footnotesize \texttt{Shepp-Logan} & \footnotesize \texttt{Bone}  & \footnotesize \texttt{Vessel} 
            \\[0.2em]

		    \includegraphics[width=#1\textwidth]{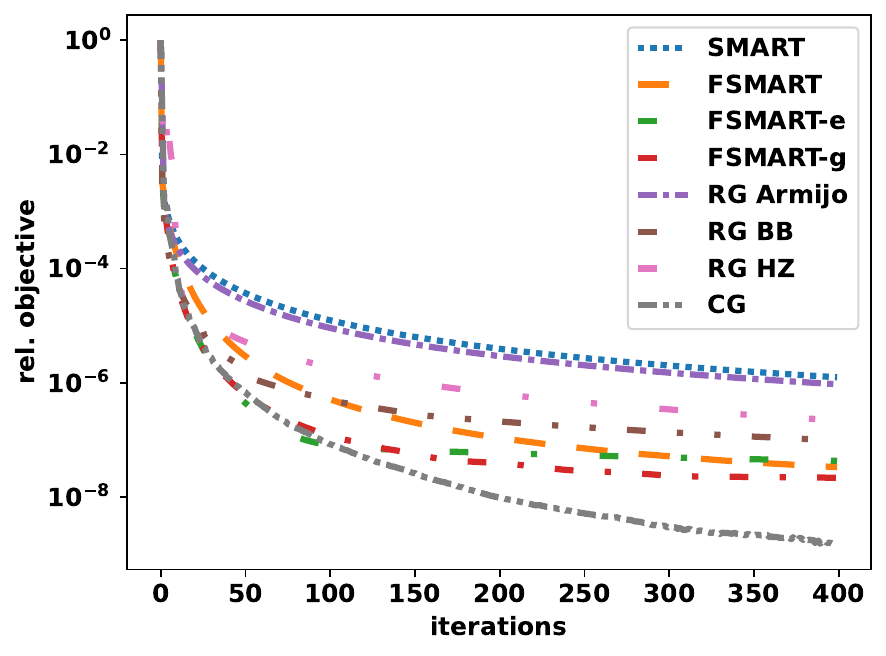}
            &
		    \includegraphics[width=#1\textwidth]{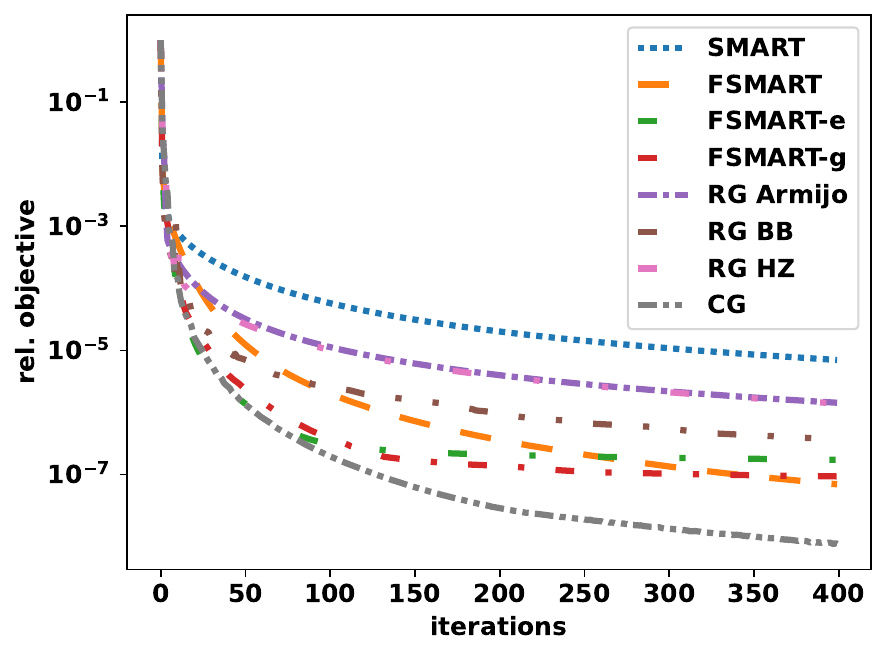}
		    &
		    \includegraphics[width=#1\textwidth]{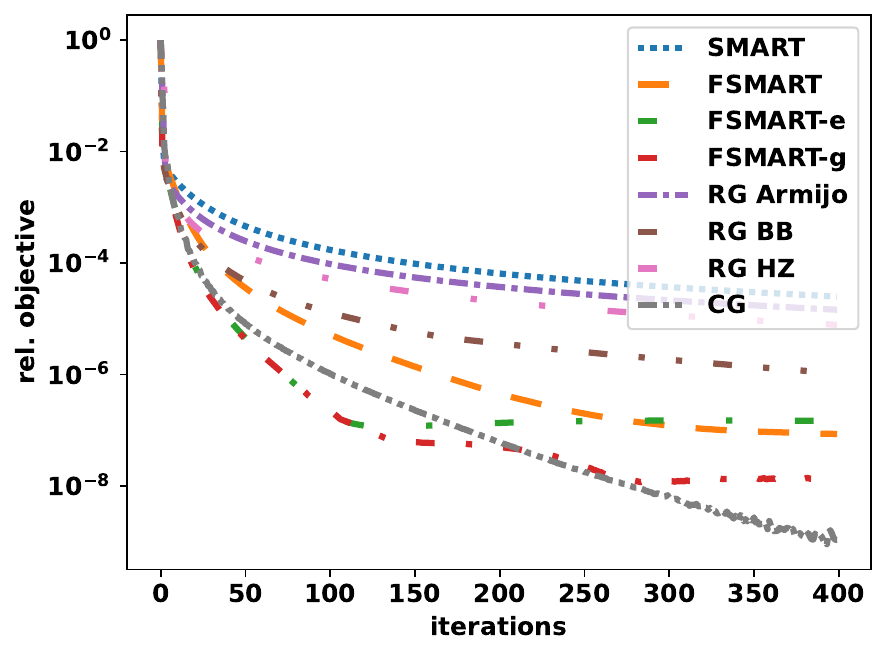}
            \\[0.2em]

		    \includegraphics[width=#1\textwidth]{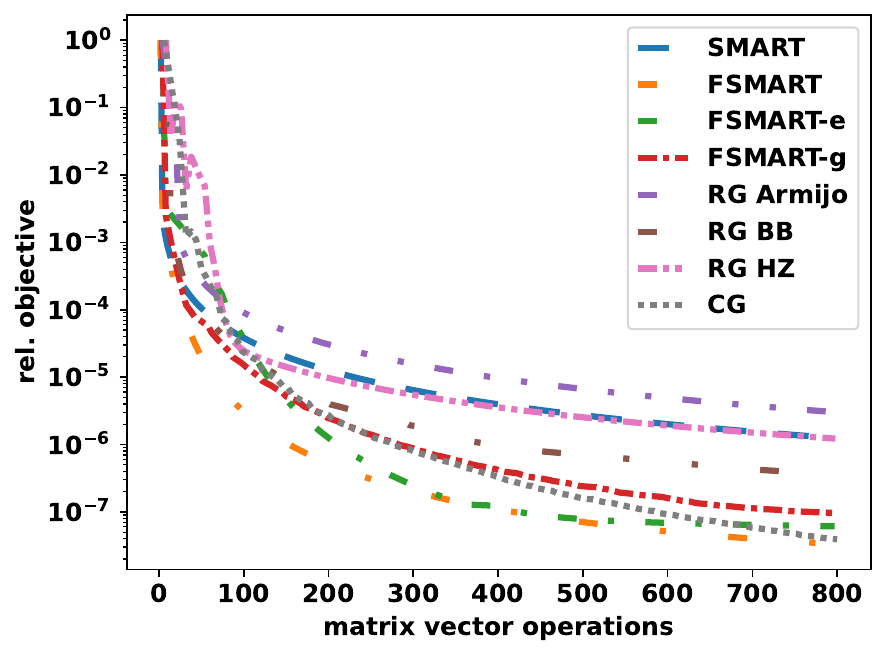}
            &
		    \includegraphics[width=#1\textwidth]{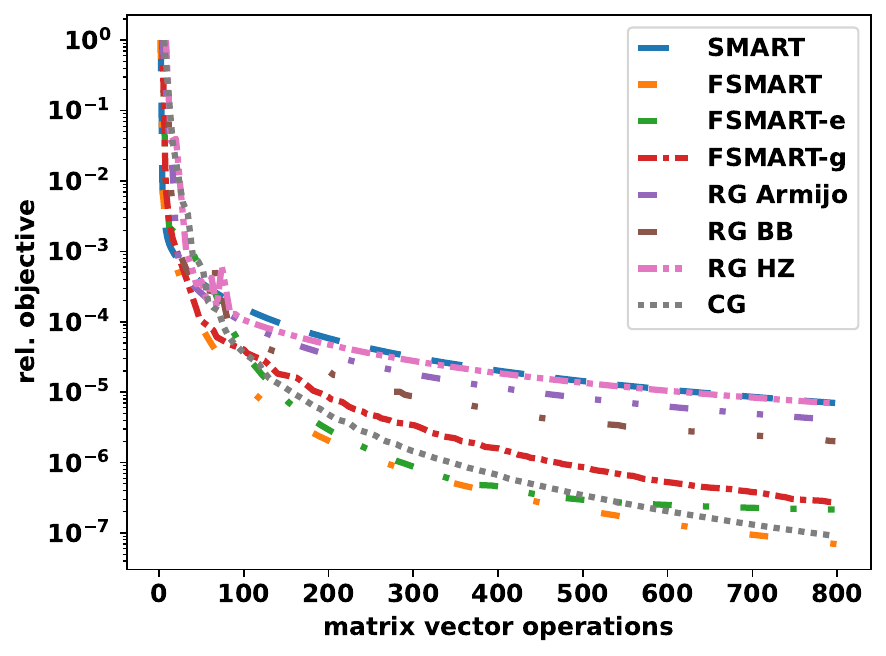}
		    &
		    \includegraphics[width=#1\textwidth]{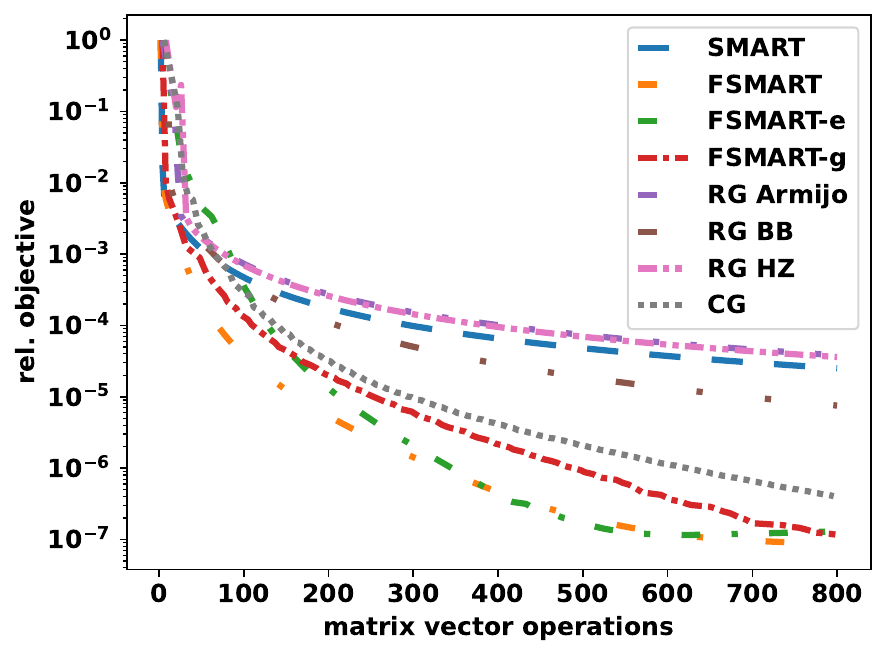}
		\end{tabular}
	\end{centering}
}

\newcommand{\showExperimentBlur}[1]{
\begin{centering}
		\begin{tabular}{c@{\hskip 0.4em}c@{\hskip 0.4em}c}
			\footnotesize \texttt{Kitten} & \footnotesize \texttt{Tiger}  & \footnotesize \texttt{QR-Code} 
            \\[0.2em]

		    \includegraphics[width=#1\textwidth]{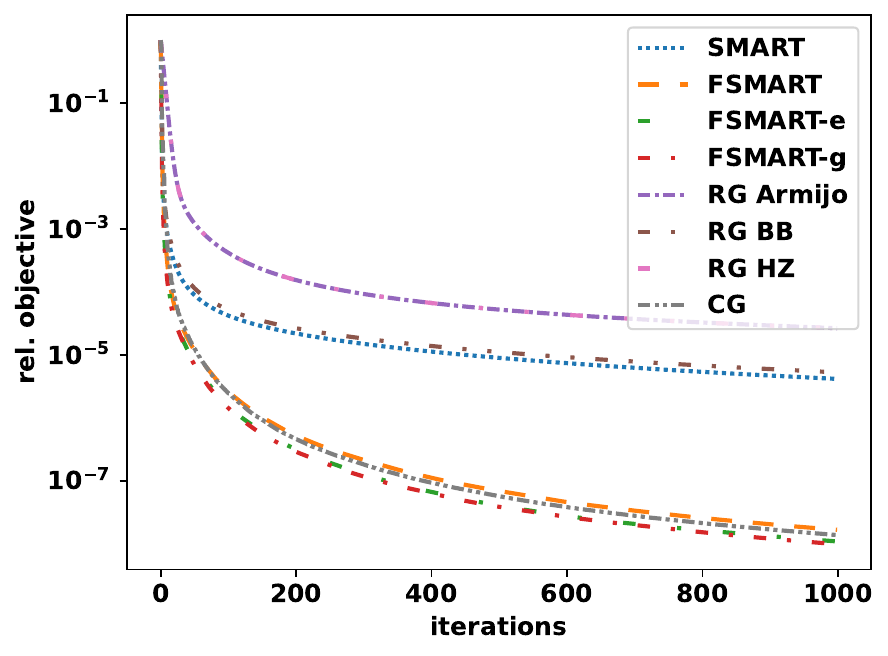}
            &
		    \includegraphics[width=#1\textwidth]{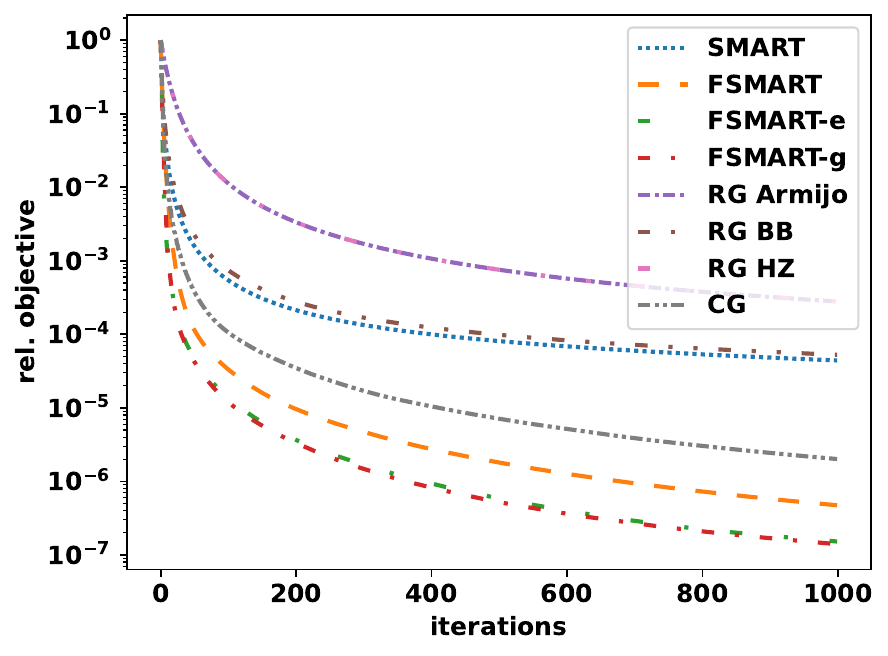}
		    &
		    \includegraphics[width=#1\textwidth]{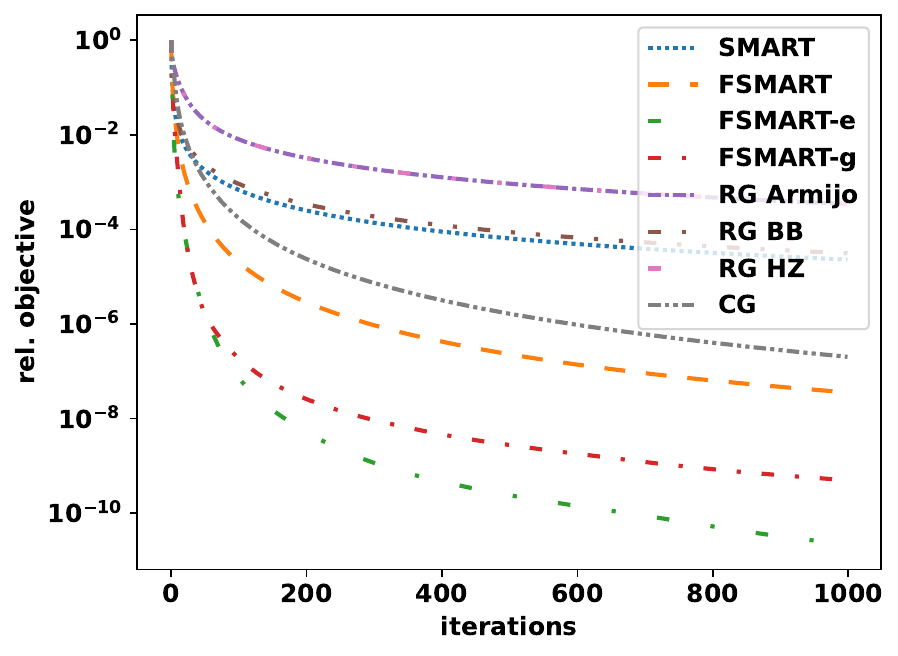}
            \\[0.2em]

		    \includegraphics[width=#1\textwidth]{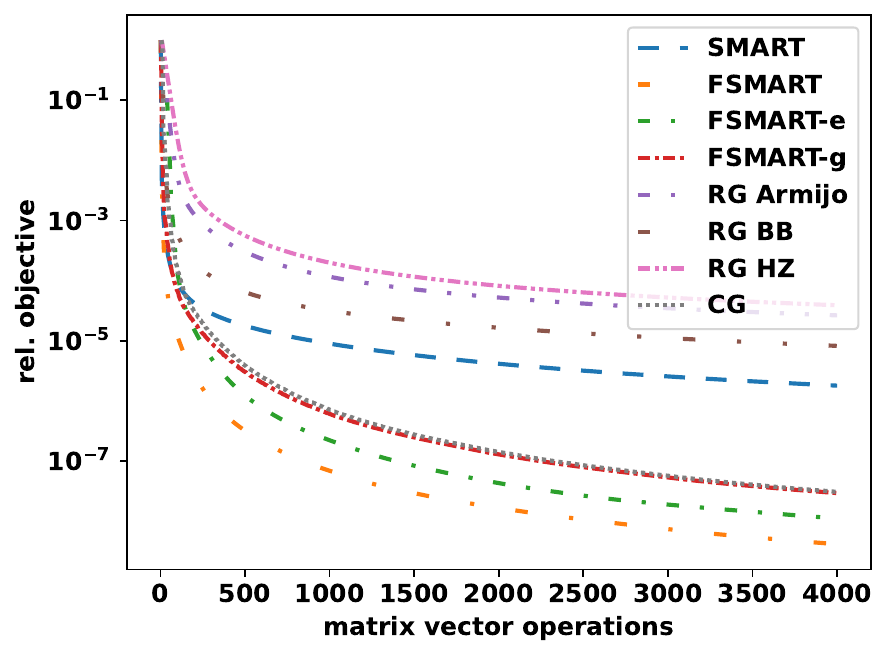}
            &
		    \includegraphics[width=#1\textwidth]{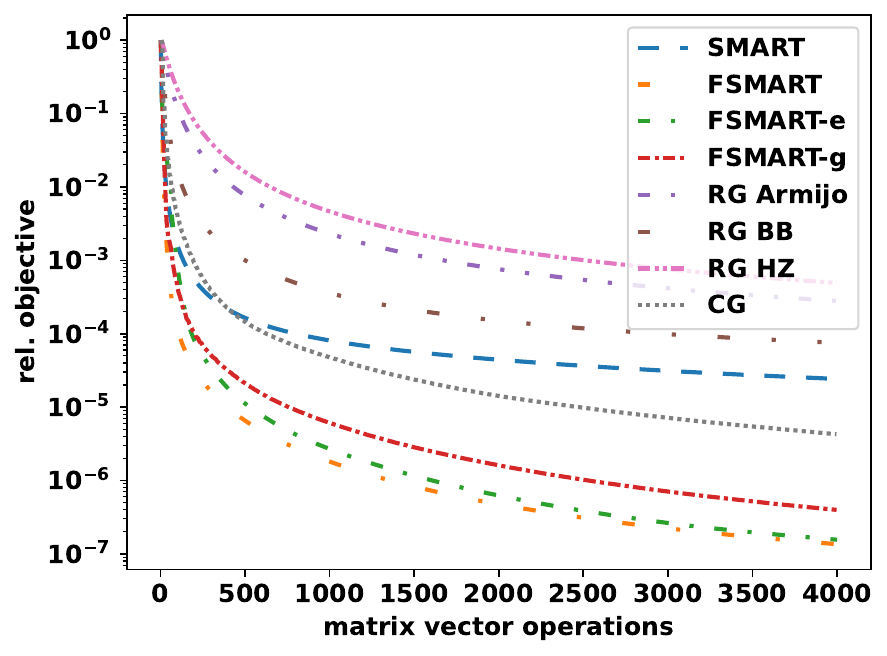}
		    &
		    \includegraphics[width=#1\textwidth]{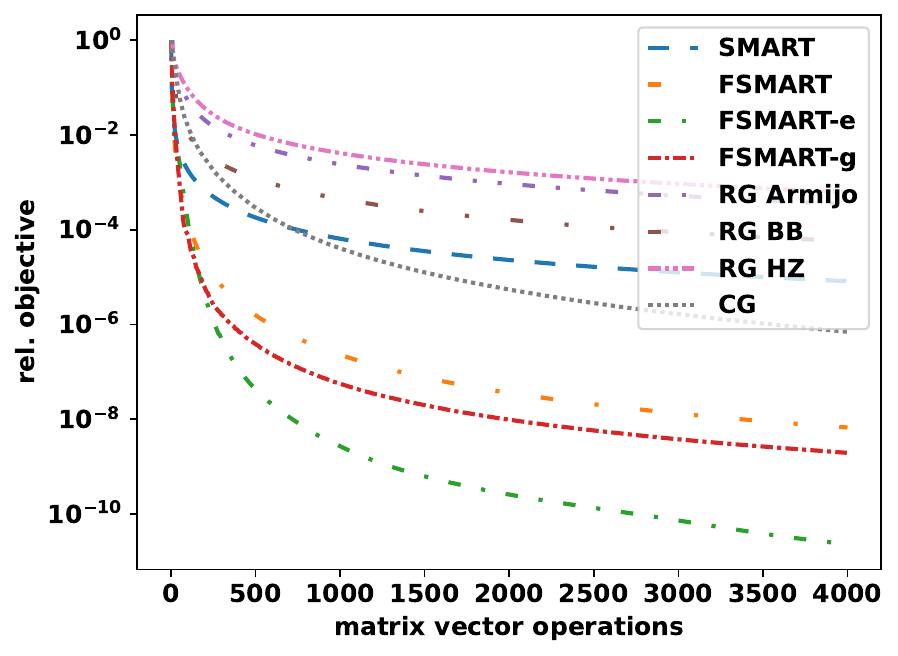}
		\end{tabular}
	\end{centering}
}

\newcommand{\showExperimentExpanderSpikes}[1]{
\begin{centering}
		\begin{tabular}{c@{\hskip 0.4em}c@{\hskip 0.4em}c}
			\footnotesize \texttt{Original} & \footnotesize \texttt{SMART}  & \footnotesize \texttt{FSMART} 
            \\[0.2em]

		    \includegraphics[width=#1\textwidth]{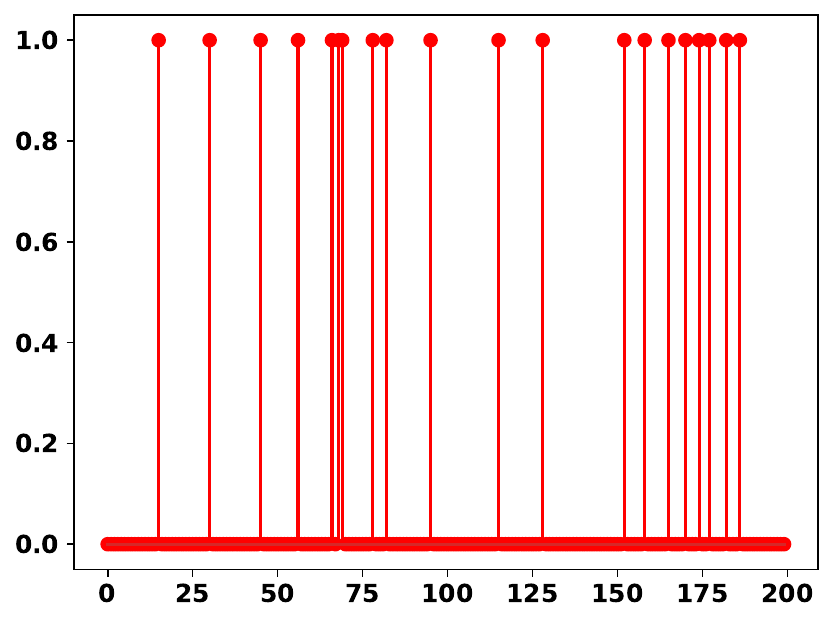}
		    &
		    \includegraphics[width=#1\textwidth]{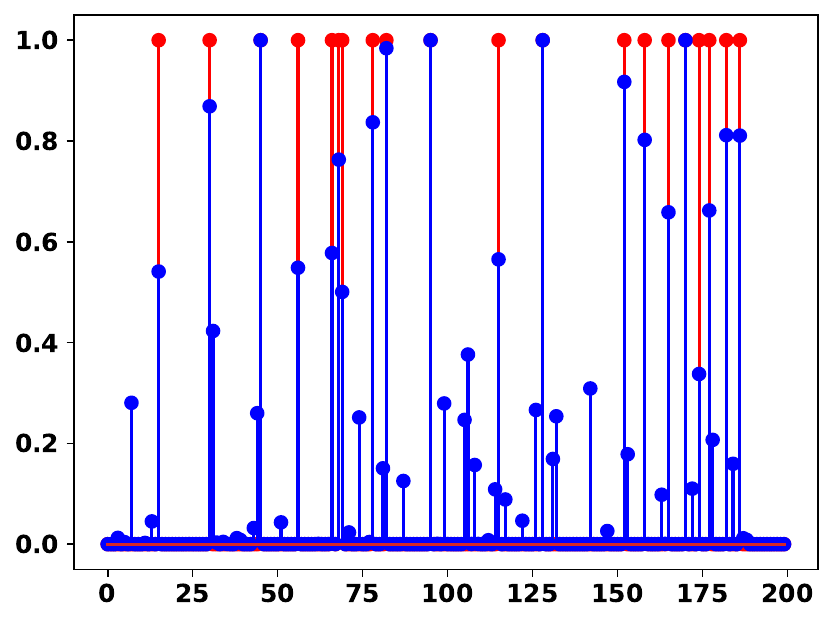}
            &
		    \includegraphics[width=#1\textwidth]{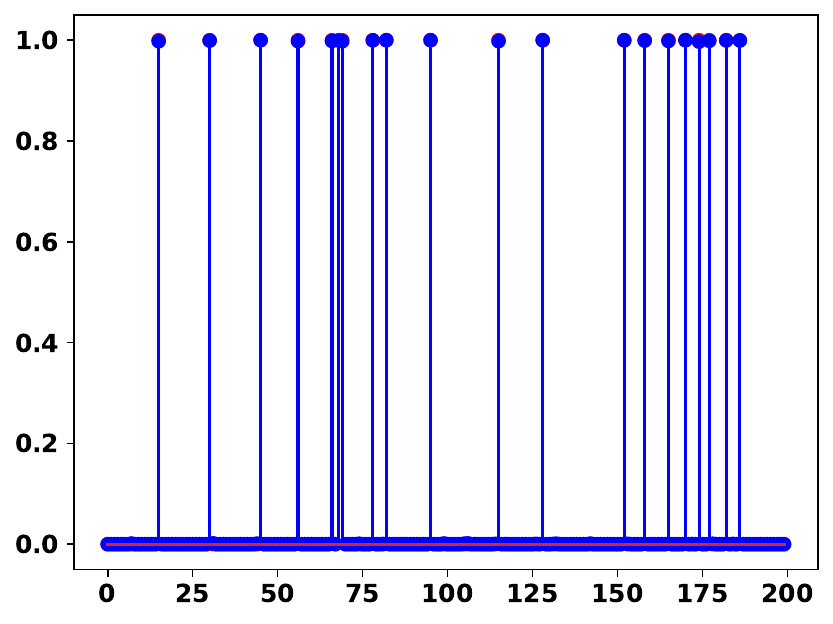}
            \\[0.6em]

			\footnotesize \texttt{FSMART-e} & \footnotesize \texttt{FSMART-g}  & \footnotesize \texttt{RG Armijo} 
            \\[0.2em]
            
		    \includegraphics[width=#1\textwidth]{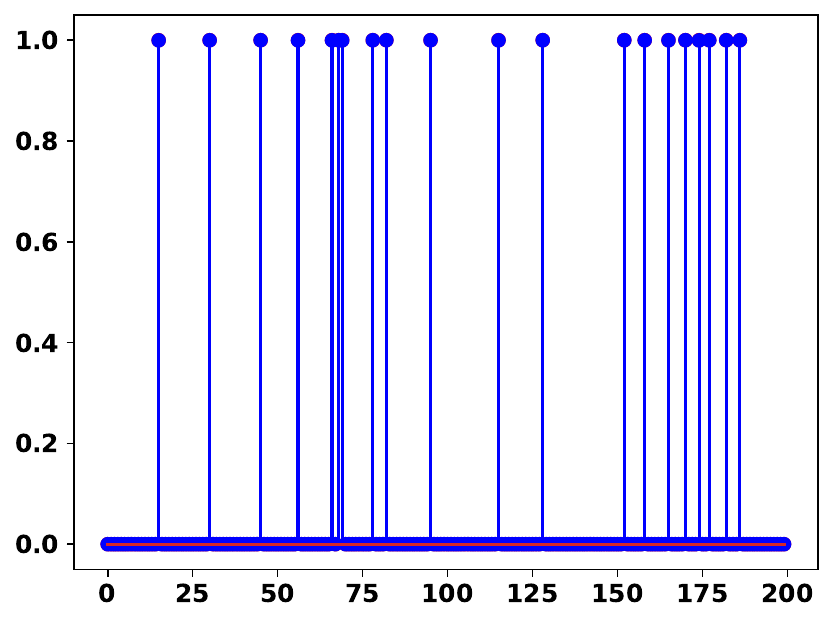}
		    &
		    \includegraphics[width=#1\textwidth]{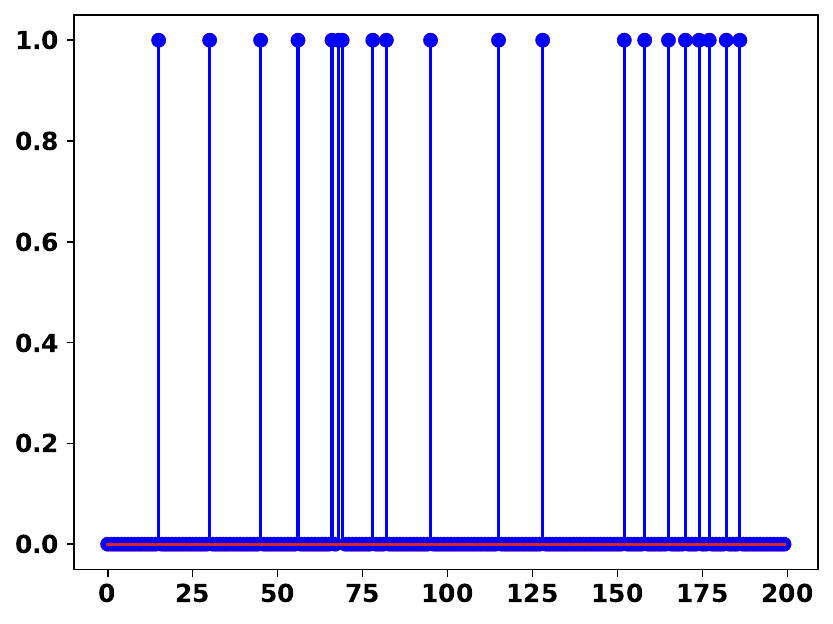}
            &
		    \includegraphics[width=#1\textwidth]{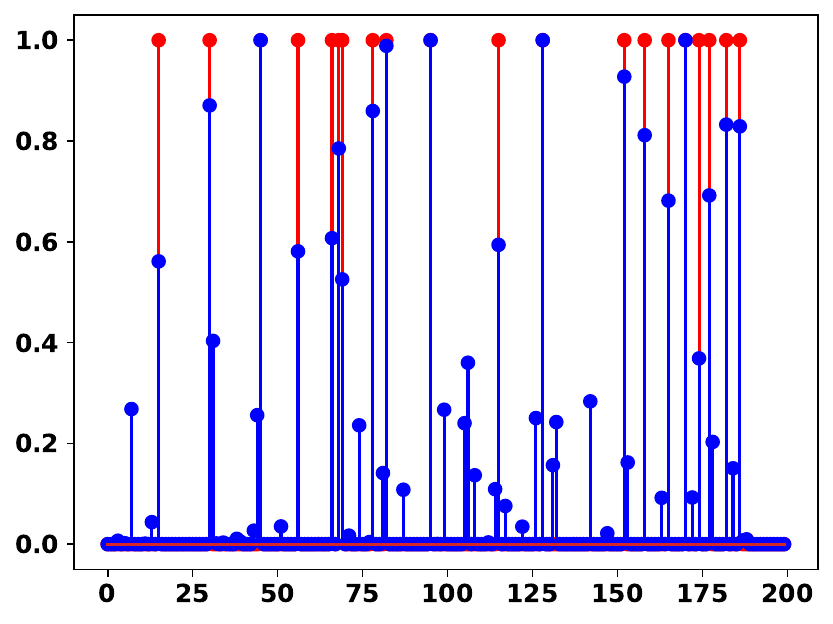}
            \\[0.6em]

			\footnotesize \texttt{RG BB} & \footnotesize \texttt{RG HZ}  & \footnotesize \texttt{CG} 
            \\[0.2em]
            
		    \includegraphics[width=#1\textwidth]{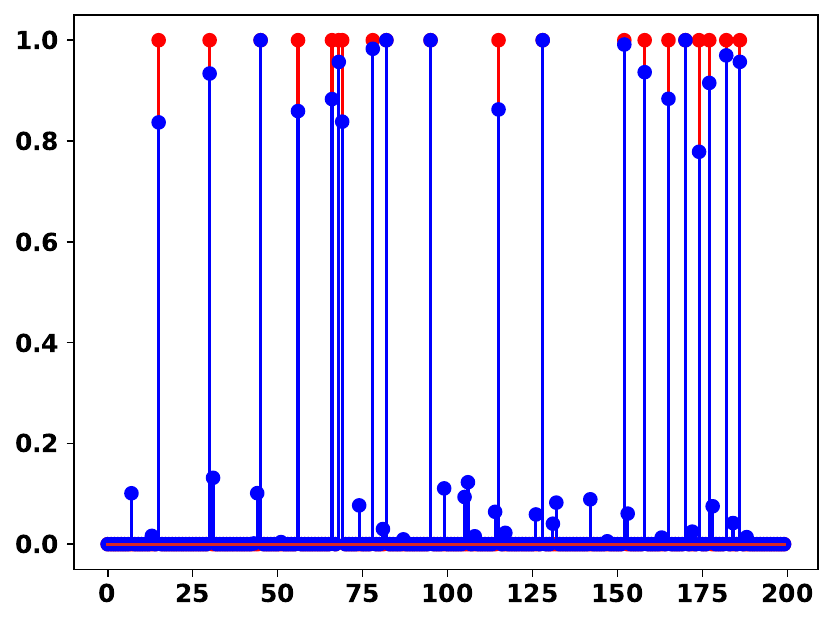}
		    &
		    \includegraphics[width=#1\textwidth]{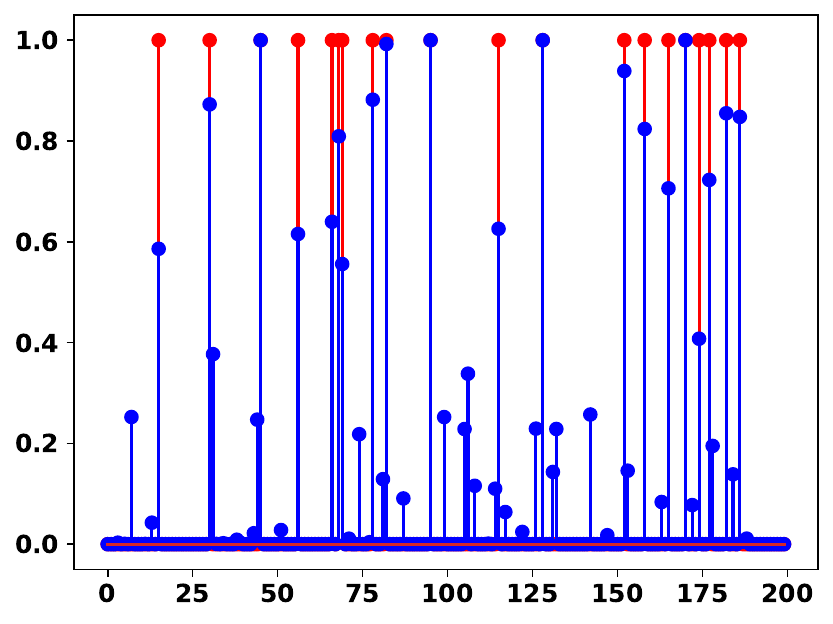}
            &
		    \includegraphics[width=#1\textwidth]{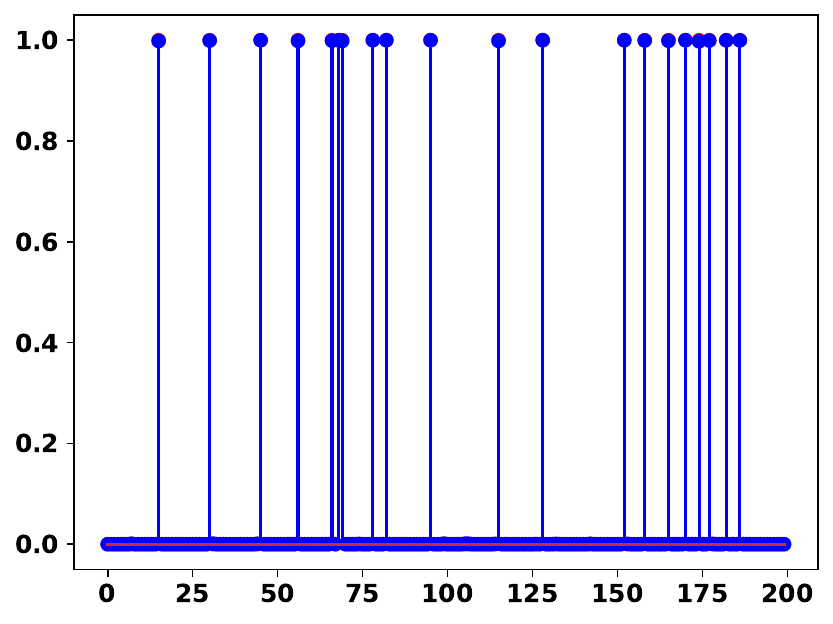}
		\end{tabular}
	\end{centering}
}

\newcommand{\showExperimentAMVOGamma}[2]{
\begin{centering}
		\begin{tabular}{c@{\hskip 2.4em}c}
  			\footnotesize \texttt{A-posteriori certificates} & \footnotesize \texttt{Average Matrix Vector Operations} 
            \\[-0.6em]
            
		    \includegraphics[valign=T, width=#1\textwidth]{Figures/Experiments/Experiment#2_CSI_Gamma.pdf}
		    &
            \includegraphics[valign=T, width=#1\textwidth]{Figures/Experiments/Experiment#2_CSI_averageMatrixVectorOperations.pdf}
			
		\end{tabular}
	\end{centering}
}

\newcommand{\showTomoImages}[1]{
\begin{centering}
		\begin{tabular}{c@{\hskip 0.4em}c@{\hskip 0.6em}c@{\hskip 0.6em}c}
			& \footnotesize \texttt{Shepp-Logan} & \footnotesize \texttt{Bone} & \footnotesize \texttt{Vessel}
            \\[0.2em]

            \rotatebox{90}{\quad \qquad \footnotesize \textbf{Original}}
            &
			\includegraphics[width=#1\textwidth]{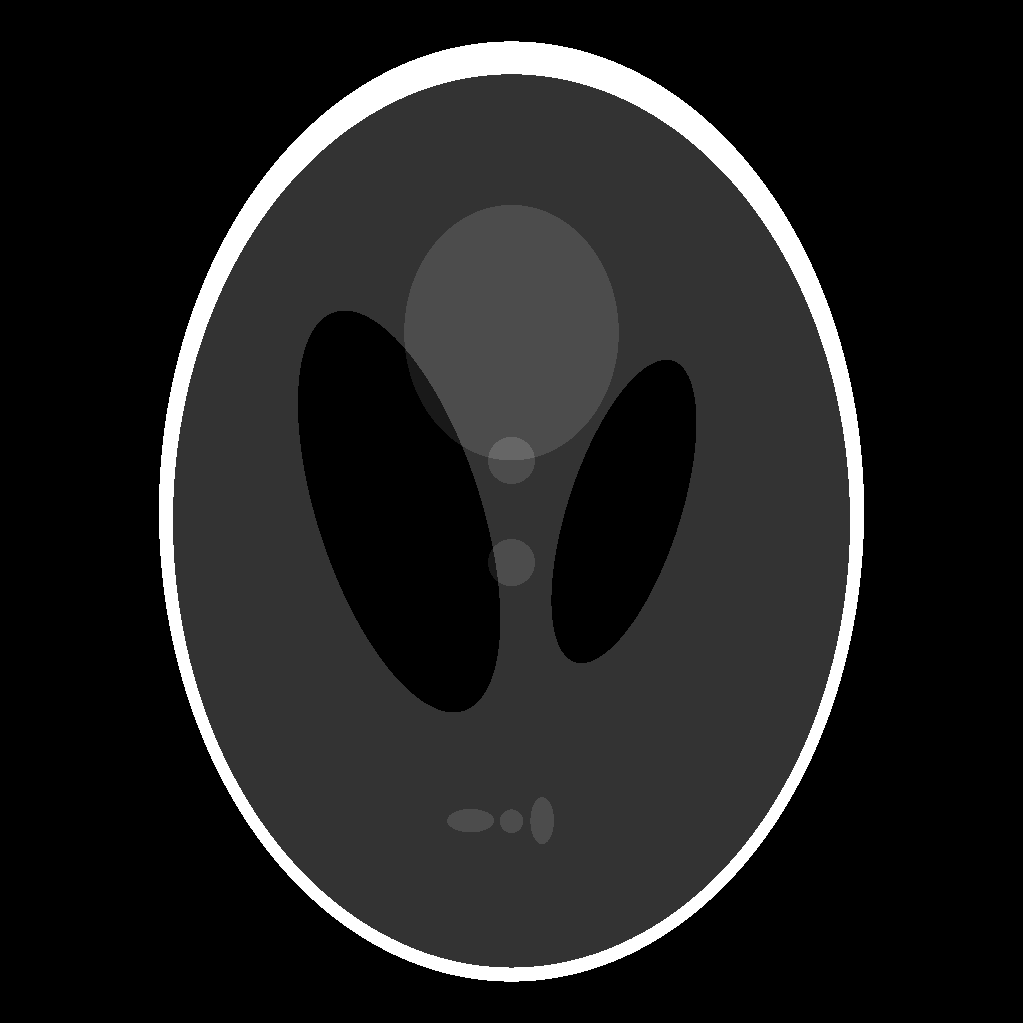}
			&
			\includegraphics[width=#1\textwidth]{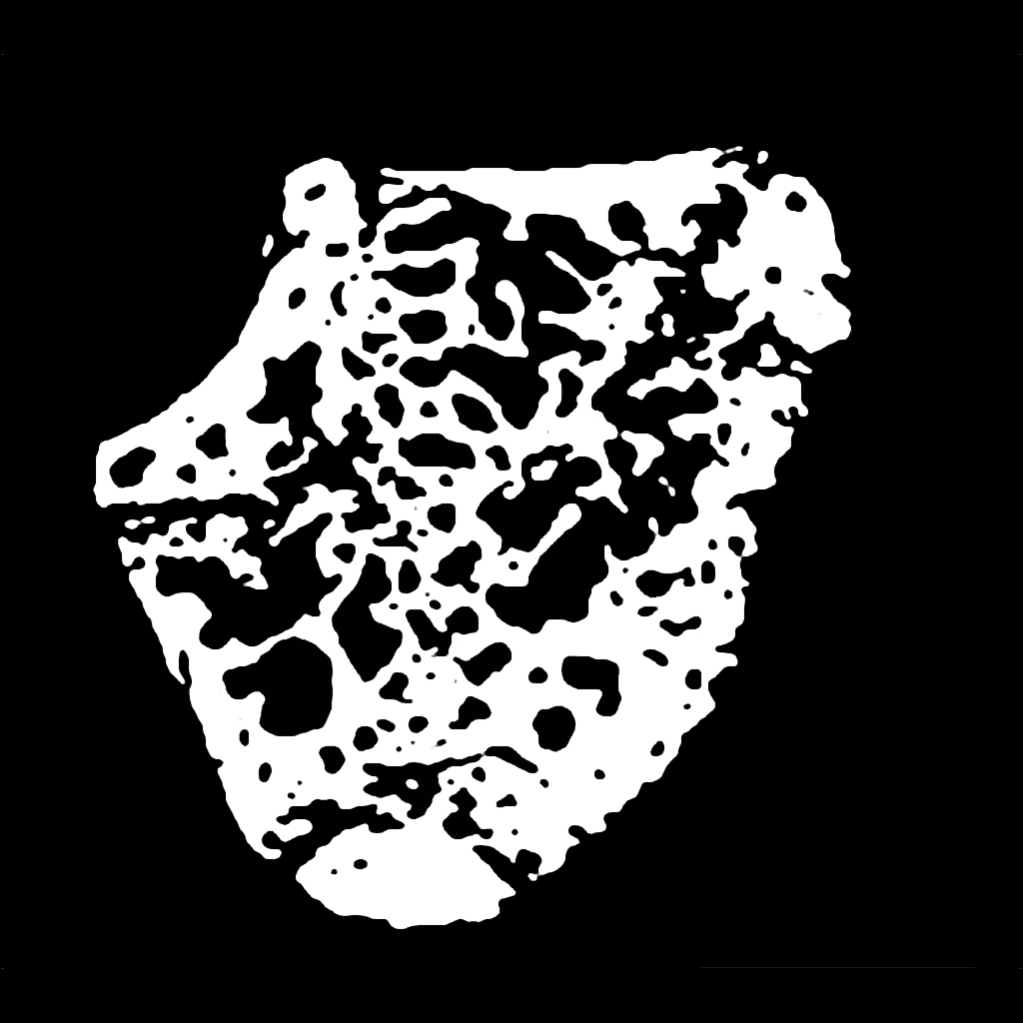}
			&
			\includegraphics[width=#1\textwidth]{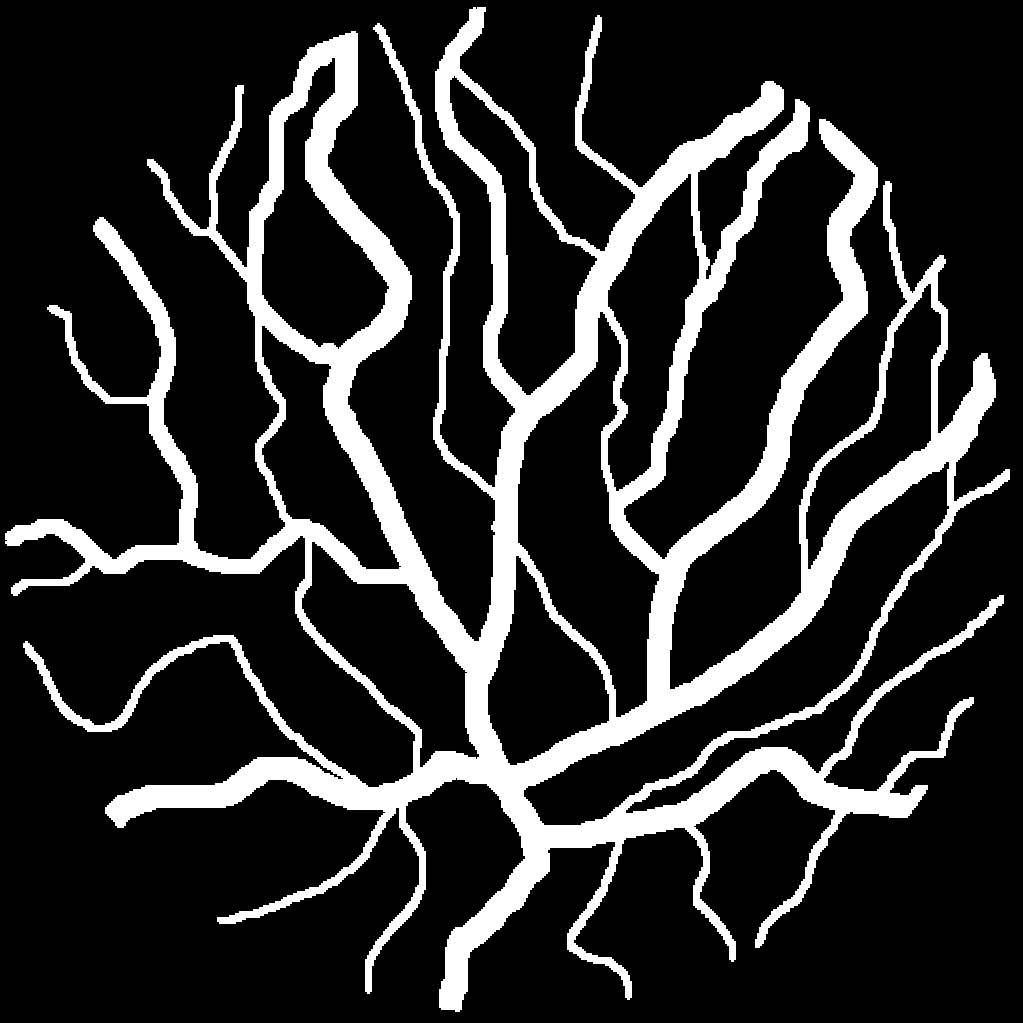}
			\\[0.2em]

            \rotatebox{90}{\quad \qquad \footnotesize \textbf{SMART}}
            &
			\includegraphics[width=#1\textwidth]{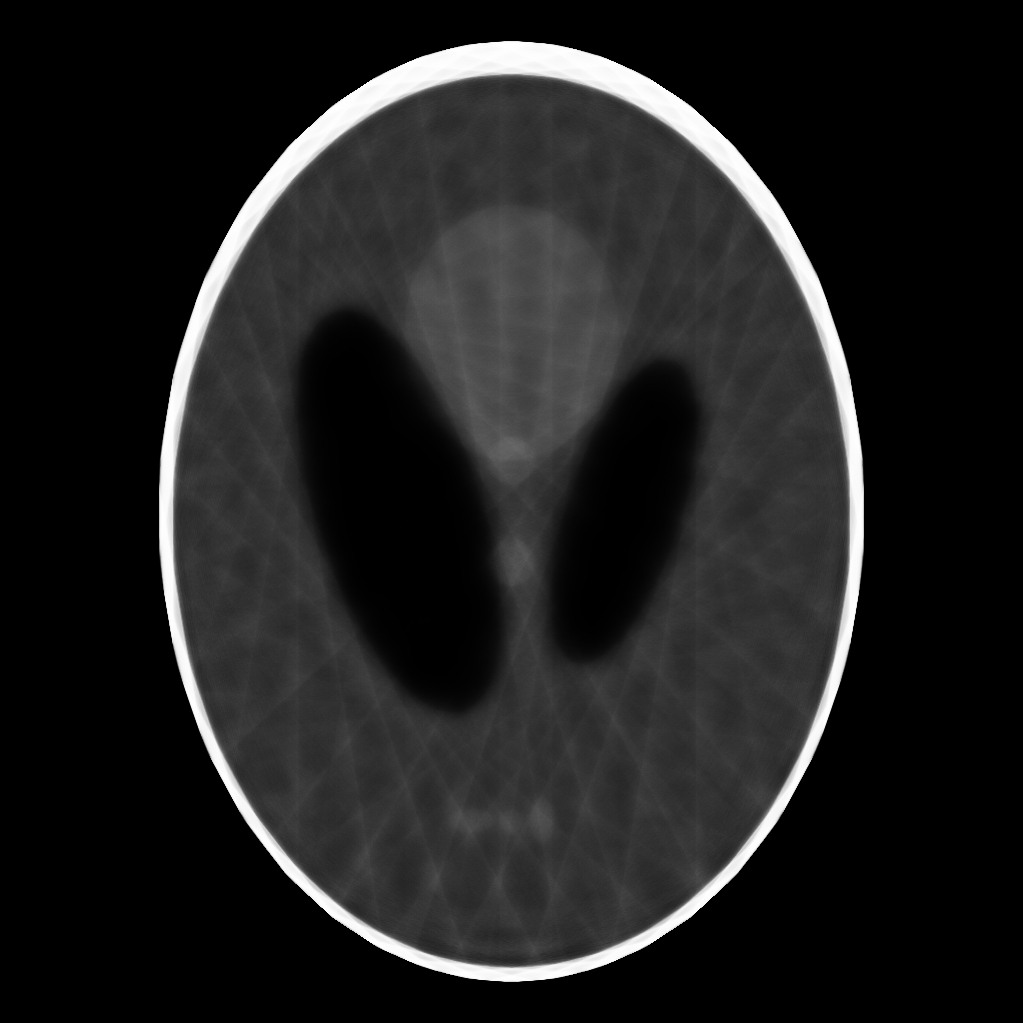}
			&
			\includegraphics[width=#1\textwidth]{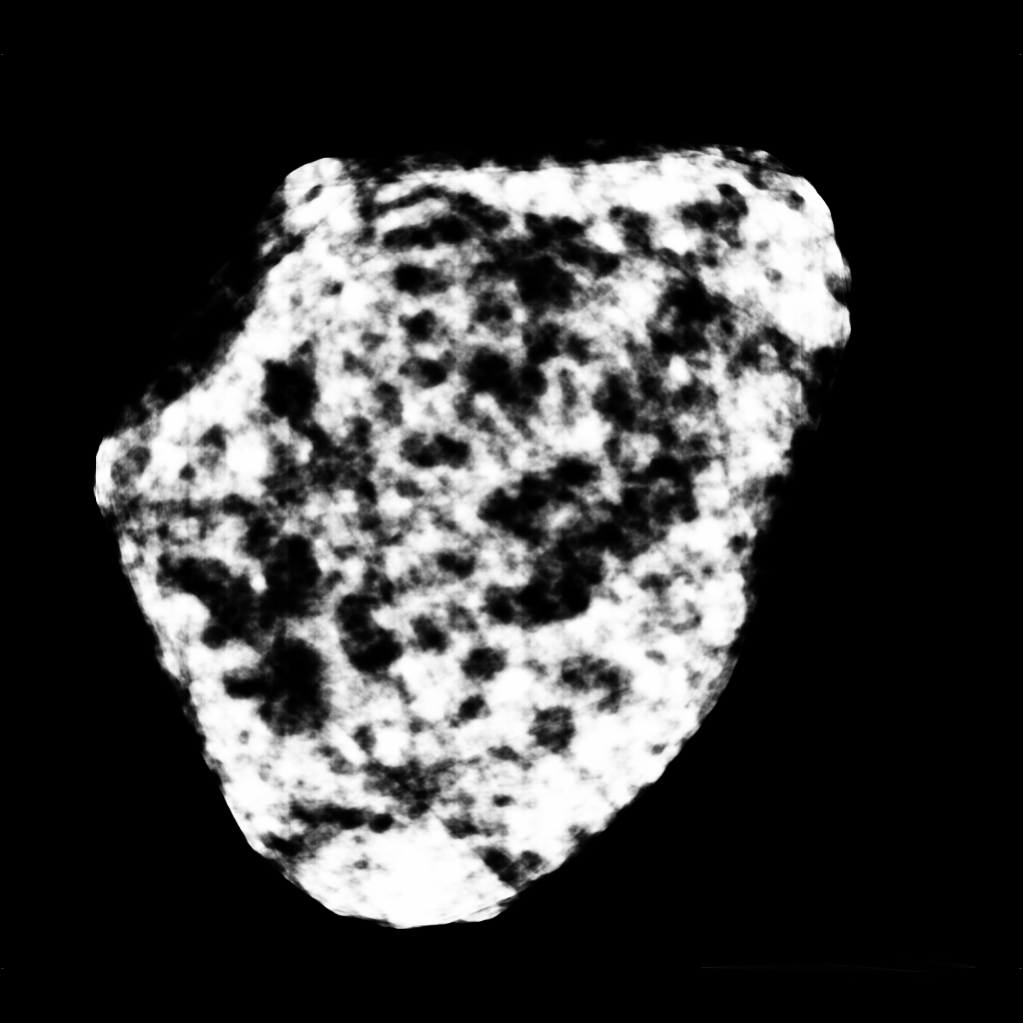}
			&
			\includegraphics[width=#1\textwidth]{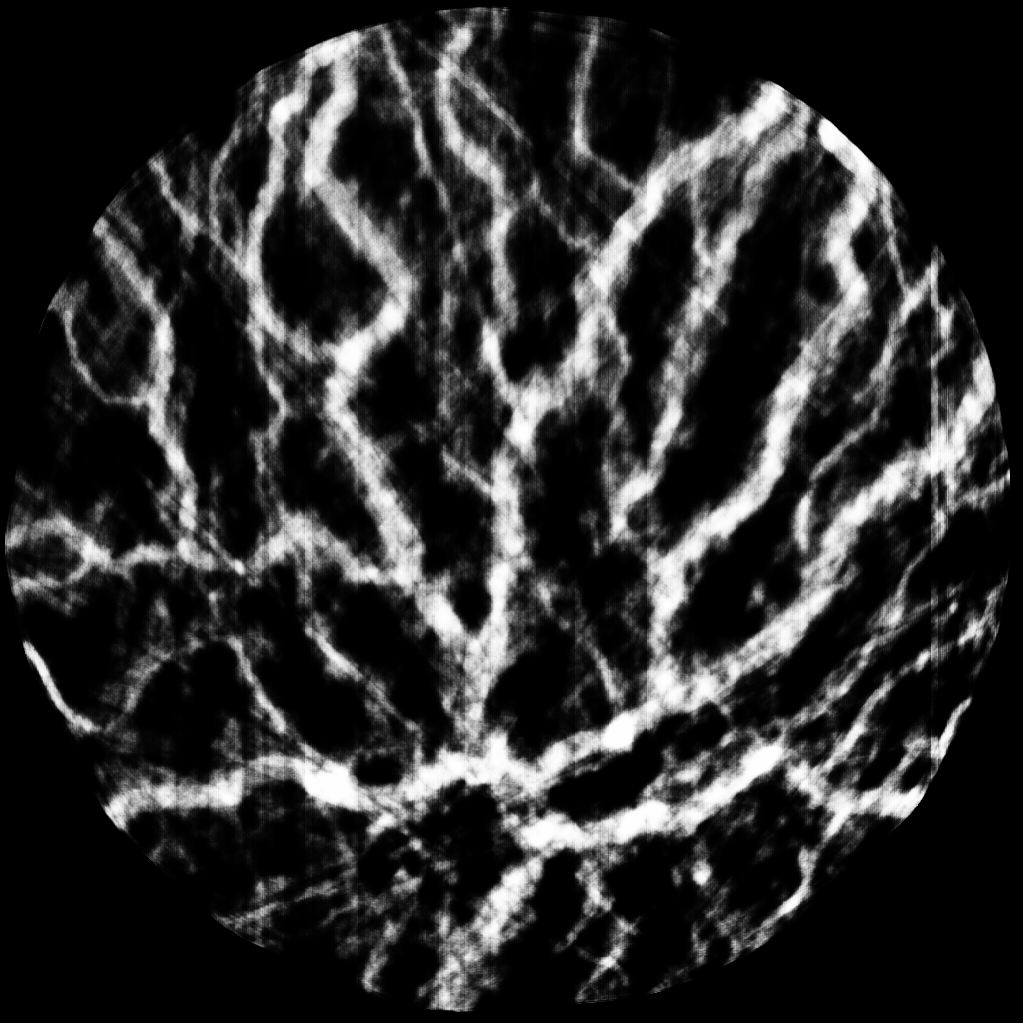}
            \\[0.2em]

            \rotatebox{90}{\quad \qquad \footnotesize \textbf{FSMART}}
            &
			\includegraphics[width=#1\textwidth]{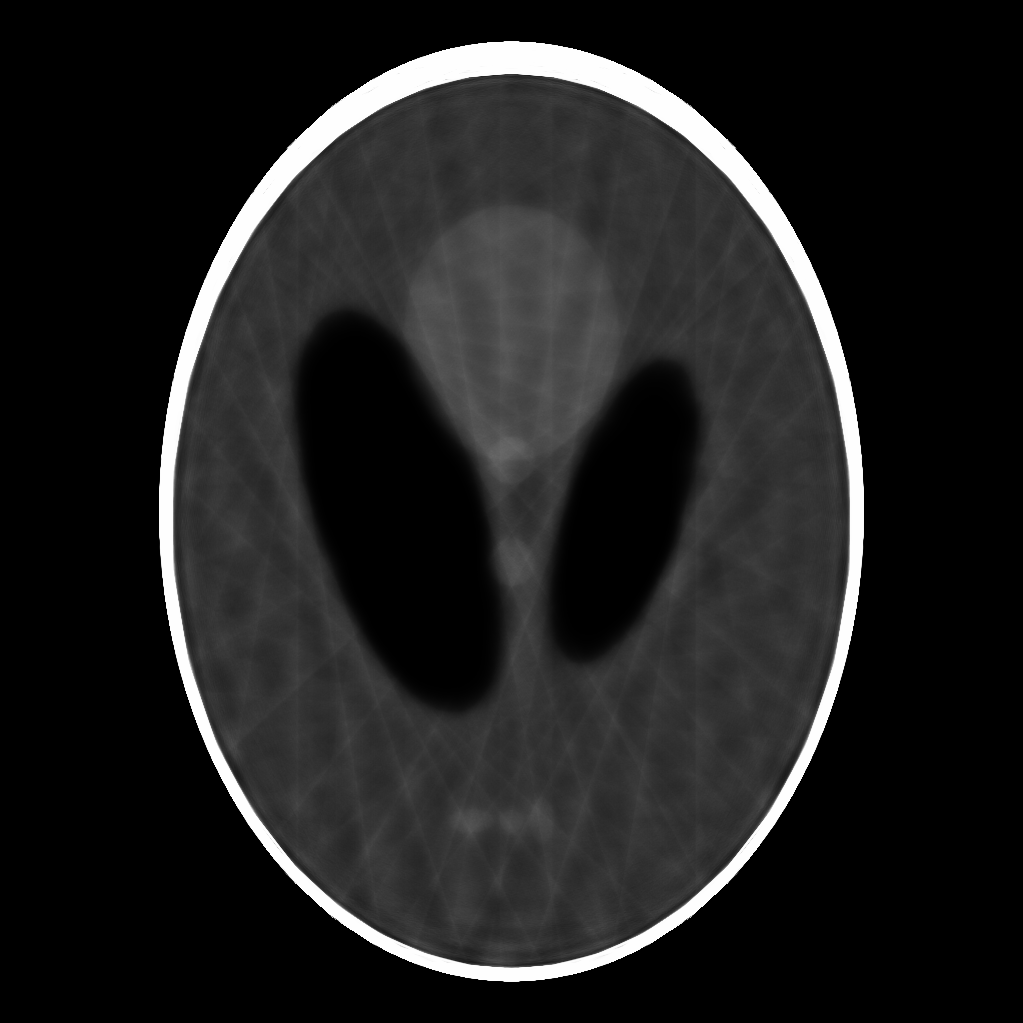}
			&
			\includegraphics[width=#1\textwidth]{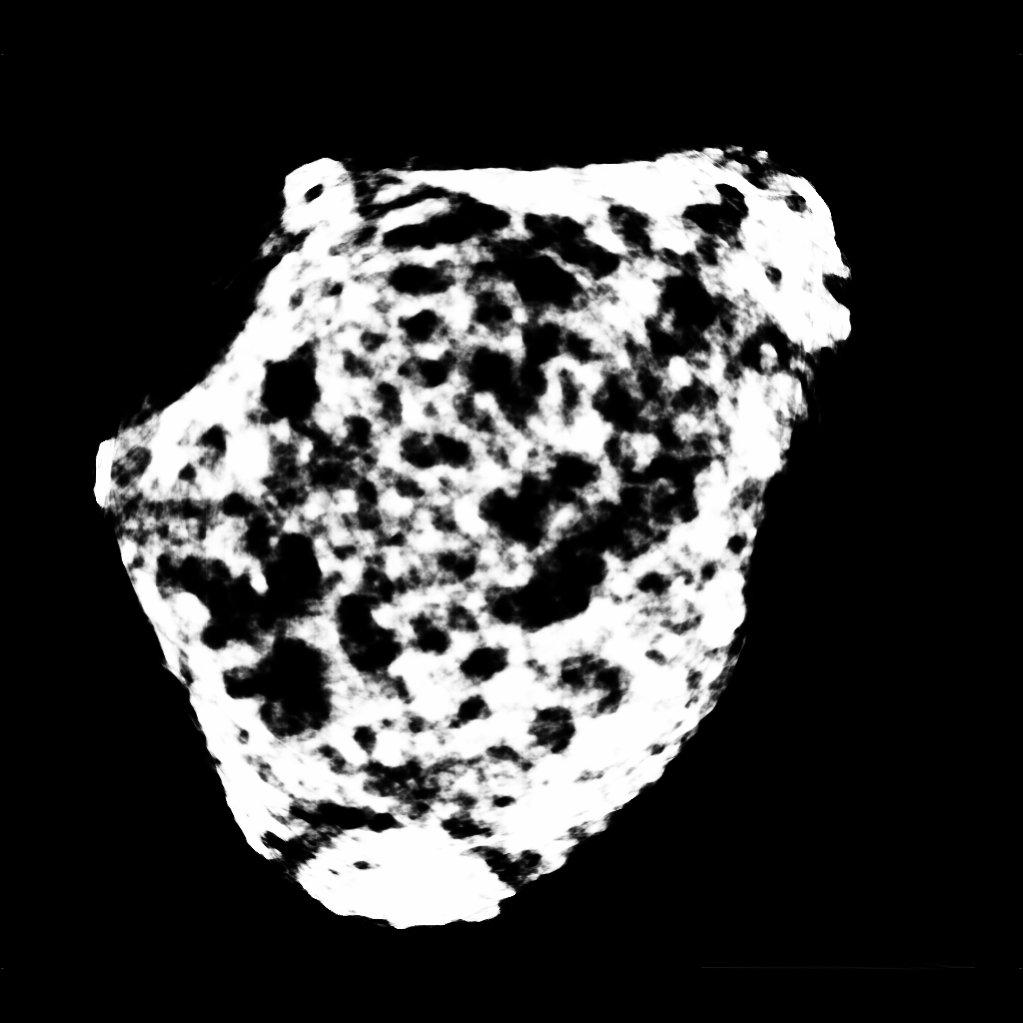}
			&
			\includegraphics[width=#1\textwidth]{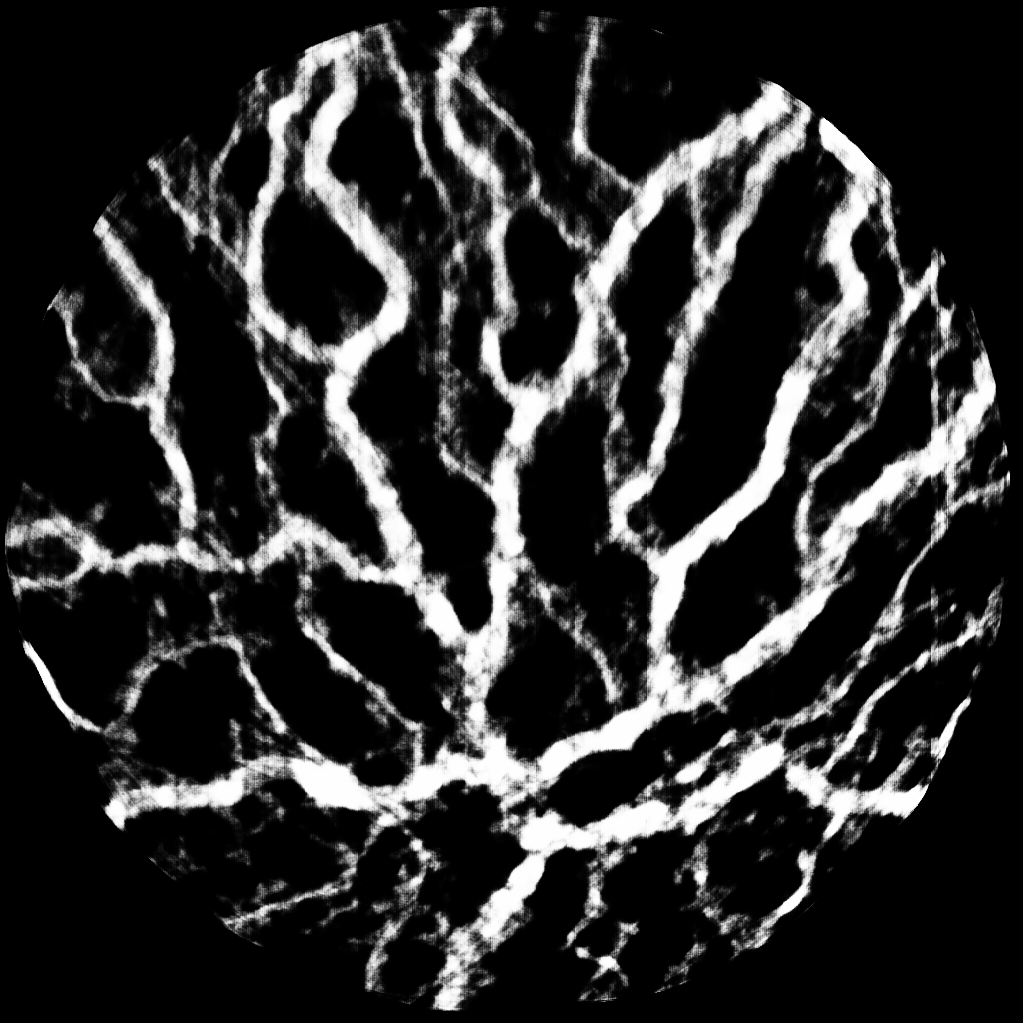}
            \\[0.2em]

            \rotatebox{90}{\quad \qquad \quad \footnotesize \textbf{CG}}
            &
			\includegraphics[width=#1\textwidth]{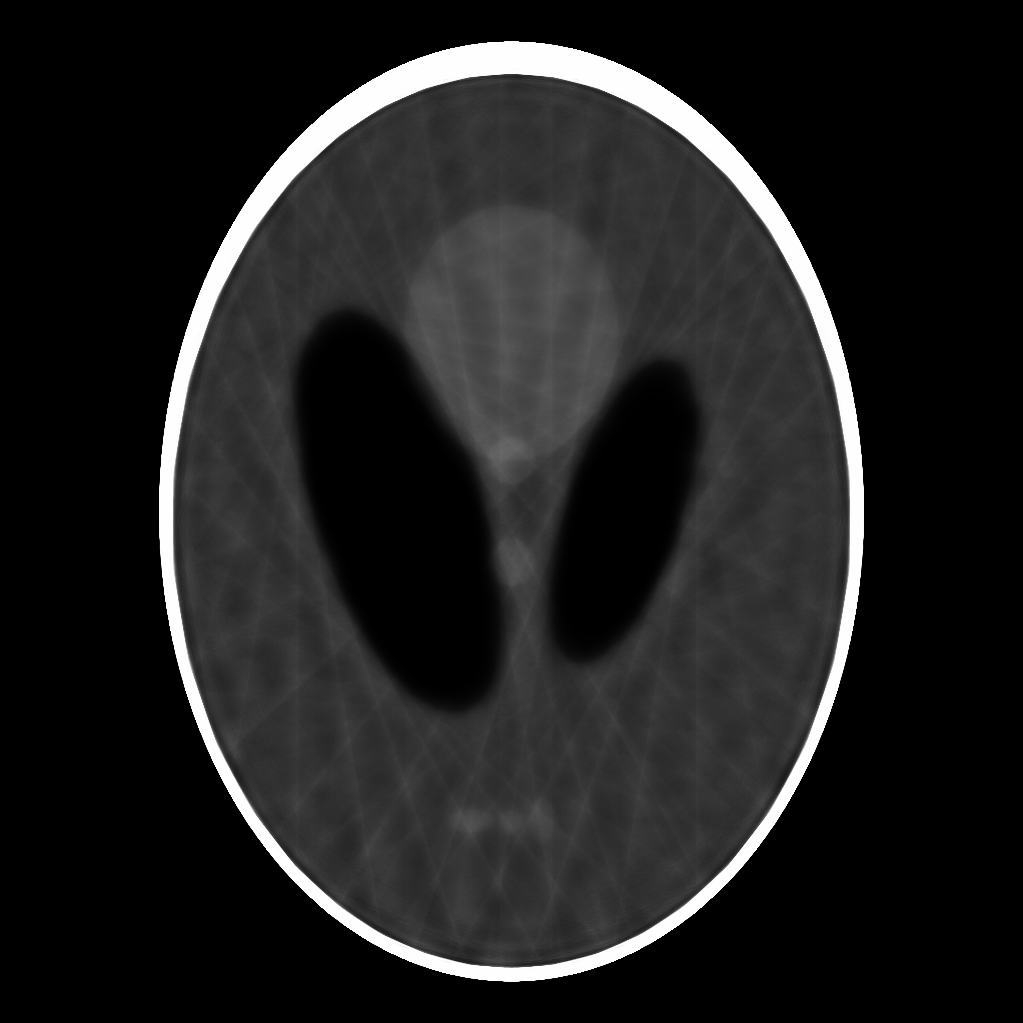}
			&
			\includegraphics[width=#1\textwidth]{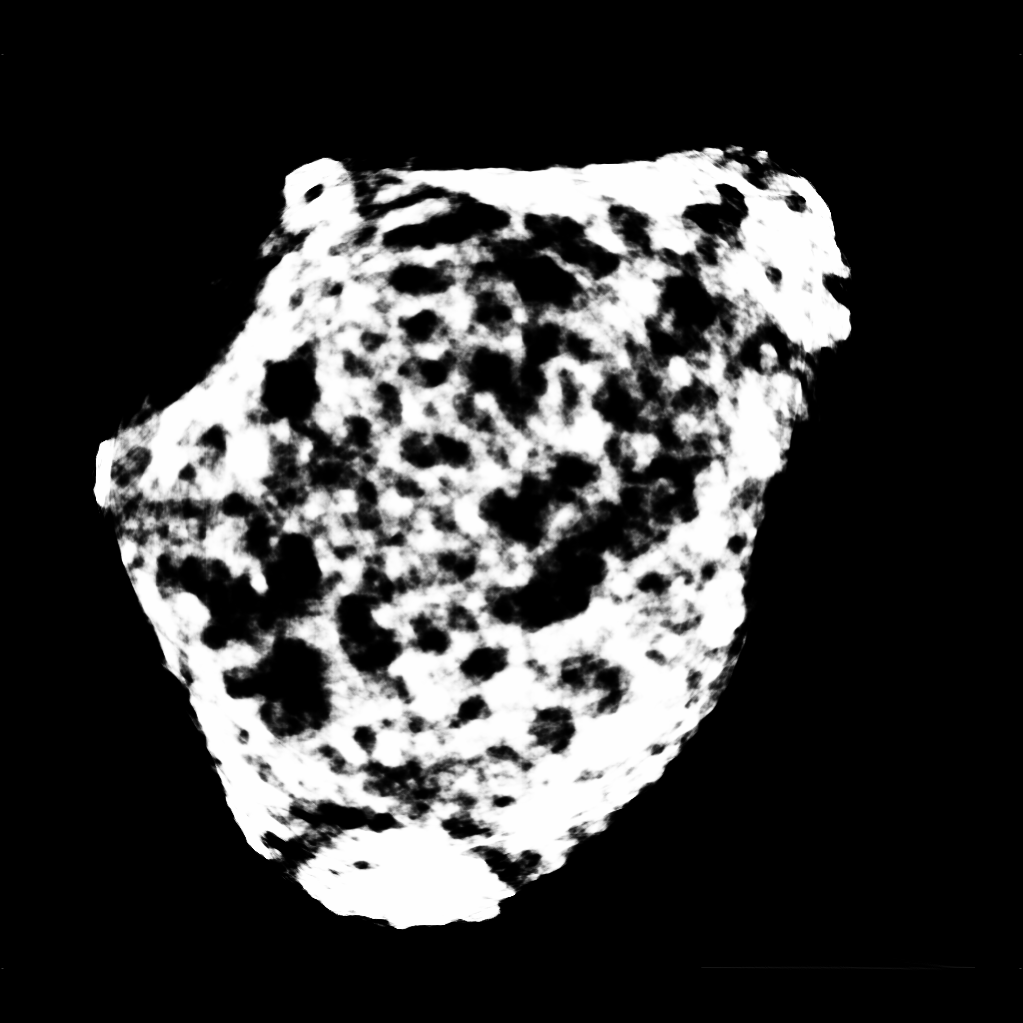}
			&
			\includegraphics[width=#1\textwidth]{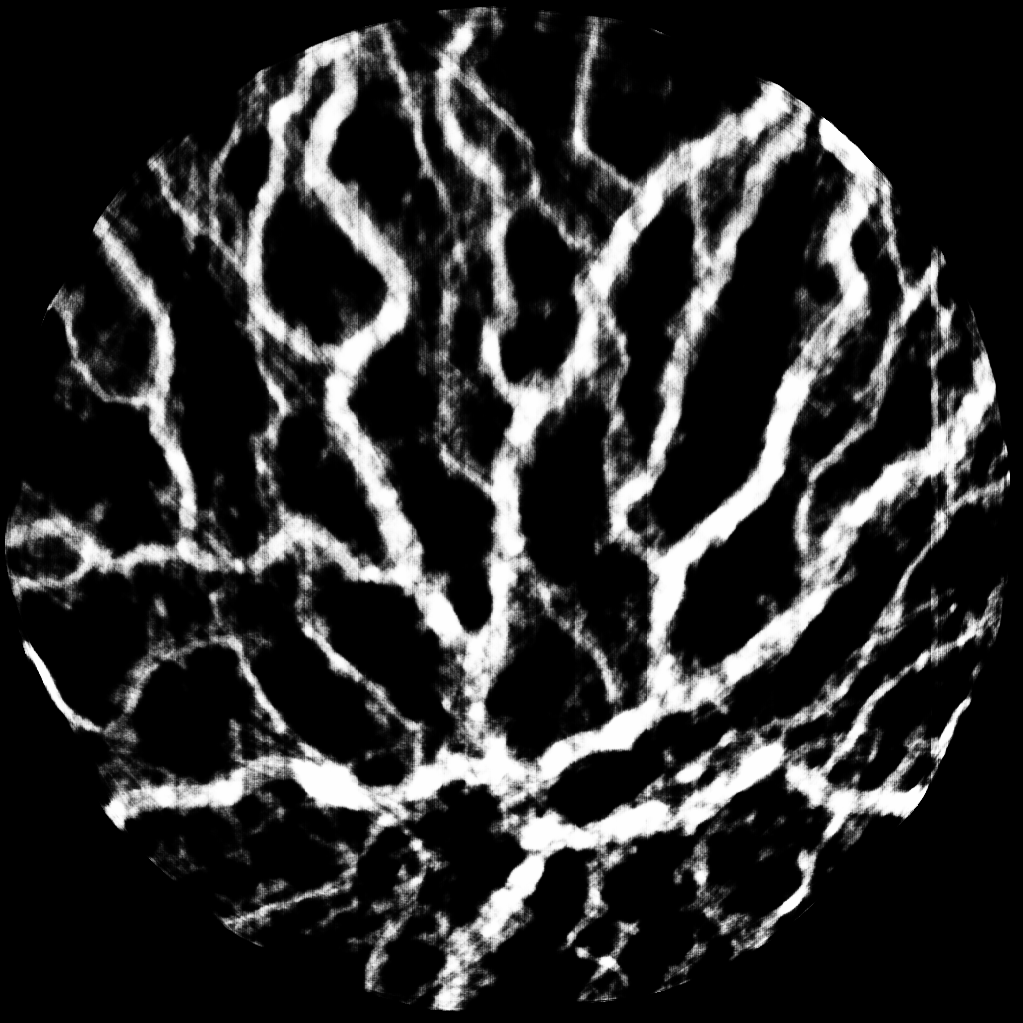}
			
		\end{tabular}
	\end{centering}
}

\newcommand{\showBlurImages}[1]{
\begin{centering}
		\begin{tabular}{c@{\hskip 0.4em}c@{\hskip 0.6em}c@{\hskip 0.6em}c}
			& \footnotesize \texttt{Kitten} & \footnotesize \texttt{Tiger} & \footnotesize \texttt{QR-Code}
            \\[0.2em]
            
            \rotatebox{90}{\quad \qquad \footnotesize \textbf{Original}}
            &
			\includegraphics[width=#1\textwidth]{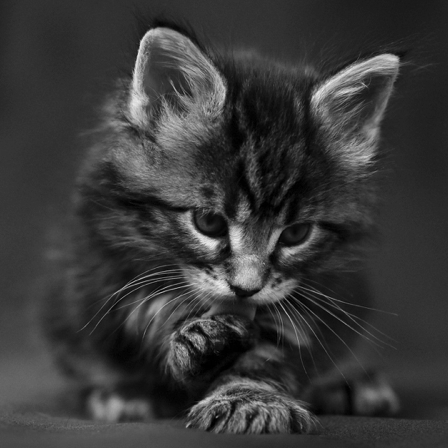}
			&
			\includegraphics[width=#1\textwidth]{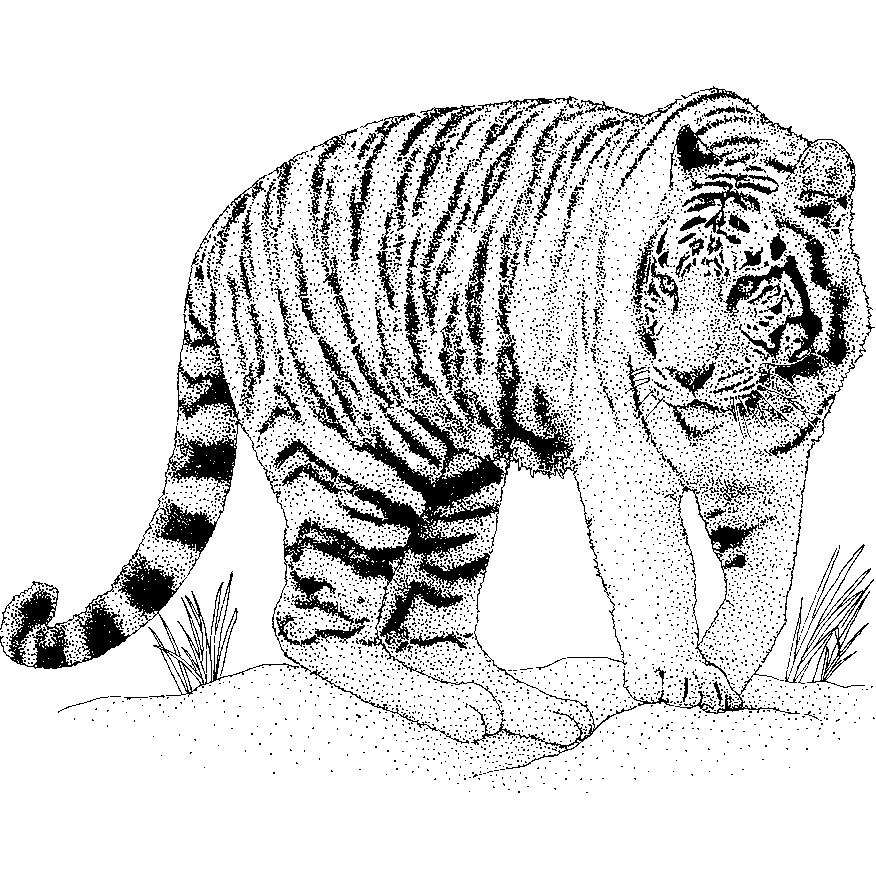}
			&
			\includegraphics[width=#1\textwidth]{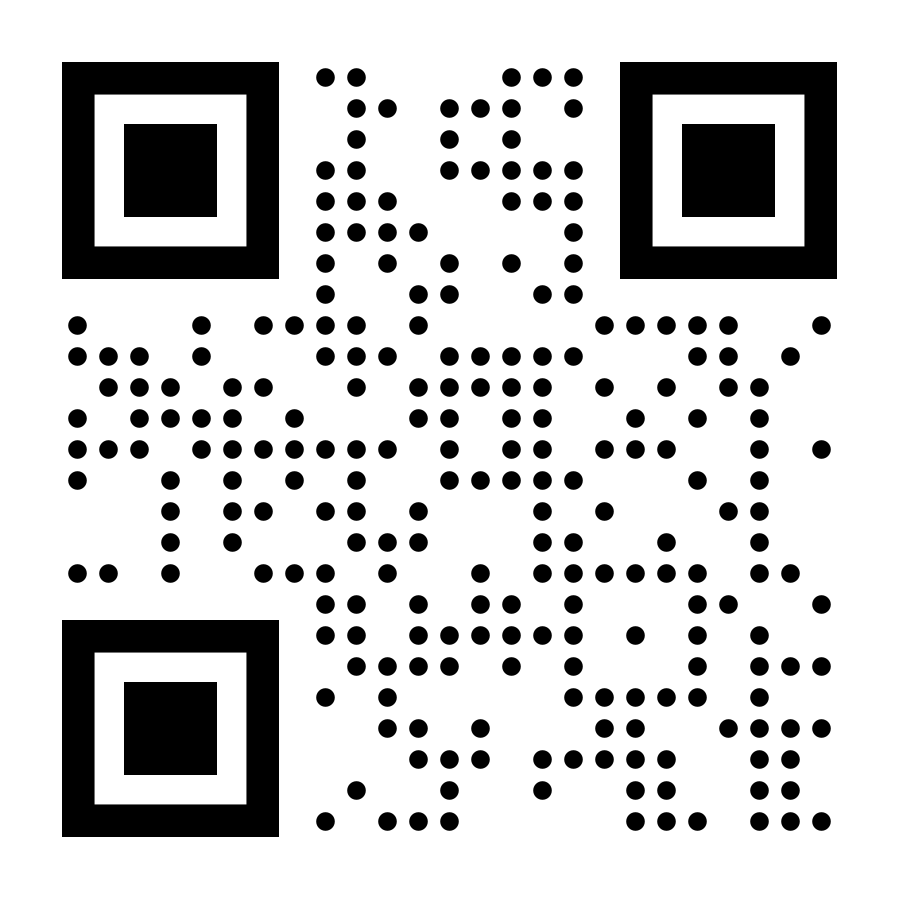}
            \\[0.2em]

            \rotatebox{90}{\quad \qquad \footnotesize \textbf{Blurred}}
            &
			\includegraphics[width=#1\textwidth]{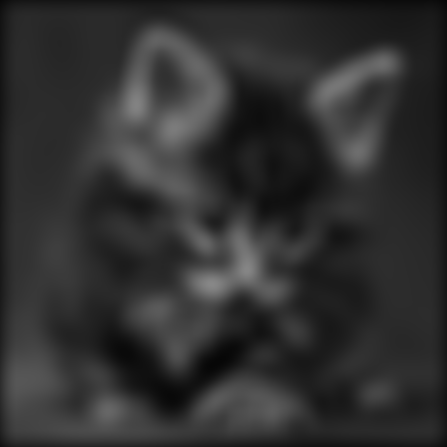}
			&
			\includegraphics[width=#1\textwidth]{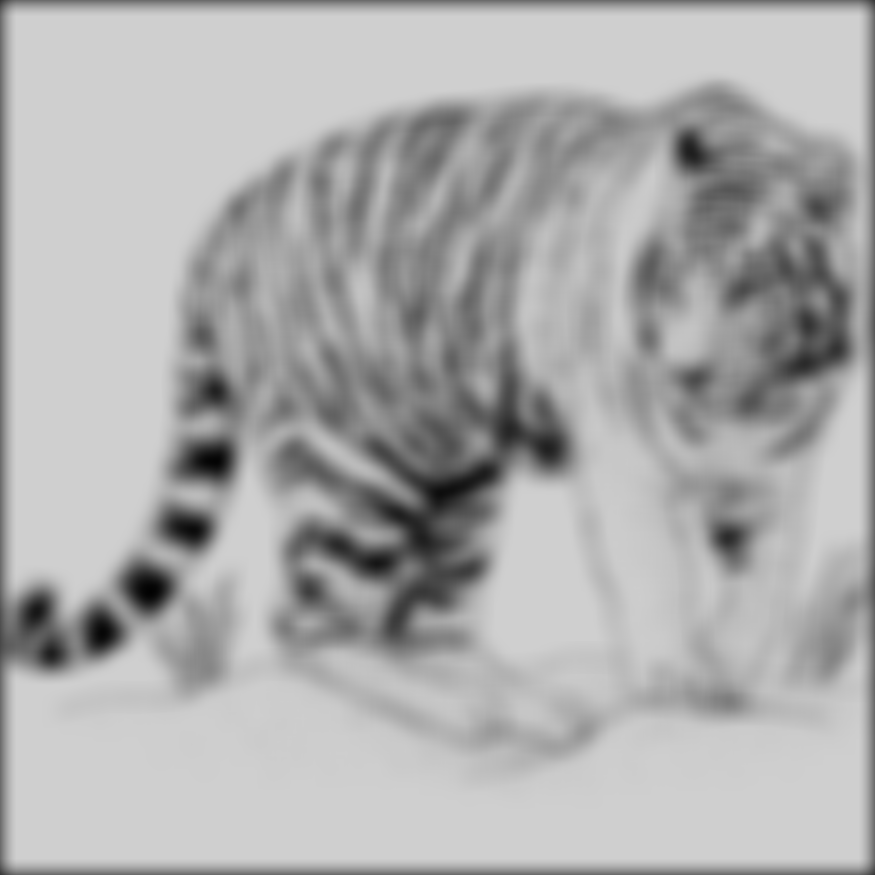}
			&
			\includegraphics[width=#1\textwidth]{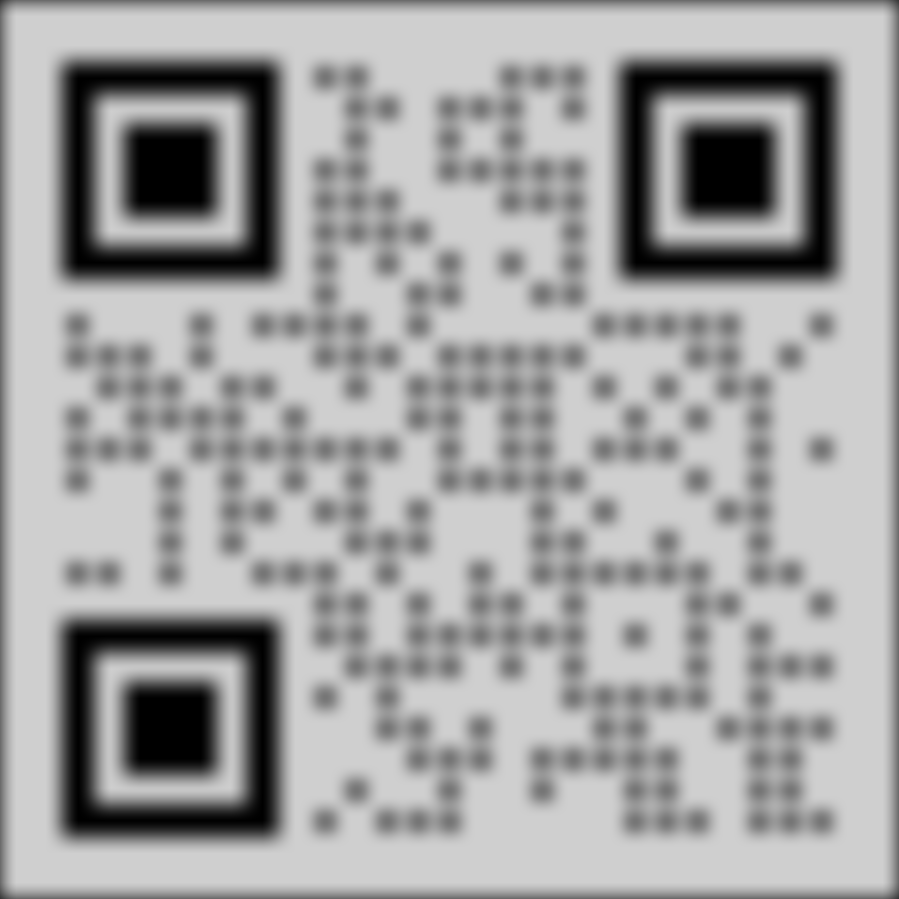}
            \\[0.2em]

            \rotatebox{90}{\quad \qquad \footnotesize \textbf{SMART}}
            &
			\includegraphics[width=#1\textwidth]{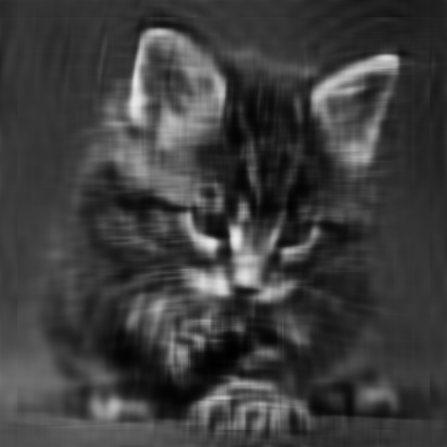}
			&
			\includegraphics[width=#1\textwidth]{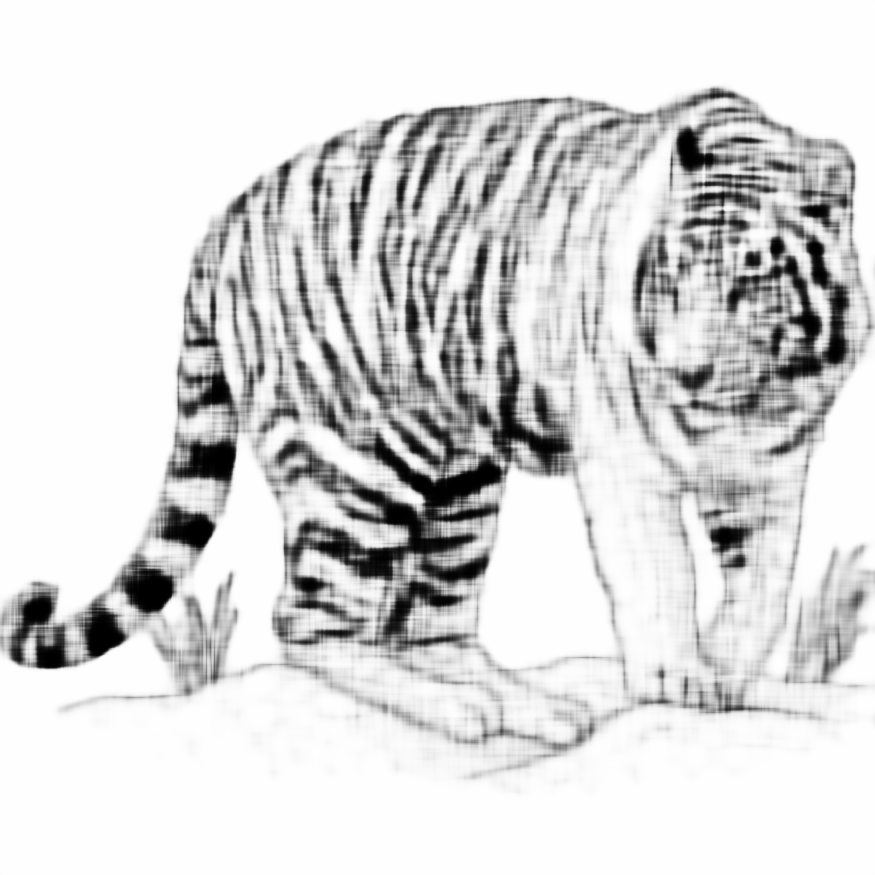}
			&
			\includegraphics[width=#1\textwidth]{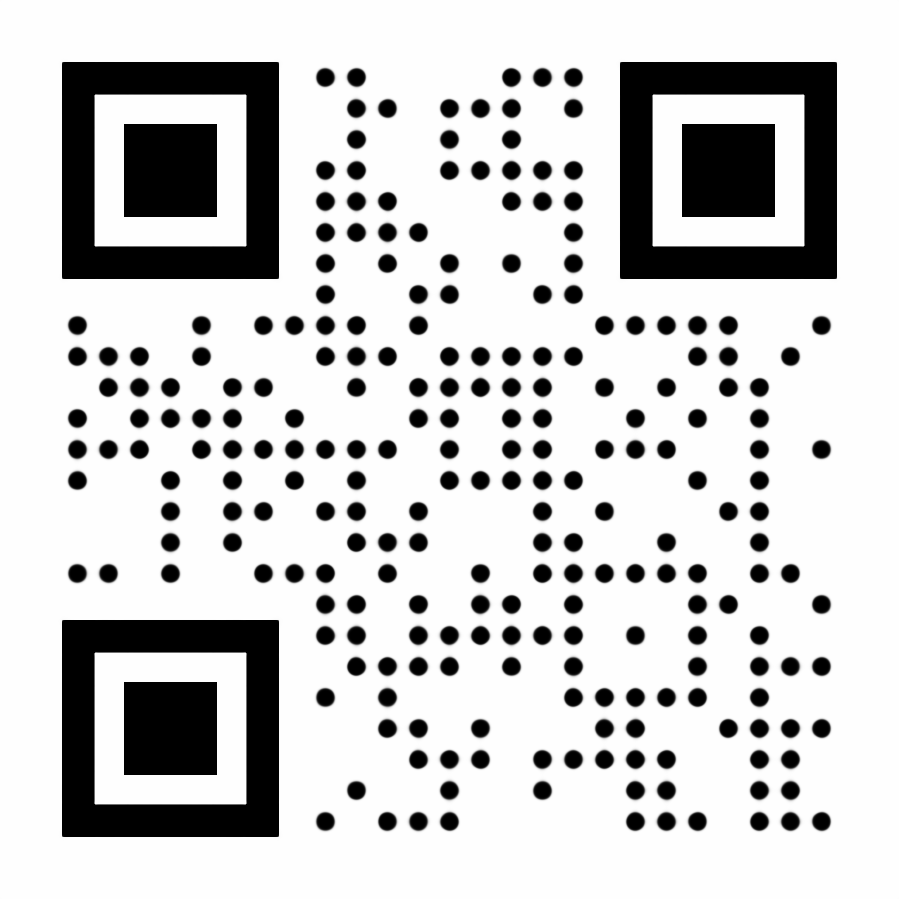}
            \\[0.2em]

            \rotatebox{90}{\quad \qquad \footnotesize \textbf{FSMART}}
            &
			\includegraphics[width=#1\textwidth]{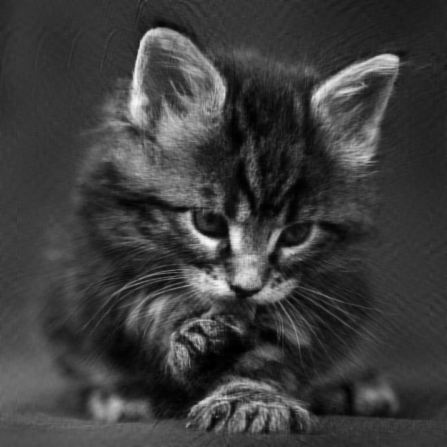}
			&
			\includegraphics[width=#1\textwidth]{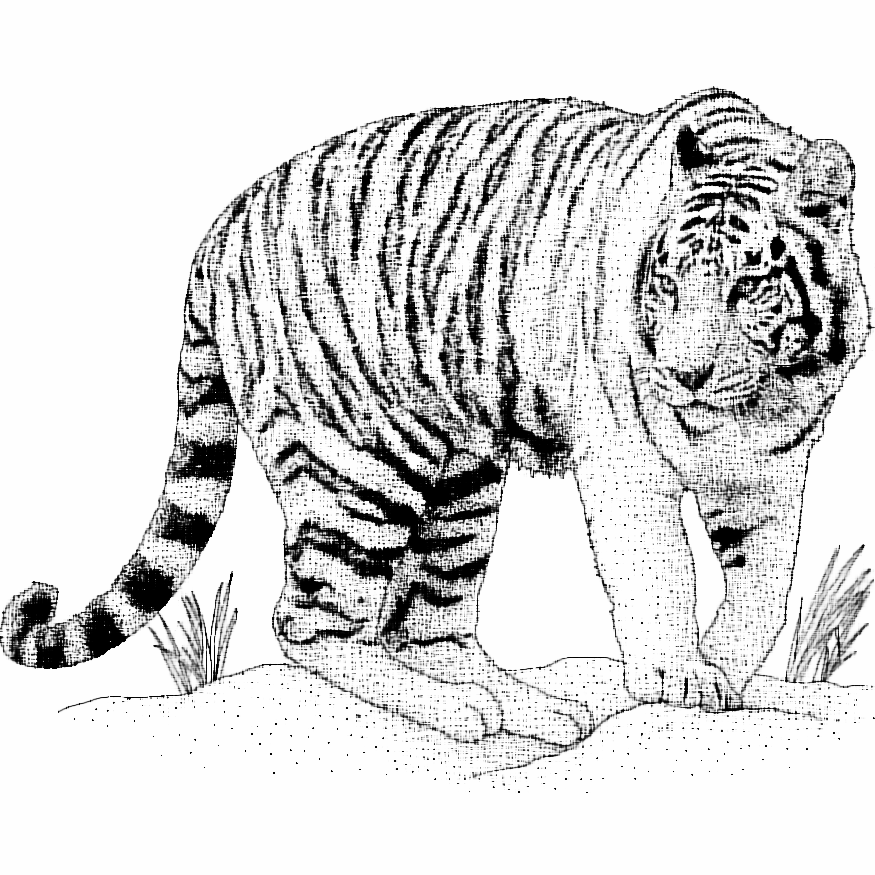}
			&
			\includegraphics[width=#1\textwidth]{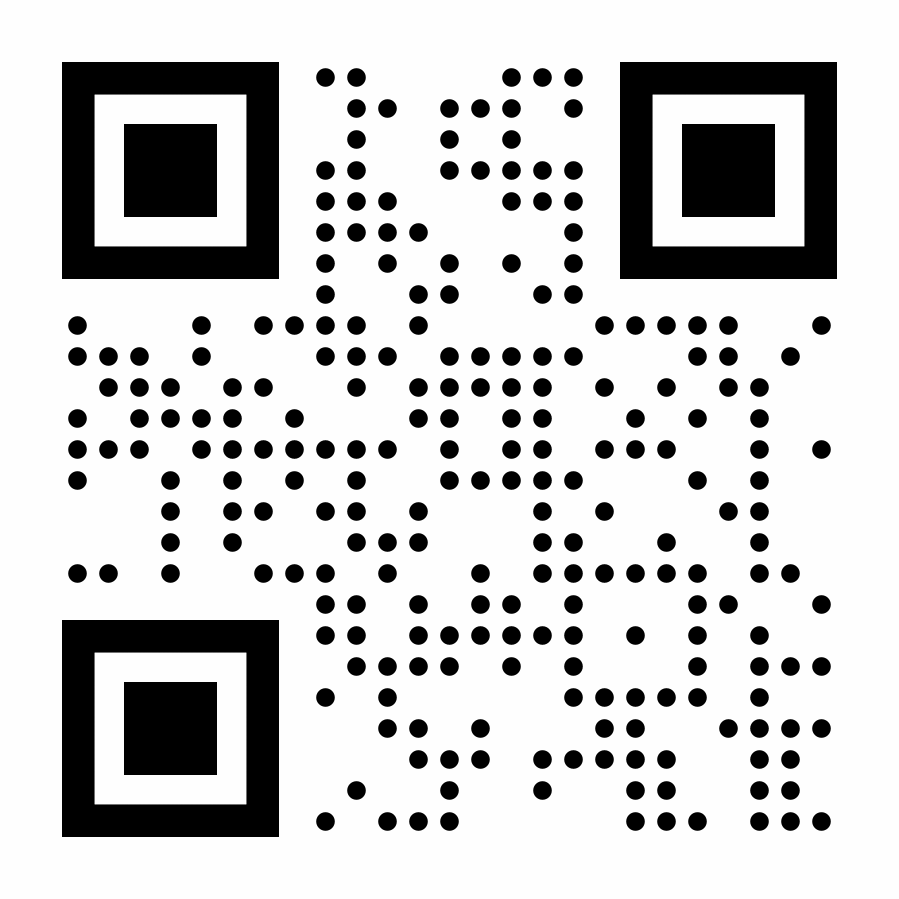}
            \\[0.2em]

            \rotatebox{90}{\quad \qquad \quad \footnotesize \textbf{CG}}
            &
			\includegraphics[width=#1\textwidth]{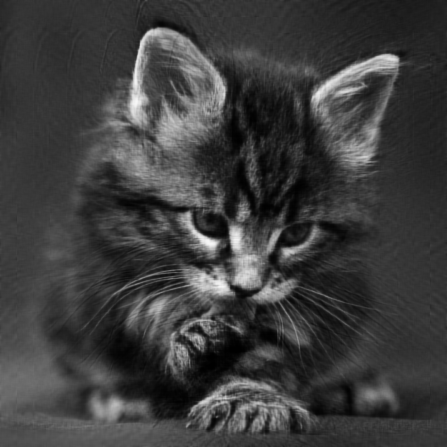}
			&
			\includegraphics[width=#1\textwidth]{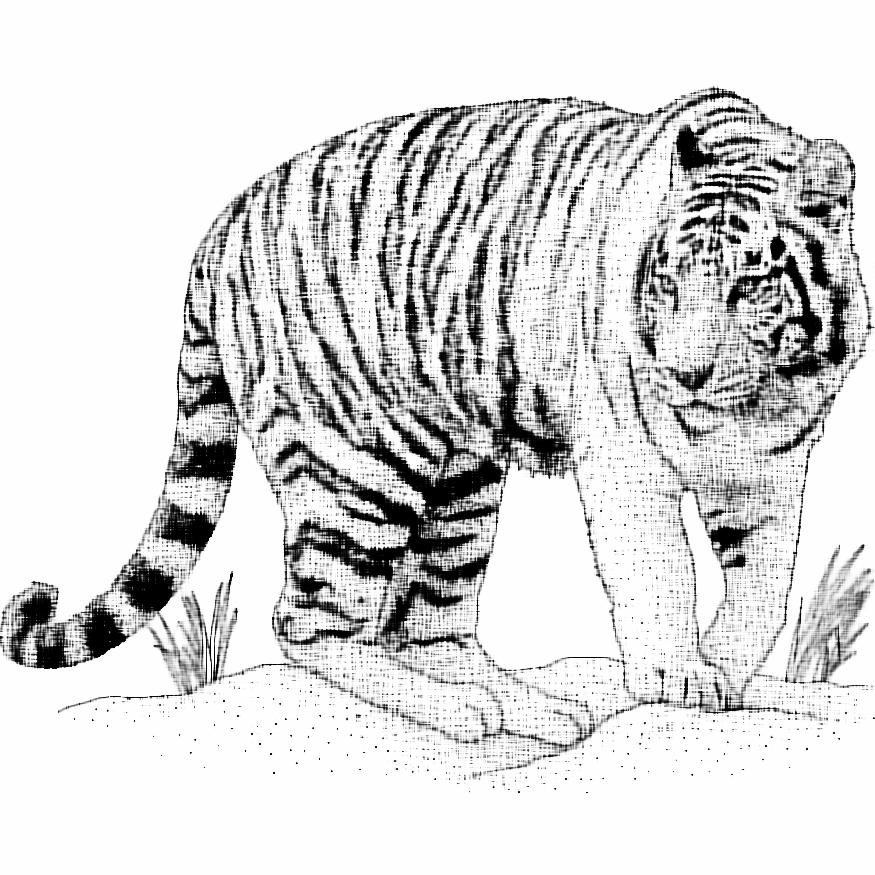}
			&
			\includegraphics[width=#1\textwidth]{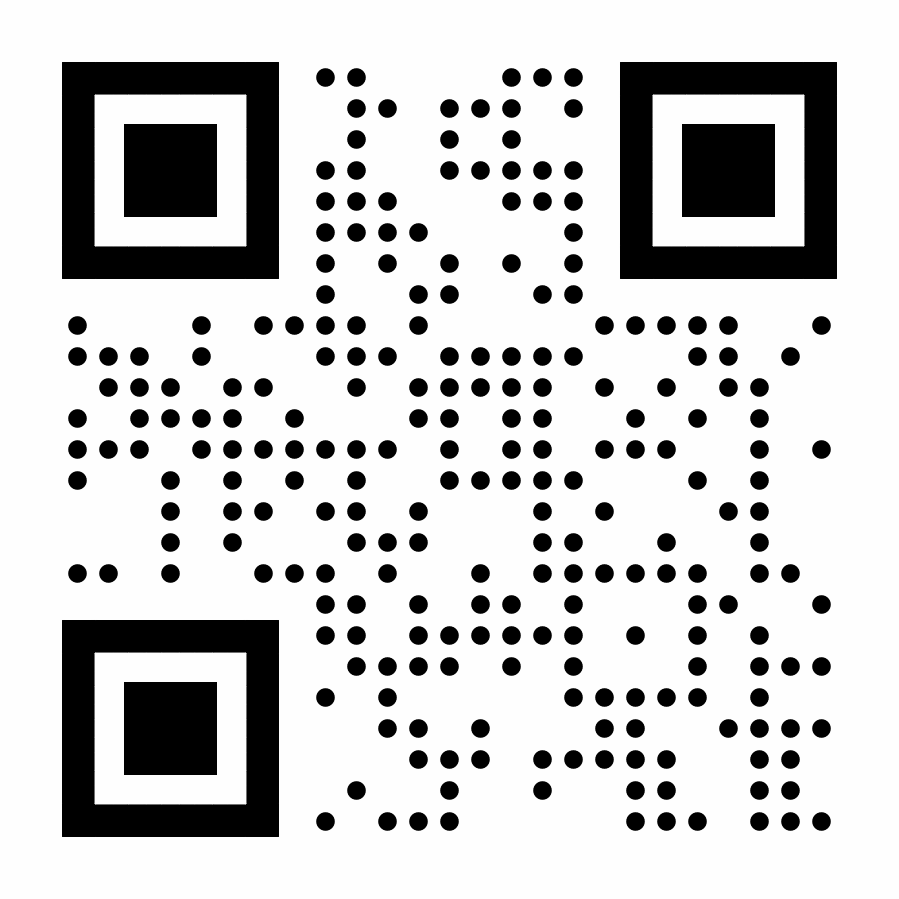}
		\end{tabular}
	\end{centering}
}

Due to the variety of algorithms and multiple manifolds, we opt to focus solely on the Bernoulli manifold, detailed in Section~\ref{sec:Bernoulli-manifold}. Our study includes an extensive comparison between SMART and the recent accelerated Bregman proximal gradient methods discussed in \cite{Hanzely:2021vc}. Our adaptation of these methods, as explained in Section \ref{sec:ConvexAcceleration}, 
result in several distinct variations of SMART, namely FSMART, FSMART-e, and FSMART-g.
Additionally, we characterize SMART as a Riemannian gradient descent scheme on the parameter manifold induced by the Fisher-Rao geometry. This characterization allows us to incorporate a retraction, as specified in equation \eqref{eq:Exp-box}, into a geometric line search strategy. We refer to this variant of SMART as Riemannian gradient (RG). We explore three well-known line search methods within our geometric setup: Armijo, Hager-Zhang, and Barzilai-Borwein.
Furthermore, we include a Riemannian conjugate gradient  from \cite{Oviedo:2022} in our assessment. 
We consider relative-entropy regression over the box in three diverse applications: image deblurring, discrete tomography, and sparse recovery with expanders.  For each application, we compare the algorithms examined in this paper and demonstrate significant performance improvements over the basic version of SMART.
For detailed results and discussions, please refer to Figures \ref{fig:ExperimentToy} --
\ref{fig:ExperimentBlurAMVOGamma}.
Across all three applications we generate problem data uniformly. We start with either a binary signal $\hat x\in\{0,1\}^n$ or a signal $\hat x\in [0,1]^n$ that undergoes \emph{undersampling} by $A\in\R^{m\times n}_+$, where $m\leq n$. This process results in $b=A\hat x +e$ with $e\in\R^m$ representing added noise. Thus, problem \eqref{eq:def-KL-objective} is applicable to both the inconsistent (or underdetermined) system $Ax\approx b$. Furthermore, we are interested in analyzing the algorithms' behavior when the solution resides at the boundary. Therefore, observing both scenarios is crucial for our analysis.

\subsection{Implementation Details}\label{sec:implDetails}
We implemented the algorithms listed in Table~\ref{tab:list-of-algorithms} for solving ~\eqref{eq:def-KL-objective} numerically. The Bregman kernel is $\phi_\square$ from \eqref{eq:phi-list}. We chose the uninformative barycenter $\tfrac{\eins}{2}$ as the natural initialization for all algorithms. For each algorithm the same set of parameters was used across all experiment problem instances. The algorithms and corresponding parameter choices are listed below:

\begin{description}
\item[SMART] solely performs the multiplicative update specified in \eqref{eq:SMART-iteration} with its stepsize fixed to $\tau_k=\frac{1}{L}$, where again $f$ is L-smooth relative to $\phi_\square$. 

\item[FSMART] based on the iteration suggested in \cite{Petra2013a}, where initially $\theta_0=1$ is chosen, which is then subsequently updated via $\theta_{k+1}=\frac{\sqrt{\theta_k^4+4\theta_k^2}-\theta^2_k}{2}$, as suggested in \cite{Tseng:2008}.
\item[FSMART-e]  as described in \cite[Algorithm 2]{Hanzely:2021vc} was applied to ~\eqref{eq:def-KL-objective} with parameters $\gamma_{\text{min}}=1$, $\gamma_0=5$ and $\delta=0.05$. The choices for $\gamma_0$ and $\delta$ deviate slightly from the recommendation in \cite{Hanzely:2021vc},  but were chosen to facilitate fastest possible convergence on the selected problem instances. To ensure comparability restarting mechanisms and stopping criteria based on the divergence of iterates were foregone. Updates for $\theta$ were conducted via Newton's method.
\item[FSMART-g] specified in \cite[Algorithm 3]{Hanzely:2021vc} is used with parameters: $\rho=1.2$, $\gamma=2$ and $G_{\text{min}}=10^{-3}$ . Restarting, stopping criteria and updating $\theta$ was handled analogously to ABPG-e.
\item[RG-Armijo] a SMART iteration with Armijo line search for choosing the step size $\tau_k$ via the retraction in \eqref{eq:Exp-box} to iterate according to \eqref{eq:SMART_vs_RG}, see Algorithm~\ref{alg:RG_Armijo}. The line search parameters are $\sigma = 10^{-3}$, $\beta = 0.8$ and initial stepsize $\tau = 0.2$. 
\item[RG-HZ]  a SMART iteration with Hager-Zhang line search which includes the Armijo line search condition implemented with the same parameters as stated in \textbf{RG-Armijo} for comparision. Additional parameters are $\sigma_{2}=10^{-3}$ and $\varrho=0.5$.

\item[RG-BB] is a SMART iteration with Barzilai-Borwein line search which includes the Armijo line search condition implemented with the same parameters as stated in \textbf{RG-Armijo} for comparision. Additional parameters are $\gamma_{\min}=10^{-7}$ and $\gamma_{\max}=1.0$. For details see
Algorithm~\ref{alg:RG_BB}.

\item[CG] is the Riemannian conjugate gradient, Algorithm \ref{alg:CG}. It includes the Armijo line search condition implemented with the same parameters as stated in \textbf{RG-Armijo} for comparision.

\item[PG] is the projected gradient algorithm with Armijo linesearch.

\end{description}

\subsection{Toy Example}
We first conduct a toy experiment with
\begin{equation}\label{eq:toyExample}
    \hat x= \begin{pmatrix} 1 & 1 \end{pmatrix}^\top, \qquad A= \begin{pmatrix} 0.25 & 0.75 \end{pmatrix}, \qquad b = 1.
\end{equation}
    In this specific underdetermined scenario, signal $\hat x$, can be uniquely reconstructed from a single measurement $b$ when confined to the box.
    We demonstrate the smooth trajectory of iterates in contrast to the behavior observed in projected gradient descent in Figure~\ref{fig:ExperimentToy} (left).

\begin{figure}[htpb]
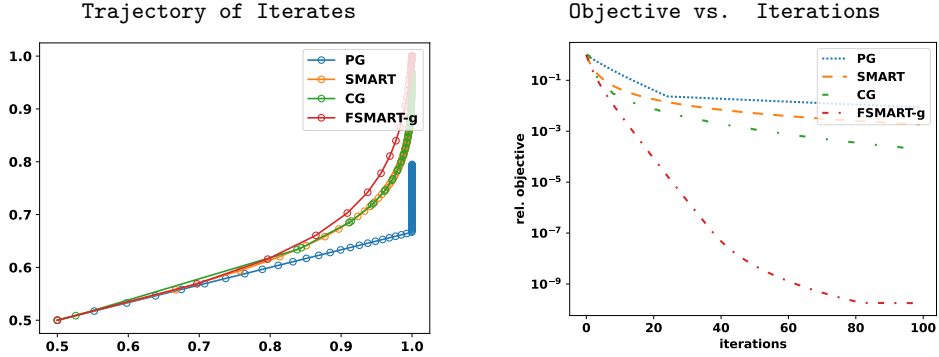

    \begin{center}
        \showExperimentToy{0.35}
    \end{center}
    \caption{\textbf{Toy example:} Trajectories of iterates, starting at $x_0=(0.5, 0.5)^{\T}$ and converging to the unique solution $\hat x=(1,1)^{\top}$, are displayed in the \textbf{left} panel. The projected gradient method (PG) shows a kink due to its nonsmooth projection step. Methods aware of the constraints geometry depict smoother, more optimal iterate trajectories by exploiting smooth Riemannian geometry to accommodate box constraints. In the \textbf{right} panel, FSMART-g stands out with larger initial steps and a rapid decrease in the relative objective value.
    }
    \label{fig:ExperimentToy}
\end{figure}

\subsection{Expander Graphs}
We consider binary measurement matrices $A\in\{0,1\}^{m\times n}$ with varying row numbers, $m \in \{40, 70, 100\}$, and a fixed number of columns $n=200$ with $\|A\|_1=12$. These matrices correspond to adjacency matrices of expander graphs  as detailed in \cite{Indyk_expander}. The sparsity of $\hat x\in \{0,1\}^{200}$ is set to $20$. With high probability, all algorithms should converge to the same solution, given the chosen sparsity and structure of $A$, ensuring likely uniqueness.

First, we compare different CG variants that differ in the $\beta_k$ parameters choice in \eqref{eq:beta-CG}. Results are depicted in Figure \ref{fig:ExperimentExpanderCG}. Additionally, Figure \ref{fig:ExperimentExpander} illustrates a comparison of various methods described in Section \ref{sec:implDetails}, where the DY mode \eqref{eq:beta-Dai-Yuan} is employed for the conjugate gradient (CG) approach. The original signal $\hat x$ and the resulting reconstructions by different algorithms are presented in Figure \ref{fig:ExperimentExpanderSpikes}. Finally, Figure \ref{fig:ExperimentExpanderAMVOGamma} displays a-posteriori certificates obtained from FSMART-e, illustrating $\gamma_k$ observations across iterations in the left panel, while visualizing average matrix vector operations on the right panel.

\begin{figure}[htpb]
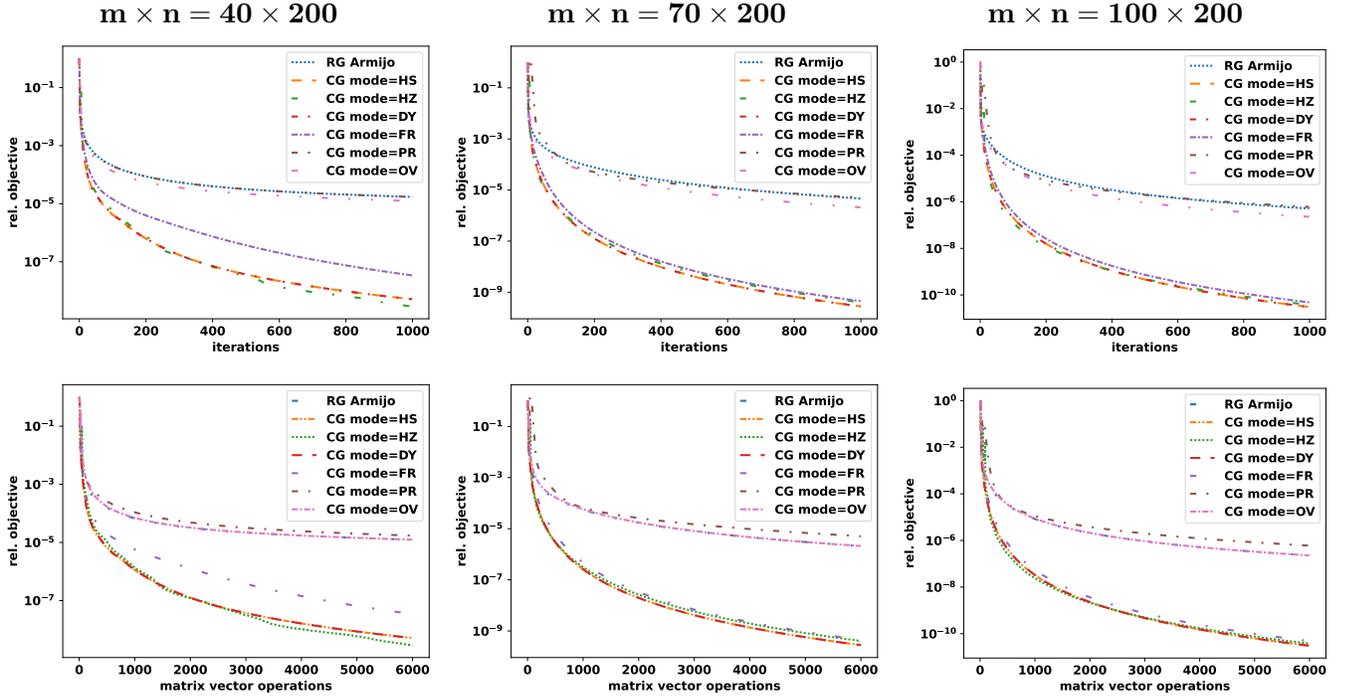

    \begin{center}
        \showExperimentExpanderCG{0.35}
    \end{center}
    \caption{
    \textbf{Comparison of CG variants on expander graphs} across three instances with varying levels of ill-posedness ($m \in \{40, 70, 100\}$ from left to right). The plots illustrate the relative decrease in objective values over iterations (\textbf{first row}) and costly operations (\textbf{second row}) for different approaches detailed in \eqref{eq:beta-CG} used to select $\beta_k$ for updating the search direction in the conjugate gradient approach. Additionally, Riemannian gradient descent with Armijo line search is considered as baseline. We observe a rough clustering of the variants into two groups: $\{$OV, PR$\}$  perform similarly to the Armijo baseline, while $\{$DY, FR, HS, HZ$\}$ exhibit notably faster convergence.
    }
    \label{fig:ExperimentExpanderCG}
\end{figure}

\begin{figure}[htpb]
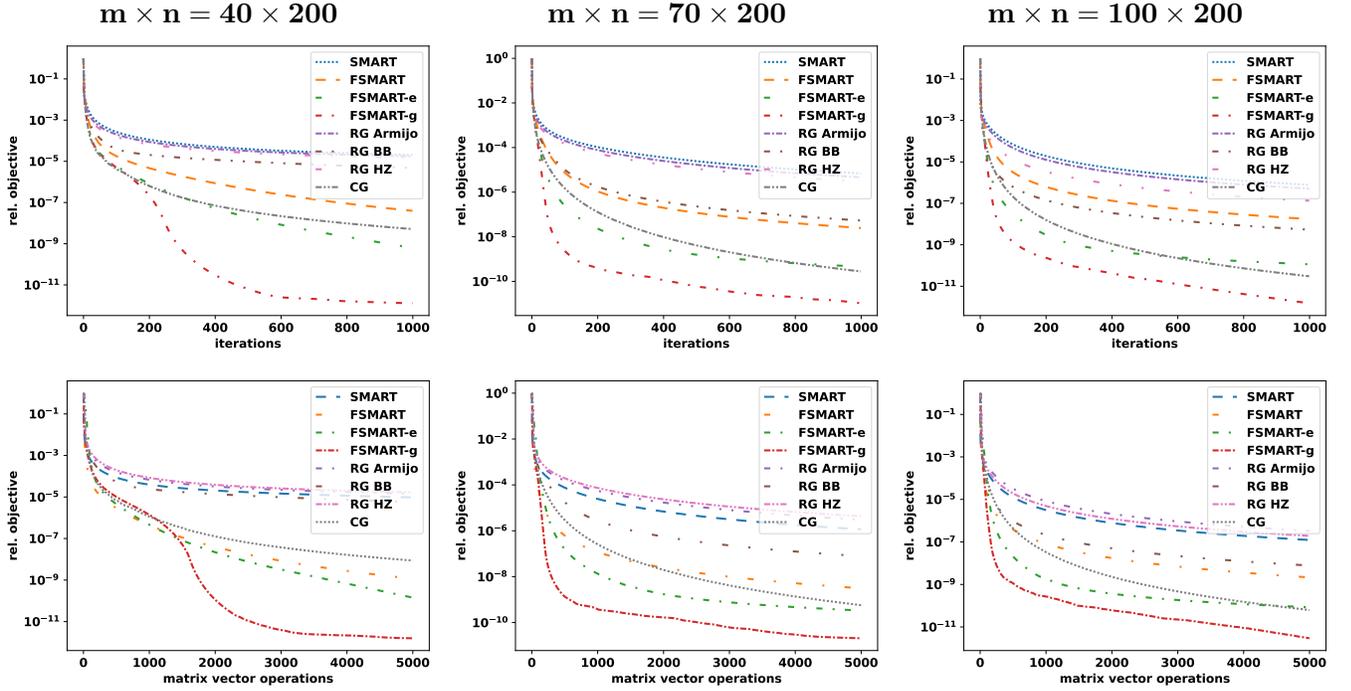

    \begin{center}
        \showExperimentExpander{0.35}
    \end{center}
    \caption{\textbf{Comparison of different algorithms on expander graphs}  across instances with varying levels of ill-posedness ($m \in \{40, 70, 100\}$ from left to right). The plots illustrate the relative decrease in objective values over iterations (\textbf{first row}) and in terms of costly operations (\textbf{second row}) for the different approaches detailed in Section \ref{sec:implDetails}. Among the accelerated algorithms within the Bregman regime, FSMART-e and particularly FSMART-g demonstrate superior performance. Within the Riemannian optimization regime, the conjugate gradient (CG) implemented with the DY mode for $\beta_k$-update rule 
    \eqref{eq:beta-Dai-Yuan} achieves results comparable to FSMART-e.
    }
    \label{fig:ExperimentExpander}
\end{figure}

\begin{figure}[htpb]
    \begin{center}
        \showExperimentExpanderSpikes{0.25}
    \end{center}
    \caption{ 
    \textbf{Sparse spike reconstructions} from adjacency matrices corresponding to expander graphs, generated by different algorithms, are displayed for $m=40$ after $1000$ iterations. The original binary signal $\hat x$ with sparsity $20$ is presented in the top left. Faster algorithms (see Figure \ref{fig:ExperimentExpander}) demonstrate perfect signal reconstruction. 
    However, it is worth noting that reconstructions from slower algorithms can be thresholded entrywise at $0.5$, transforming them into perfect reconstructions.
    }
    \label{fig:ExperimentExpanderSpikes}
\end{figure}

\begin{figure}[htpb]
    \begin{center}
        \showExperimentAMVOGamma{0.35}{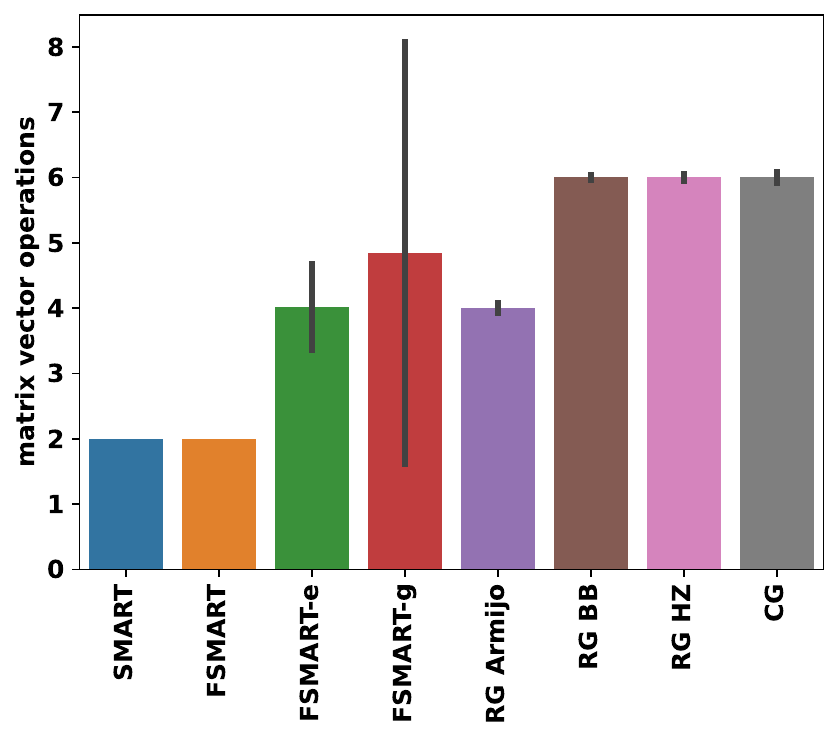}
    \end{center}
    \caption{\textbf{A-posteriori certificates for spike recovery} are obtained from FSMART-e by observing $\gamma_k$-values over iterations shown in the (\textbf{left}) panel for all problem instances. We observe that $\gamma_k$ drops to 1 in all instances. We explored the drop-down-point for each instance and observed that it occurs when FSMART-e approaches the solution. 
    \textbf{Average of matrix vector operations} are shown in the (\textbf{right}) panel for each algorithm over iterates and instances. By definition SMART, FSMART constantly employ two matrix vector operations per iteration. The other algorithms require potentially more operations as they employ line search strategies.
    }
    \label{fig:ExperimentExpanderAMVOGamma}
\end{figure}

\subsection{Tomographic Reconstruction}
We consider large scale tomographic reconstruction as a problem class, where we reconstruct the three phantoms shown in Fig.~\ref{fig:ExperimentTomoImages}. Tomographic projection matrices $A$ were generated using the ASTRA-toolbox\footnote{https://www.astra-toolbox.com}, with  parallel beam geometry and equidistant angles in the range $[0, \pi]$. Each entry in $A$ is nonnegative as it corresponds to the length of the intersection of a ray with a pixel. The undersampling rate was chosen to be $2\%$ which corresponds to $20$ projection angles. None of the images in Fig.~\ref{fig:ExperimentTomoImages} is expected to be the unique solution to $Ax=b$ within the $[0,1]^n$ box due to the large undersampling ratio \cite{Petra2014}. Hence, different algorithms might converge to different solutions within  $[0,1]^n$. 
Figure \ref{fig:ExperimentTomoAMVOGamma} shows the a-posteriori certificates obtained from FSMART-e by observing $\gamma_k$ over the iterations in the left panel, while the right panel visualizes the average matrix vector operations.

\begin{figure}[htpb]
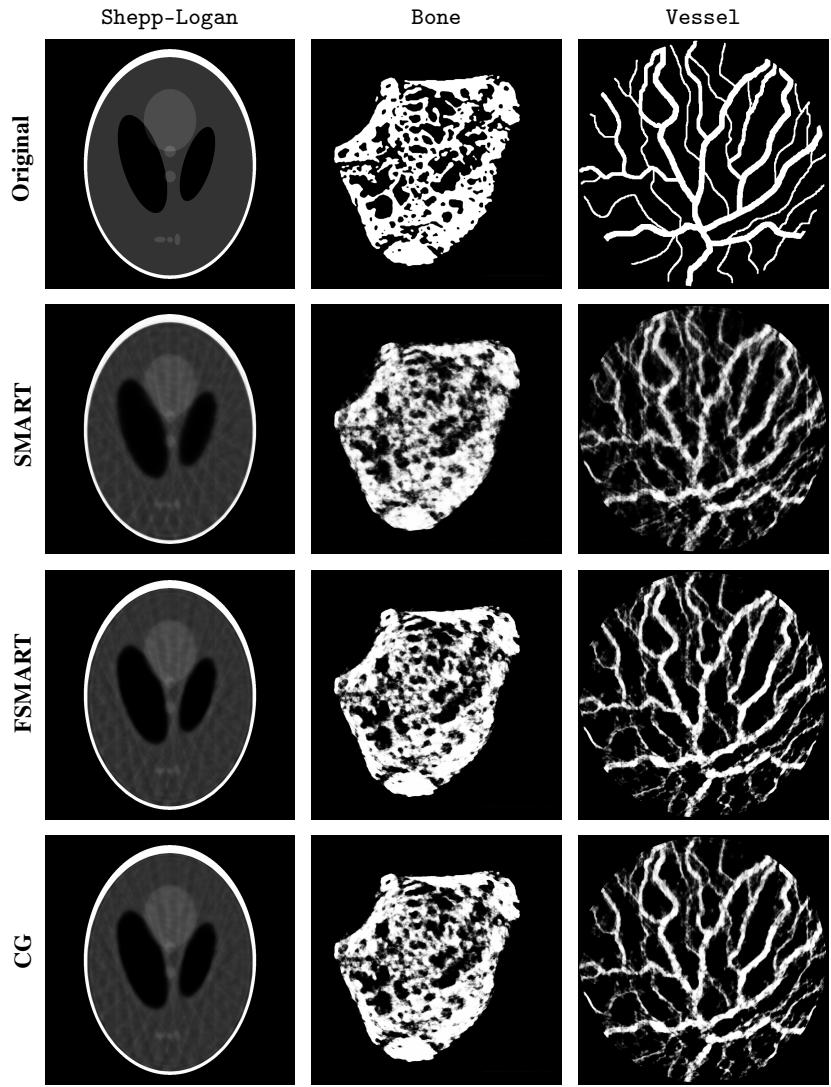

\begin{center}
    \showTomoImages{0.2}
\end{center}
  \caption{
  \textbf{Phantoms and reconstructions:} The first row shows large-scale phantoms of size $1024\times 1024$ used for numerical evaluation. Subsequent rows illustrate the resulting reconstructions from $2\%$ undersampling  from the algorithms SMART, FSMART, and CG (with DY mode), after 400 iterations. We refer to Figure \ref{fig:ExperimentTomo} for the plot depicting the relative decrease of the objective and remark
  that reconstructions from faster algorithms (last 2 rows) resemble the original more closely.
  }
  \label{fig:ExperimentTomoImages}
\end{figure}

\begin{figure}[htpb]
    \begin{center}
        \showExperimentTomo{0.35}
    \end{center}
    \caption{\textbf{Comparison of different algorithms on tomographic reconstruction} of phantoms shown in Figure \ref{fig:ExperimentTomoImages}.
    The plots show the relative decrease in objective values over iterations (\textbf{first row}) and in terms of costly operations (\textbf{second row}) for the different approaches detailed in Section \ref{sec:implDetails}. In these instances, the conjugate gradient method with DY mode \eqref{eq:beta-Dai-Yuan} from the Riemannian optimization framework exhibits the best performance, followed by FSMART-g from the Bregman framework.
    }
    \label{fig:ExperimentTomo}
\end{figure}

\begin{figure}[htpb]
    \begin{center}
        \showExperimentAMVOGamma{0.35}{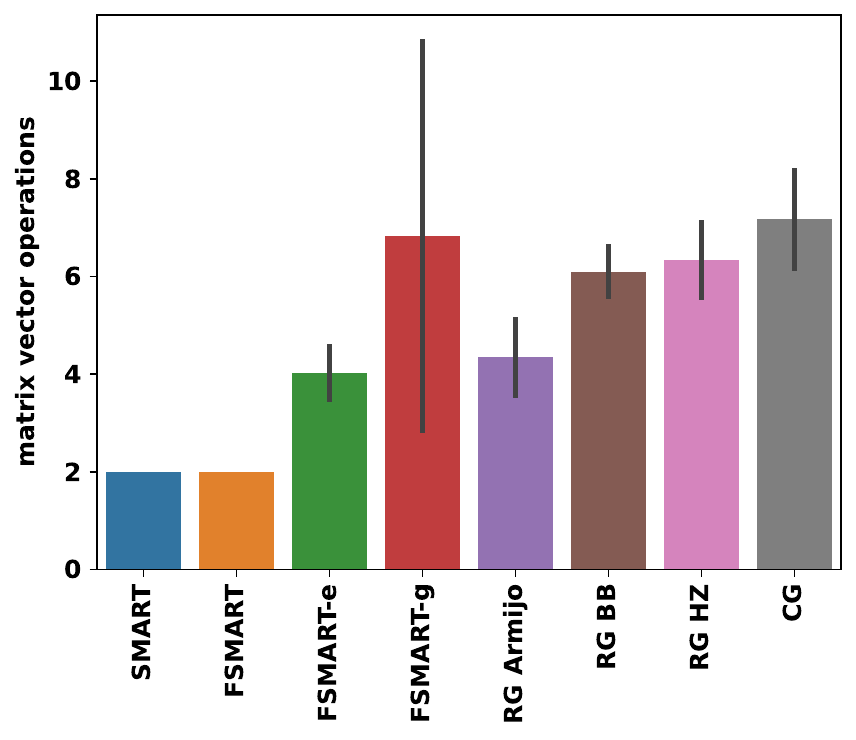}
    \end{center}
    \caption{
    \textbf{A-posteriori certificates for tomographic reconstruction} are obtained from FSMART-e by tracking $\gamma_k$ values across iterations, depicted in the (\textbf{left}) panel for all problem instances. It is notable that $\gamma_k$ consistently drops to 1 in all instances. We investigated this drop-point for each instance and found it aligns with FSMART-e approaching the solution. The (\textbf{right}) panel illustrates the \textbf{average number of matrix-vector operations} for each algorithm across iterations and instances. SMART and FSMART constantly employ two matrix-vector operations per iteration by definition. However, other algorithms require more operations due to their utilization of line search or triangle gain adaption strategies.
    }
    \label{fig:ExperimentTomoAMVOGamma}
\end{figure}

\subsection{Deblurring}
We evaluated the algorithms listed in Table \ref{tab:list-of-algorithms} for the deblurring task using the images displayed in the first row of Figure \ref{fig:ExperimentBlurImages}. The first image, 'Kitten,' has pixel values within the range $[0,1]$, while 'Tiger' and 'QR-Code' admit binary pixel values. We degraded these images by severe Gaussian blur with a $33 \times 33$ mask and $\sigma=10$, as illustrated in the second row of Figure \ref{fig:ExperimentBlurImages}. Subsequent rows display the resulting reconstructions after 1000 iterations using the SMART, FSMART, and CG (with DY mode) algorithms. We refer to Figure \ref{fig:ExperimentBlur} for the plot illustrating the relative decrease in the objective. Finally, Figure \ref{fig:ExperimentBlurAMVOGamma} shows a-posteriori certificates obtained from FSMART-e, depicting observations of $\gamma_k$ over the iterations in the left panel, and average matrix vector operations in the right panel.

\begin{figure}[htpb]
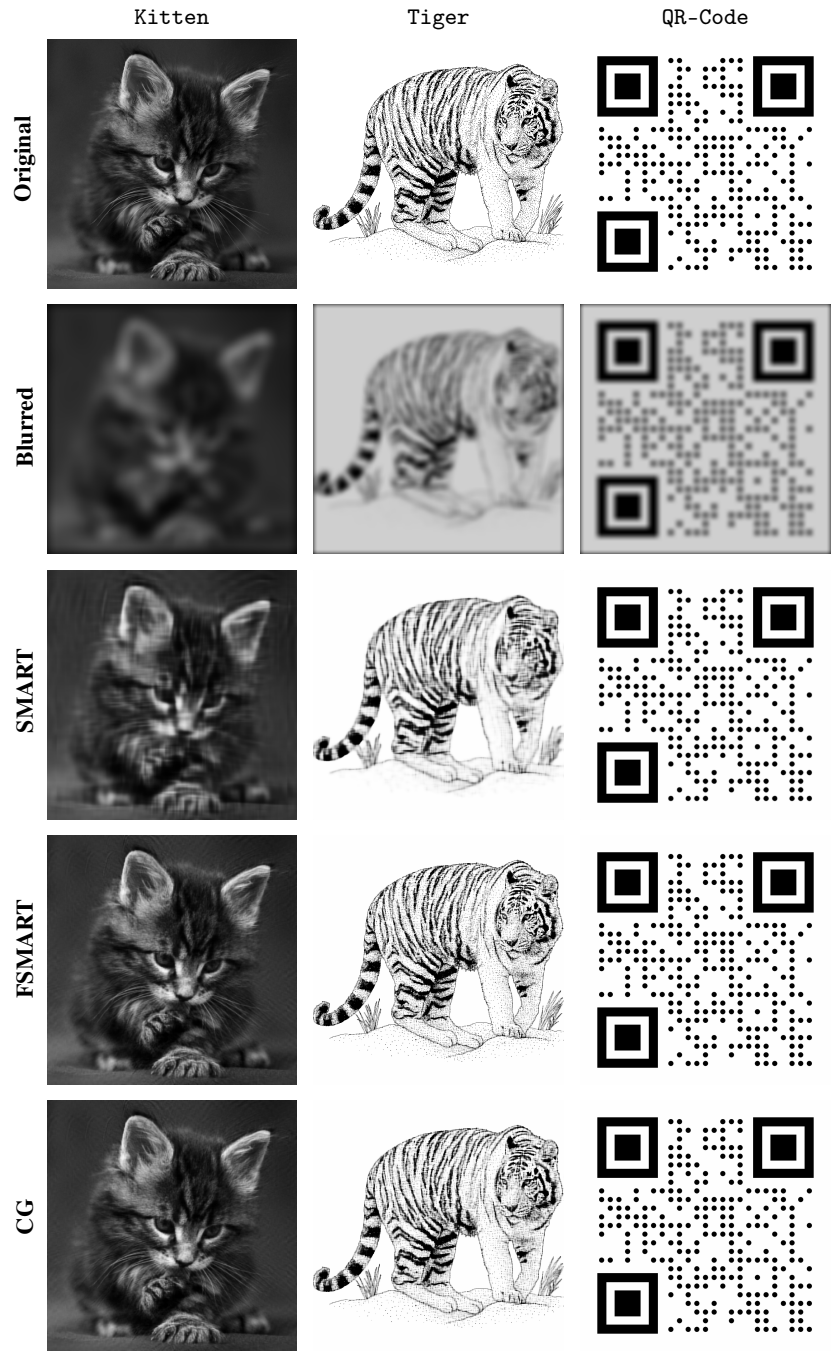

\begin{center}
    \showBlurImages{0.2}
\end{center}
  \caption{
  \textbf{Original images and deblurring:} These plots display the images used for numerical evaluation in the first row and their degraded versions due to severe Gaussian blur with a mask of size $33 \times 33$ and $\sigma=10$ in the second row. The rows below, ahow the resulting reconstructions after 1000 iterations from the algorithms SMART, FSMART, and CG (with DY mode).
  }
  \label{fig:ExperimentBlurImages}
\end{figure}

\begin{figure}[htpb]
    \begin{center}
        \showExperimentBlur{0.35}
    \end{center}
    \caption{\textbf{Comparison of different algorithms on deblurring} of the instances shown in Figure \ref{fig:ExperimentBlurImages}. 
    The plots illustrate the relative decrease in objective values over iterations (\textbf{first row}) and as a function of costly operations (\textbf{second row}) for the various approaches detailed in Section \ref{sec:implDetails}. Among the accelerated algorithms within the Bregman framework, FSMART, FSMART-e, and FSMART-g perform the best. Within the Riemannian optimization group, the conjugate gradient (CG) implemented with DY mode as the $\beta_k$-update rule \eqref{eq:beta-Dai-Yuan} the best results. 
    }
    \label{fig:ExperimentBlur}
\end{figure}

\begin{figure}[htpb]
    \begin{center}
        \showExperimentAMVOGamma{0.35}{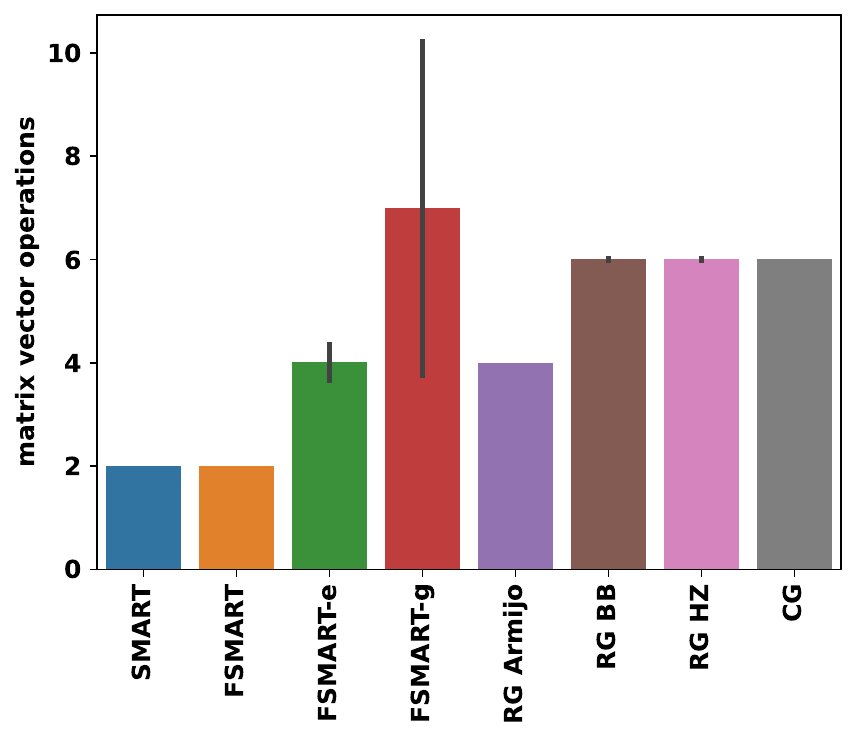}
    \end{center}
    \caption{
    \textbf{A-posteriori certificates for deblurring} are obtained from FSMART-e by observing $\gamma_k$ values over iterations, shown in the (\textbf{left}) panel for all problem instances. It is observed that $\gamma_k$ does not always drop to 1 in all instances, especially in the two more challenging cases (kitten, tiger). In these instances, the algorithms are still far from the sought solutions even after 10000 iterations.
    \textbf{Average of matrix vector operations} are shown in the (\textbf{right}) panel for each algorithm over iterates and instances. By definition SMART, FSMART constantly employ two matrix vector operations per iteration. The other algorithms require more operations as they employ line search or triangle gain adaption strategies.}
    \label{fig:ExperimentBlurAMVOGamma}
\end{figure}

\subsection{Discussion}\label{sec:Discussion}
We discuss our observations according to the following aspects:
\begin{enumerate}[(1)]
\item Observations about acceleration in the Bregman framework;
\item Observations about line search strategies;
\item Bregman vs.~Riemannian framework: empirical findings;
\end{enumerate}

\begin{enumerate}[(1)]
\item

\begin{enumerate}[(a)]
\item
\textit{Accelerated Bregman proximal gradient: convergence rate.} In all experiments, the ABPG methods (FSMART, FSMART-e, FSMART-g), especially the adaptive variants (FSMART-e, FSMART-g), demonstrate superior performance compared with the SMART method. 
Moreover, we obtain numerical certificates for the empirical $O(1/k^\gamma)$ rate, where $\gamma \in (1,2)$ in all our experiments. Specifically, the adaptive variants can automatically search for the largest possible $\gamma_k$ for which the convergence rate holds for finite $k$, even though $\gamma=1$ in the TSE \eqref{eq:TSE} for the $\KL$ based divergences considered in this paper. Notably, $\gamma_k\ge 2$ when far from the solution and consistently $\gamma_k=1$ as we approach the solution. Consequently, this causes the convergence rate to slow down as we near the solution.
\item
\textit{Accelerated Bregman proximal gradient: performance.}
The comparison involves three algorithms across three problem instances: In the case of expander graphs, FSMART-g is superior to FSMART and FSMART-e. However, for tomography, FSMART and FSMART-g perform comparably and outperform FSMART-g. A similar pattern emerges in image deblurring, except for the QR-code instance, where FSMART-e notably leads over FSMART-g and FSMART.
\end{enumerate}

\item
\begin{enumerate}[(a)]
\item
\textit{Line search strategies: overall performance.} 
As the SMART iteration represents a Riemannian gradient (RG) step with a fixed steplength, it serves as our baseline for comparison against RG with Armijo, Hager-Zhang (HZ), or Barzilai-Borwein (BB) steplength selection. When not considering computational cost, BB line search outperforms HZ, with HZ slightly surpassing Armijo. Meanwhile, Armijo itself surpasses SMART (without line search) on expander graphs. A similar trend is observed in the tomography examples, with HZ displaying better performance than Armijo on the non-binary image. However, in deblurring, Armijo and HZ perform comparably but are notably surpassed by BB. Interestingly, SMART outperforms the line search methods in this scenario.
It is worth noting that certain sets of line search parameters could reverse this scenario. Nevertheless, our aim was to avoid tuning line search parameters individually for each problem instance.

\item
\textit{Line search strategies: computational costs.} 
Due to the specific line search parameters chosen, the average number of matrix-vector operations per iteration ranges between 4 and 7 for all RG methods using Armijo, HZ, or BB. Among these, Armijo consistently demonstrates the highest computational efficiency and robustness across all problem instances. 
This contradicts the empirical findings in \cite{Sutti:2020vn}, where HZ was performing the best, albeit on a different type of manifold.
\end{enumerate}

\item
\begin{enumerate}[(a)]
\item
\textit{Bregman vs.~Riemannian framework: convergence.} 
Only SMART and Riemannian Gradient with Armijo line search guarantee a monotonically decreasing objective function. A non-monotonic decrease is observed for RG with HZ and BB line searches, particularly noticeable in the tomographic instances. There's no significant difference in convergence rates when the solution lies at the boundary. These trends are consistent across algorithms from both the Bregman and Riemannian gradient frameworks.

\item
\textit{Bregman vs.~Riemannian framework: best performance is achieved by 
accelerated BPG and by CG.} 
In tomographic recovery, the conjugate gradient method with DY mode \eqref{eq:beta-Dai-Yuan} from the Riemannian optimization framework demonstrates superior performance, followed by FSMART-g from the Bregman framework. Moreover, in deblurring and spike recovery tasks, the accelerated Bregman methods alongside the Riemannian CG method notably outperform the Riemannian gradient methods.

\end{enumerate}
\end{enumerate}

\section{Conclusion}
\label{sec:conclusion}
 We studied recent acceleration techniques derived from Bregman proximal gradient methods for SMART and conducted numerical a-posteriori certification for several large-scale problems. Although we could not certify the $\mc{O}(1/k^2)$ rate, the accelerated variants of SMART proved to be remarkably efficient.

Moreover, we characterized SMART as a Riemannian gradient descent scheme on parameter manifolds induced by the Fisher-Rao geometry. This approach utilizes a retraction based on e-geodesics. We computed closed-form expressions for retractions and the corresponding vector transports for three manifolds as case studies. Additionally, we explored various Riemannian line search strategies such as geometric Hager-Zhang-type or Barzilai-Borwein line search, and investigated acceleration using Riemannian conjugate gradient.

Conditions ensuring convergence in geometric settings relate to Lipschitz continuity, extending the well-known gradient Lipschitz condition from Euclidean settings. This geometric $L$-smoothness condition does not precisely align with the scenarios studied in our paper, however. On the other hand, in Section \ref{BPG-section}, we showed how convergence of SMART, considered as Riemannian gradient descent iteration with a fixed step size, can be proven under relative $L$-smoothness. This result might pave the way for proving convergence in the geometric context even when geometric $L$-smoothness conditions are not met.

Our further work will explore possibilities for connecting the local triangle scaling property used by accelerated Bregman proximal gradient, to certify convergence rates with line search methods based on suitable retractions.

\vspace{2mm}
\noindent\textbf{Acknowledgement.}
This work is funded by Deutsche Forschungsgemeinschaft (DFG) under Germany's Excellence Strategy EXC-2181/1 - 390900948 (the Heidelberg STRUCTURES Excellence Cluster).
\newcommand{\etalchar}[1]{$^{#1}$}
\providecommand{\bysame}{\leavevmode\hbox to3em{\hrulefill}\thinspace}
\providecommand{\MR}{\relax\ifhmode\unskip\space\fi MR }
\providecommand{\MRhref}[2]{%
  \href{http://www.ams.org/mathscinet-getitem?mr=#1}{#2}
}
\providecommand{\href}[2]{#2}


\begin{thebibliography}{HLWY19}

\bibitem[ABB04]{Alvarez:2004}
F.~Alvarez, J.~Bolte, and O.~Brahic, \emph{{Hessian Riemannian Gradient Flows in Convex Programming}}, SIAM J. Control Optim. \textbf{43} (2004), no.~2, 477--501.

\bibitem[AC10]{InformationGeometry-2010}
S.-I. Amari and A.~Cichocki, \emph{{Information Geometry of Divergence Functions}}, Bull. Polish Acad. Sci \textbf{58} (2010), no.~1, 183--195.

\bibitem[AMS08]{Absil:2008aa}
P.~A. Absil, R.~Mahony, and R.~Sepulchre, \emph{{Optimization Algorithms on Matrix Manifolds}}, Princeton University Press, 2008.

\bibitem[AN00]{Amari:2000}
S.-I. Amari and H.~Nagaoka, \emph{{Methods of Information Geometry}}, Amer. Math. Soc. and Oxford Univ. Press, 2000.

\bibitem[AT06]{Auslender-IGA-2006}
A.~Auslender and M.~Teboulle, \emph{{Interior Gradient and Proximal Methods for Convex and Conic Optimization}}, SIAM Journal on Optimization \textbf{16} (2006), no.~3, 697--725.

\bibitem[BB88]{BarzilaiBorwein:88}
J.~Barzilai and J.~M. Borwein, \emph{{Two-Point Step Size Gradient Methods}}, IMA Journal of Numerical Analysis \textbf{8} (1988), 141--148.

\bibitem[BB97]{Bauschke:1997aa}
H.~G. Bauschke and J.~M. Borwein, \emph{{Legendre Functions and the Method of Random Bregman Projections}}, J. Convex Anal. \textbf{4} (1997), 27--67.

\bibitem[BBDV09]{Bertero_2009}
M.~Bertero, P.~Boccacci, G.~Desider\`{a}, and G.~Vicidomini, \emph{{Image Deblurring with Poisson Data: from Cells to Galaxies}}, Inverse Problems \textbf{25} (2009), no.~12, 123006.

\bibitem[BBT17]{Bauschke:2017aa}
H.~H. Bauschke, J.~Bolte, and M.~Teboulle, \emph{{A Descent Lemma Beyond Lipschitz Gradient Continuity: First-Order Methods Revisited and Applications}}, Mathematics of Operations Research \textbf{42} (2017), no.~2, 330--348.

\bibitem[BI08]{Indyk_expander}
R.~Berinde and P.~Indyk, \emph{{Sparse Recovery Using Sparse Random Matrices}}, 2008, Technical Report, MIT.

\bibitem[Bro86]{Brown:1986vy}
L.~D. Brown, \emph{{Fundamentals of Statistical Exponential Families}}, Institute of Mathematical Statistics, Hayward, CA, 1986.

\bibitem[BT03]{Beck:2003aa}
A.~Beck and M.~Teboulle, \emph{{Mirror Descent and Nonlinear Projected Subgradient Methods for Convex Optimization}}, Operations Research Letters \textbf{31} (2003), no.~3, 167--175.

\bibitem[Byr93]{Byrne:1993aa}
C.~L. Byrne, \emph{{Iterative Image Reconstruction Algorithms based on Cross-Entropy Minimization}}, IEEE Trans. Image Process. \textbf{2} (1993), no.~1, 96--103.

\bibitem[Byr98]{Byrne_bounded_SMART}
\bysame, \emph{{Iterative Algorithms for Deblurring and Deconvolution with Constraints}}, Inverse Problems \textbf{14} (1998), no.~6, 1455.

\bibitem[Byr14]{Byrne2014}
\bysame, \emph{{Iterative Optimization in Inverse Problems}}, CRC Press, 2014.

\bibitem[Cen81]{Censor1981}
Y.~Censor, \emph{{Row-Action Methods for Huge and Sparse Systems and their Applications}}, SIAM Review \textbf{23} (1981), no.~4, 444--466.

\bibitem[Cen97]{Censor:1997aa}
\bysame, \emph{{Parallel Optimization: Theory, Algorithms and Applications}}, Oxford University Press, 1997.

\bibitem[CL81]{Censor:1981vy}
Y.~Censor and A.~Lent, \emph{{An Iterative Row-Action Method for Interval Convex Programming}}, J. Optimiz. Theory Appl. \textbf{34} (1981), no.~3, 321--353.

\bibitem[CS87]{CensorSegman87}
Y.~Censor and J.~Segman, \emph{On block-iterative entropy maximization}, J. Inform. Optim. Sci. \textbf{8} (1987), no.~3, 275--291.

\bibitem[Csi75]{Csiszar:75}
I.~Csiszar, \emph{{$I$-Divergence Geometry of Probability Distributions and Minimization Problems}}, The Annals of Probability \textbf{3} (1975), no.~1, 146--158.

\bibitem[CT93]{Chen_Teboulle_1993}
G.~Chen and M.~Teboulle, \emph{{C}onvergence analysis of a proximal-like minimization algorithm using {B}regman functions}, SIAM J. Optim. \textbf{3} (1993), no.~3, 538--543.

\bibitem[CT06]{Cover:2006aa}
T.~M Cover and J.~A. Thomas, \emph{{Elements of Information Theory {(2.} ed.)}}, Wiley, 2006.

\bibitem[CZ92]{Censor1992}
Y.~Censor and S.A. Zenios, \emph{{Proximal Minimization Algorithm with D-functions}}, J. Optim. Theory Appl \textbf{73} (1992), no.~3, 451--464.

\bibitem[DR72]{DarrochRatcliff72}
J.~N. Darroch and D.~Ratcliff, \emph{{Generalized Iterative Scaling for Log-Linear Models}}, The Annals of Mathematical Statistics \textbf{43} (1972), no.~5, 1470 -- 1480.

\bibitem[DTdB22]{Dragomir:2022aa}
R.-A. Dragomir, A.~B. Taylor, A.~d'Aspremont, and J.~Bolte, \emph{{Optimal Complexity and Certification of Bregman First-Order Methods}}, Math. Program. \textbf{194} (2022), 41--83.

\bibitem[FLP19]{Ferreira:2019}
O.~P. Ferreira, M.~S. Louzeiro, and L.~F. Prudente, \emph{{Gradient Method for Optimization on Riemannian Manifolds with Lower Bounded Curvature}}, SIAM Journal on Optimization \textbf{29} (2019), no.~4, 2517--2541.

\bibitem[GHS20]{Ghai:2020va}
U.~Ghai, E.~Hazan, and Y.~Singer, \emph{{Exponentiated Gradient vs. Meets Gradient Descent}}, Proc. Mach. Learning Res. \textbf{117} (2020), 1--23.

\bibitem[GLL86]{Grippo1986}
L.~Grippo, F.~Lampariello, and S.~Lucidi, \emph{{A Nonmonotone Line Search Technique for Newton's Method}}, SIAM J. Numer. Anal. \textbf{23} (1986), no.~4, 707--716.

\bibitem[GPn23]{Gutman:2022tu}
D.~H. Gutman and J.~F. Pe\~{n}a, \emph{{Perturbed Fenchel Duality and First-Order Methods}}, Math. Program. \textbf{198} (2023), no.~1, 443--469.

\bibitem[HCGL85]{Herman1985a}
G.T. Herman, Y.~Censor, D.~Gordon, and R.M. Lewitt, \emph{{Comment on A Statistical Model for Positron Emission Tomography}}, J. Amer. Statist. Assoc \textbf{80} (1985), 22--25.

\bibitem[HLWY19]{Hu2019}
J.~Hu, X.~Liu, Z.~Wen, and Y.~Yuan, \emph{{A Brief Introduction to Manifold Optimization}}, 2019, arXiv:1906.05450.

\bibitem[HRX21]{Hanzely:2021vc}
F.~Hanzely, P.~Richt\'{a}rik, and L.~Xiao, \emph{{Accelerated Bregman Proximal Gradient Methods for Relatively Smooth Convex Optimization}}, Comput. Optim. Appl. \textbf{79} (2021), 405--440.

\bibitem[IP17]{Iannazzo2017}
B.~Iannazzo and M.~Porcelli, \emph{{The Riemannian Barzilai–Borwein Method with Nonmonotone Line Search and the Matrix Geometric Mean Computation}}, IMA Journal of Numerical Analysis \textbf{38} (2017), no.~1, 495--517.

\bibitem[JKM23]{Juditsky:2022wf}
A.~Juditsky, J.~Kwon, and \'{E}. Moulines, \emph{{Unifying Mirror Descent and Dual Averaging}}, Math. Programming \textbf{199} (2023), 793--830.

\bibitem[Jos17]{Jost:2017aa}
J.~Jost, \emph{{Riemannian Geometry and Geometric Analysis}}, 7th ed., Springer-Verlag Berlin Heidelberg, 2017.

\bibitem[JS22]{Jin:2022aa}
J.~Jin and S.~Sra, \emph{{Understanding Riemannian Acceleration via a Proximal Extragradient Framework}}, Proc. Mach. Learning Res. \textbf{178} (2022), 1--39.

\bibitem[KBB15]{Krichene:2015}
W.~Krichene, A.~Bayen, and P.~L. Bartlett, \emph{{Accelerated Mirror Descent in Continuous and Discrete Time}}, Advances in Neural Information Processing Systems (C.~Cortes, N.~Lawrence, D.~Lee, M.~Sugiyama, and R.~Garnett, eds.), vol.~28, Curran Associates, Inc., 2015.

\bibitem[KPS{\etalchar{+}}23]{Kahl:2023aa}
M.~M. Kahl, S.~Petra, C.~Schn{\"{o}}rr, G.~Steidl, and M.~Zisler, \emph{{On the Remarkable Efficiency of the SMART Iteration}}, {Scale Space and Variational Methods in Computer Vision (SSVM)} (L.~Calatroni, M.~Donatelli, S.~Morigi, M.~Prato, and M.~Santacesaria, eds.), LNCS, vol. 14009, Springer, 2023, pp.~418--430.

\bibitem[LC91]{Lent1991a}
A.~Lent and Y.~Censor, \emph{{The Primal-Dual Algorithm as a Constraint-Set-Manipulation Device}}, Math. Program. \textbf{50} (1991), no.~1--3, 343--357.

\bibitem[MPZ23]{Mueller2022multilevel}
S.~M\"uller, S.~Petra, and M.~Zisler, \emph{{Multi-level Geometric Optimization for Regularised Constrained Linear Inverse Problems}}, Pure and Applied Functional Analysis \textbf{8} (2023), no.~3, 855--880.

\bibitem[NY83]{Nemirovski:1983}
A.~Nemirovski and D.~Yudin, \emph{{Problem Complexity and Method Efficiency in Optimization}}, Wiley, 1983.

\bibitem[Ovi22]{Oviedo:2022}
H.~Oviedo, \emph{{Global Convergence of Riemannian Line Search Methods with a Zhang-Hager-type Condition}}, Numer Algor \textbf{91} (2022), 1183--1203.

\bibitem[PS14]{Petra2014}
S.~Petra and C.~Schn\"{o}rr, \emph{{A}verage {C}ase {R}ecovery {A}nalysis of {T}omographic {C}ompressive {S}ensing}, Linear Algebra and its Applications \textbf{441} (2014), 168--198, Special issue on Sparse Approximate Solution of Linear Systems.

\bibitem[PSBL13]{Petra2013a}
S.~Petra, C.~Schn\"orr, F.~Becker, and F.~Lenzen, \emph{{B-SMART: Bregman-Based First-Order Algorithms for Non-Negative Compressed Sensing Problems}}, Proc. SSVM, LNCS, vol. 7893, Springer, 2013, pp.~110--124.

\bibitem[Rau23]{Thesis_Maren}
M.~Raus, \emph{{}}, 2023, Master's thesis, Heidelberg University.

\bibitem[RM15]{Raskutti:2015aa}
G.~Raskutti and S.~Mukherjee, \emph{{The Information Geometry of Mirror Descent}}, IEEE Transactions on Information Theory \textbf{61} (2015), no.~3, 1451--1457.

\bibitem[Roc76]{Rockafellar1976}
R.~T. Rockafellar, \emph{{Monotone Operators and the Proximal Point Algorithm}}, SIAM J. Control Optim \textbf{14} (1976), no.~5, 877--898.

\bibitem[Roc97]{Rockafellar:1997}
\bysame, \emph{{Convex Analysis}}, Princeton paperbacks, {Princeton University Press}, Princeton, N.J, 1997.

\bibitem[RW10]{RockafellarWets2010}
R.~T. Rockafellar and R.~J.~B. Wets, \emph{{Variational Analysis {(3.} ed.)}}, Springer, 2010.

\bibitem[SSI23]{Sakai2022}
H.~Sakai, H.~Sato, and H.~Iiduka, \emph{{Global Convergence of Hager-Zhang type Riemannian Conjugate Gradient Method}}, 2023, p.~127685.

\bibitem[SV21]{Sutti:2020vn}
M.~Sutti and B.~Vandereycken, \emph{{Riemannian Multigrid Line Search for Low-Rank Problems}}, SIAM Journal on Scientific Computing \textbf{43} (2021), no.~3, A1803--A1831.

\bibitem[Tse08]{Tseng:2008}
P.~Tseng, \emph{{On Accelerated Proximal Gradient Methods for Convex-Concave Optimization}}, 2008, unpublished manuscript.

\bibitem[ZH04]{HagerZhang:2004}
H.~Zhang and W.~W. Hager, \emph{{A Nonmonotone Line Search Technique and its Application to Unconstrained Optimization}}, SIAM J. Optim. \textbf{14} (2004), no.~4, 1043--1056.

\bibitem[ZS16]{Zhang:2016}
H.~Zhang and S.~Sra, \emph{{First-order Methods for Geodesically Convex Optimization}}, 2016, arXiv:1602.06053.

\end{thebibliography}
\end{document}